\documentclass{amsart}

\usepackage[nonewpage]{imakeidx}

\usepackage
           [top=1in, bottom=1in, left=1in, right=1in]
           {geometry}

\usepackage{mathrsfs}
\usepackage{amsthm}
\usepackage{thmtools}
\usepackage{amsmath}
\usepackage{stmaryrd}
\usepackage[shortlabels]{enumitem}

\usepackage{amssymb,amsxtra,amsfonts}
\usepackage{graphicx}
\usepackage{tikz-cd}

\usepackage[normalem]{ulem}

\usepackage[unicode=true,colorlinks=true,linkcolor=blue,urlcolor=blue, citecolor=red, pagebackref=True]{hyperref}

\usepackage{amscd}
\usepackage{xcolor}
\usepackage{ifpdf} 
\usepackage[all]{xy}
\ifpdf
\else
\fi

\makeatletter
\@namedef{subjclassname@2020}{%
 \textup{2020} Mathematics Subject Classification}
\makeatother

\numberwithin{equation}{section}
\numberwithin{table}{section}

\declaretheorem[style=plain,parent=section]{theorem}
\declaretheorem[style=plain,sibling=theorem]{corollary}
\declaretheorem[style=plain,sibling=theorem]{lemma}

\declaretheorem[style=plain,sibling=theorem]{proposition}
\declaretheorem[style=plain,sibling=theorem]{conjecture}

\declaretheorem[style=plain]{question}

\declaretheorem[style=definition,sibling=theorem]{definition}

\declaretheorem[style=definition, qed=\hfill $\diamond$, sibling=definition]{example}
\declaretheorem[style=remark,sibling=theorem]{remark}

\theoremstyle{definition}
\newtheorem*{theorem*}{Theorem}
\newtheorem*{exampleth}{Proof}

\newtheoremstyle{named}{}{}{\itshape}{}{\bfseries}{.}{.5em}{\thmnote{#3 }#1}
\theoremstyle{named}
\newtheorem{namedconjecture}{Conjecture}

\DeclareMathOperator{\Trop}{Trop} 
\DeclareMathOperator{\ptrop}{Trop_{>0}}

\DeclareMathOperator{\Hom}{Hom}

\newcommand{\Q}{\mathbb{Q}}
\newcommand{\R}{\mathbb{R}}

\newcommand{\CC}{\mathbb{C}}
\newcommand{\Z}{\mathbb{Z}}
\newcommand{\N}{\mathbb{N}}
\newcommand{\D}{\mathbb{D}}
\newcommand{\bS}{\mathbb{S}}

\newcommand{\cC}{\mathcal{C}}
\newcommand{\cD}{\mathcal{D}}
\newcommand{\cF}{\mathcal{F}}

\newcommand{\cM}{\mathcal{M}}
\newcommand{\cN}{\mathcal{N}}
\newcommand{\cO}{\mathcal{O}}

\newcommand{\cS}{\mathcal{S}}

\newcommand \cY {\mathcal{Y}}

\newcommand{\tv}{\mathcal{X}}

\newcommand{\wtuv}[2]{\ensuremath{\ell_{#1,#2}}} 
\newcommand{\du}[1]{\ensuremath{d_{#1}}} 

\newcommand{\valv}[1]{\ensuremath{\delta_{#1}}} 

\newcommand{\roottree}{r}

\newcommand{\nodesT}[1]{\ensuremath{\mathring{V}({#1})}} 
\newcommand{\leavesT}[1]{\ensuremath{\partial\,#1}} 

\newcommand{\nodesTevRoot}[1]{\ensuremath{\partial_{#1}\Gamma}}

\newcommand{\starT}[2]{\ensuremath{\operatorname{Star}_{#1}(#2)}} 

\newcommand{\wtNve}[2]{\ensuremath{m_{#1,#2}}} 
\newcommand{\wtNveL}[3]{\ensuremath{m_{#1,#2,#3}}} 
\newcommand{\zu}{\ensuremath{\underline{z}}}

\newcommand{\zexp}[1]{\ensuremath{\zu^{#1}}} 

\newcommand{\fvi}[2]{\ensuremath{f_{#1,#2}}} 
\newcommand{\fvibar}[2]{\ensuremath{\overline{f}_{#1,#2}}} 

\newcommand{\gvi}[2]{\ensuremath{g_{#1,#2}}} 
\newcommand{\fgvi}[2]{\ensuremath{F_{#1,#2}}} 
\newcommand{\fgvibar}[2]{\ensuremath{\overline{F}_{#1,#2}}} 

\newcommand{\cvei}[3]{\ensuremath{c_{#1,#2,#3}}} 
\newcommand{\cvi}[2]{\ensuremath{c_{#1,#2}}}

\newcommand{\sG}[1]{\ensuremath{\mathcal{S}(#1)}} 
\newcommand{\dG}[1]{\ensuremath{\mathcal{D}(#1)}} 

\newcommand{\wtG}[1]{\ensuremath{\widetilde{#1}}}

\newcommand{\init}{\ensuremath{\operatorname{in}}}
\newcommand{\initwf}[2]{\ensuremath{\init_{#1}(#2)}} 

\newcommand{\Nw}[1]{\ensuremath{N(\partial #1)}} 
\newcommand{\Mw}[1]{\ensuremath{M(\partial #1)}} 
\newcommand{\Nwp}[1]{\ensuremath{N(#1)}} 
\newcommand{\Mwp}[1]{\ensuremath{M(#1)}} 

\newcommand{\wu}[1]{\ensuremath{w_{#1}}} 
\newcommand{\wubar}[1]{\ensuremath{\overline{w}_{#1}}} 
\newcommand{\Mwbar}[1]{\ensuremath{\overline{M}({#1})}}
\newcommand{\Nwbar}[1]{\ensuremath{\overline{N}({#1})}}

\newcommand{\codim}{\operatorname{codim}}

\newcommand{\ZHS}{\ensuremath{\Z}\!\operatorname{HS}}

\newcommand \Rp {\ensuremath{\R_{\geq 0}}}

\makeindex  

\begin{document}
\title[The Milnor fiber conjecture and an overview of its proof]{The Milnor fiber conjecture of Neumann and Wahl, \\and an overview of its proof}
\author[M.A. Cueto, P. Popescu-Pampu, D. Stepanov]{Maria Angelica Cueto, Patrick Popescu-Pampu${}^{\S}$ and Dmitry Stepanov\\ }
\thanks{${\S}$ \emph{Corresponding author}}
\date{29 November 2022}

\begin{abstract}
  \emph{Splice type surface singularities}, introduced  
   in 2002 by Neumann and Wahl, provide all 
   examples known so far of integral homology spheres which appear as links of 
   complex isolated complete intersections of dimension two. 
   They are determined, up to a form of equisingularity, 
   by decorated trees called \emph{splice diagrams}.  
   In {2005}, Neumann and Wahl formulated their \emph{Milnor fiber conjecture}, 
   stating that any choice of an internal edge of a splice diagram determines a special kind of 
   decomposition into pieces of the Milnor fibers of the associated singularities. {These}    
   pieces are constructed from the Milnor fibers of the splice type singularities 
   determined by the subdiagrams on both sides {of} the chosen edge. 
   \emph{In this paper we give an overview of this conjecture and a detailed 
   outline of its proof}, based on techniques from tropical geometry and {log geometry} in the 
   sense of Fontaine and Illusie. The crucial {log  geometric} ingredient is the operation 
   of {rounding} of a complex logarithmic space introduced in 1999 by Kato and Nakayama. 
   It is a functorial generalization of the operation of real oriented blowup. The use of the latter  
   to study Milnor fibrations was pioneered by A'Campo in 1975. 
 \end{abstract}
 
 \dedicatory{To Norbert A'Campo, on the occasion of his 80\textsuperscript{th}~birthday.}

\maketitle

\vspace{-1cm}

\tableofcontents

\section{Introduction}   \label{sec:introd}

Let $(X,o)$ be an irreducible germ of a complex analytic surface with isolated singularity at $o$, meaning 
that there exists a representative of it which is smooth outside $o$, and also possibly at $o$. We will 
say simply that $(X,o)$ is an \emph{isolated surface singularity}.
Denote by $\partial(X,o)$ its \emph{link}\index{link of a singularity}, 
obtained by intersecting a representative 
$X$ embedded in $\CC^n$ with a \emph{Milnor sphere}\index{Milnor!sphere}, that is, a sphere centered 
at $o$ of radius $r_0 > 0$,  such that any sphere 
centered at $o$ of smaller radius is transversal to $X$. 
The link $\partial(X,o)$ is a closed connected three-manifold, canonically 
oriented as the boundary of $(X \setminus \{o\})\cap \mathbb{B}(o,r_0)$, where  
$\mathbb{B}(o,r_0)$ denotes  the ball of radius $r_0$.

The classical works \cite{V 44}, \cite{M 61}, \cite{G 62} of Du Val, Mumford and Grauert 
show that oriented three-manifolds appearing 
as links of isolated surface singularities are exactly the graph manifolds which may be described 
by a negative definite and connected plumbing graph. Such oriented 
three-manifolds have canonical \emph{fillings} (that is, compact 
oriented four-manifolds having them as boundaries), given by the minimal good resolutions of 
$(X,o)$. Indeed, by work of Neumann~\cite[Theorem 2]{N 81}, the oriented topological type of the link
determines the oriented topological type of the minimal good resolution. 

Whenever $(X,o)$ is a complete intersection 
singularity, the link $\partial(X,o)$ admits another privileged filling, namely, the \emph{Milnor fiber} of 
any {\em smoothing}\index{smoothing} $f : (Y,o)\to (\CC,0)$ of $(X,o)$ 
(i.e., $f$ is the germ of a holomorphic map with smooth generic fibers 
and special fiber identified with $(X,o)$), a notion originating in Milnor's seminal 
book~\cite{M 68}. 
Indeed, by Tyurina's \cite[Theorem 8.1]{T 70}, the miniversal deformation 
of a complete intersection $(X,o)$ has an irreducible smooth base, 
therefore all the smoothings of $(X,o)$ have diffeomorphic 
Milnor fibers.

The following topological questions remain open:

\begin{question}\label{q:whichGraphsAreLinks}   
    Which oriented three-dimensional graph manifolds occur as links of hypersurface or 
    complete intersection isolated surface singularities?
\end{question}

\begin{question}\label{q:whichMlfdAreMilnorFibers}  
    Which oriented four-dimensional manifolds occur as Milnor fibers of smoothings of hypersurface,  
    complete intersection or arbitrary isolated surface singularities?
\end{question}

Most notably, very few conjectures have been proposed to address these questions. Among the most fruitful is the following one formulated by Neumann and Wahl in 1990~\cite{NW 90}:

\smallskip

\begin{namedconjecture}[\textbf{Casson Invariant}]\index{conjecture!Casson invariant}
  {If $(X,o)$ is a complete intersection isolated surface 
     singularity and its 
   link $\partial(X,o)$ is an integral homology sphere, then the Casson invariant of  
   $\partial(X,o)$ is equal to one-eighth of the signature of the Milnor fiber of $(X,o)$.}
\end{namedconjecture}

\smallskip

In \cite{NW 90}, Neumann and Wahl confirmed several instances of this conjecture, including its validity for all weighted homogeneous singularities, all suspension 
hypersurface singularities and a particular family of singularities in $\CC^4$. What hindered further progress was the lack of other examples of complete intersections with integral homology sphere links.

Fifteen years later, Neumann and Wahl made a breakthrough in this 
direction, by introducing a wide class of  examples, which they called 
\emph{splice type singularities}~\cite{NW 05}. This name is motivated by their construction. These singularities are defined by 
systems of equations whose structure is governed by special types of decorated trees, called \emph{splice diagrams}, which were introduced by Siebenmann 
in~\cite{S 80} to encode graph manifolds which are integral homology spheres. 

The term  \emph{splicing} was coined by Siebenmann to indicate a cut-and-paste operation introduced by Dehn in~\cite{D 07} to build new three-dimensional integral homology spheres from old ones. 
When splicing, solid tori are removed from two oriented integral homology spheres, and the resulting boundary 2-tori are then glued together by the unique isotopy class of diffeomorphisms which produces a new oriented integral homology sphere. Siebenmann proved that  any 
integral homology sphere graph manifold can be obtained by iterating this operation, starting from Seifert 
fibered integral homology spheres and always removing tubular neighborhoods of 
fibers (which can be special or not). The starting Seifert-fibered manifolds are encoded by weighted 
star-shaped trees. Their edges correspond to the special fibers and to the fibers used during 
the splicing process, and the weights record the orders of holonomies around those fibers. 
The splicing is recorded by joining the corresponding edges of the two trees involved. The resulting  
weighted tree is Siebenmann's splice diagram. 

Neumann and Wahl's splice type singularities are given by explicit systems of equations (see \autoref{def:splicesystem})  associated to splice diagrams which satisfy 
supplementary constraints (see Definitions \ref{def:edgedet} and \ref{def:semgpcond}).

In~\cite[Section 6]{NW 05}, Neumann and Wahl proposed an inductive approach for  
proving the Casson invariant conjecture for splice type singularities. The base case 
involved the so-called Brieskorn-Hamm-Pham complete intersections 
with integral homology sphere links which they had already established in~\cite{NW 90}. 
The inductive step would be achieved by an explicit description of the topology of Milnor fibers 
in terms of splicing. To this end, they proposed the following conjecture (formulated precisely in~\autoref{conj:MFC}):

 \begin{namedconjecture}[\textbf{Milnor Fiber}]\index{conjecture!Milnor fiber}
  { Let $(X,o)$ be a splice type singularity with an integral 
   homology sphere link. Assume that its splice diagram $\Gamma$ is the result of  splicing two other splice diagrams $\Gamma_a$ and $\Gamma_b$. Then, the Milnor fiber of $(X,o)$ is obtained by a four-dimensional splicing operation 
   from the Milnor fibers associated to $\Gamma_a$ and $\Gamma_b$.} 
\end{namedconjecture}

When restricted to the boundaries, Neumann and Wahl's \emph{four-dimensional splicing operation} becomes  Dehn's three-dimensional splicing. It resembles it in that it requires one 
to remove tubular neighborhoods of proper surfaces $G_a$ and $G_b$ embedded 
in the Milnor fibers $F_a$ and $F_b$ associated to $\Gamma_a$ and $\Gamma_b$, 
but it differs from it in that one does not glue directly the resulting 
four-dimensional manifolds with corners.  
Instead, they are glued to parts of the boundary of a third manifold with corners, namely 
the cartesian product $G_a \times G_b$.

The Casson invariant conjecture for splice type singularities with integral homology sphere links was proven by N\'emethi and Okuma \cite{NO 09} by rephrasing it as a statement about the geometric genus of the singularity $(X,o)$. Their proof involved explicit computations with resolutions of $(X,o)$, with no analysis of the Milnor fiber. As a result, the Milnor fiber conjecture remained open, only verified  by Neumann and Wahl themselves for suspension 
hypersurface singularities (see~\cite[Section 8]{NW 05}) and by Lamberson~\cite{L 09} for iterated suspensions.

\medskip

\textbf{In this  article, we present a step-by-step strategy for proving the Milnor fiber conjecture in full generality.} Technical details will appear in forthcoming work by the three authors. 
Our proof combines tools from both \emph{tropical geometry} and \emph{logarithmic geometry, 
  in the sense of Fontaine and Illusie}, and it is outlined in~\autoref{sec:stepsproof}. Central to our arguments is the concept of {\em rounding of a complex log structure} in the sense of Kato and Nakayama, which can be viewed as a generalization of A'Campo's 
{\em real oriented blowup}\index{real oriented blow up!A'Campo}. Roundings allow us to find good representatives for Milnor fibrations without the need to work with tubular neighborhoods. In addition, rather than requiring good resolutions for our constructions, we broaden the setting and work with toric modifications involving toric varieties whose associated fans are not regular. This extended setting facilitates the transition from the tropical to the logarithmic category, since it allows us to work with natural fans subdividing the local tropicalizations of our germs, without the need to further refine them into regular fans.

Described very concisely, our proof involves the following stages, starting from a 
splice type singularity $(X,o)$ defined by a splice type 
system associated to a splice diagram $\Gamma$:

\begin{enumerate}[(i)]
    
\item   \label{stage i}
   We define a particular deformation of the splice type system,  
   associated to a fixed  internal edge $[a,b]$ of $\Gamma$. 
   We let $(Y,0)$ be the three-dimensional germ obtained as the total space of this deformation. 
   We prove that the deformation is a smoothing $(Y,o)\to (\D,0)$ of $(X,o)$, 
   where  $\D$ denotes a compact two-dimensional disk with center $0$.

\item  \label{stage ii}
     Analogously, we define  $a$-side and $b$-side deformations of $a$-side 
      and $b$-side splice type singularities associated to the starting system, 
      by performing special monomial changes of variables in the previous deformed system, 
      that is, by taking pullbacks through special affine toric morphisms. We let $(Y_a, o)$ 
      and $(Y_b, o)$ denote their total spaces. 
   
    \item  \label{stage iii}
        We describe explicit fans subdividing the local tropicalizations of the deformations $(Y,o)$, 
       $(Y_a, o)$ and $(Y_b, o)$, which are compatible with the local tropicalizations 
       of the corresponding splice type singularities. As a preliminary step, 
       we describe an explicit fan subdividing the local tropicalization of $(X,o)$: topologically 
       it is a cone over the corresponding splice diagram. 
       This gives the first tropical interpretation of splice diagrams (see \autoref{rem:tropNND2D}).
       
     \item  \label{stage iv}
       We consider toric birational morphisms defined by these three fans and the corresponding
      strict transforms  of $(Y,o)$,   $(Y_a, o)$ 
      and $(Y_b, o)$, which we denote by $(\tilde{Y},o)$,   $(\tilde{Y}_a, o)$ and $(\tilde{Y}_b, o)$, 
      respectively. We show that the induced morphism from each strict transform to the 
      corresponding germ is a modification\index{modification}, that is, a proper bimeromorphic morphism.
      
   \item  \label{stage v}
      We  consider the associated morphisms $(\tilde{Y}, \cD) \to (\D, 0)$, 
     $(\tilde{Y}_a, \cD_a) \to (\D, 0)$, $(\tilde{Y}_b, \cD_b) \to (\D, 0)$, where $\cD$, $\cD_a$ 
     and $\cD_b$ are the preimages of $0$ under the previous modifications. 
      This allows us to apply a local triviality theorem of Nakayama and Ogus to the roundings 
        of the associated logarithmic morphisms, yielding representatives of the Milnor fibrations 
        of $(Y,o)\to (\D,0)$, $(Y_a,o)\to (\D,0)$ and $(Y_b,o)\to (\D,0)$ canonically associated 
        to the previous modifications. 
      
    \item  \label{stage vi}
       We show that the toric morphisms used to define the $a$-side and $b$-side 
       deformations induce embeddings of suitable log enriched exceptional divisors 
       of $\tilde{Y}_{a} \to Y_{a}$ and $\tilde{Y}_{b} \to Y_{b}$  into a similar enrichment 
       of the exceptional divisor of $\tilde{Y} \to Y$. 
       This implies analogous results for their roundings.
        
    \item  \label{stage vii}
       These facts, combined with the knowledge that one of the components of the exceptional divisor 
        of $\tilde{Y} \to Y$ is a cartesian product of two curves, establish the conjecture.
\end{enumerate}

As a direct consequence of the proposed proof, we uncover an unknown property of splice type singularities (see~\autoref{thm:invMF}):

\begin{theorem*} {\em The diffeomorphism type of the Milnor fiber of  a splice type singularity with integral homology sphere link depends solely on the underlying splice diagram.}
\end{theorem*}

The combined use of tropical and logarithmic geometry techniques to study  the topology of Milnor fibers is rather new. Since the inception of the research discussed in this paper and conference talks given by the second author on this subject, several articles applying logarithmic geometry to the study of problems about Milnor fibrations of singularities have appeared, including  works of 
Cauwbergs \cite{C 16}, 
Bultot and Nicaise \cite{BN 20}, Campesato, Fichou and  Parusi\'nski \cite{CFP 21}, and 
 Fern\'andez de Bobadilla and Pelka \cite{FP 22}. 
By presenting an overview of our techniques, we hope that this mainly expository  
article will help researchers apply similar ideas to address other questions involving 
the topological structure of Milnor fibers of smoothings of singularities.

\smallskip

The rest of the paper is organized as follows. In~\autoref{sec:splicetypeMF} 
we introduce background results leading to Neumann and Wahl's notion of \emph{splice type singularities}. \autoref{ssec:SeifZHS} surveys their genesis by reviewing  a presentation of the structure of \emph{Seifert fibered integral homology spheres} and the way they appear as links of isolated complete intersections of Brieskorn-Pham hypersurface singularities. \autoref{ssec:splicingtosing} provides detailed explanations on the three-dimensional \emph{splicing operation} and \emph{splice type integral homology spheres}. \autoref{ssec:spltypesing} reviews the construction of \emph{splice type singularities}.
Finally,~\autoref{ssec:mfconj} presents Neumann and Wahl's \emph{four-dimensional splicing operation}  and the precise formulation of their \emph{Milnor fiber conjecture}.

In~\autoref{sec:princproof} we discuss the main ideas of the proof of this conjecture.   \autoref{ssec:canMiln}  shows that A'Campo's operation of \emph{real oriented blowup} yields canonical representatives of the Milnor fibration over the circle of a smoothing, provided we are given an embedded resolution of the smoothing. 
In our context we do not work with embedded resolutions, but with more general 
morphisms which we call \emph{quasi-toroidalizations}. In~\autoref{ssec:ourusequasitor} we present a general theorem of  Nakayama and Ogus, stating the local triviality of a continuous map obtained by \emph{rounding}  in the sense of Kato and Nakayama of  suitable \emph{logarithmic morphisms} 
in the sense of Fontaine and Illusie. 
Finally,~\autoref{ssec:loctrop} shows how to build the aforementioned  
quasi-toroidalizations through explicit fan structures on the local tropicalizations 
of suitable deformations of splice type systems, combined with the Newton 
non-degeneracy property of these deformations.

\autoref{sec:logingred} presents detailed accounts of the logarithmic tools used to prove 
the Milnor fiber conjecture. In~\autoref{ssec:quasitor} we introduce 
the notions of {\em boundary-transversality}, of \emph{quasi-toroidal subboundary} 
and of \emph{quasi-toroidalization of a smoothing}. 
In \autoref{ssec:loggeomround} we lead the reader to the notion of log structure through 
a reformulation in a coordinate-independent way of the classical passage to polar coordinates. 
In~\autoref{ssec:logspacesmorph} we explain basic facts about the category of morphisms of 
complex \emph{log spaces} in the sense of Fontaine and Illusie. 
In \autoref{ssec:chartslogstruct} we list various kinds of {\em monoids} needed in 
the sequel, as well as the associated log structures, defined in terms of {\em charts}.  
In~\autoref{ssec:KNrounding} we define Kato and Nakayama's \emph{rounding operation} 
on complex log spaces and we explain some of its basic properties. 
In~\autoref{ssec:loctrivquasitor} we revisit 
Nakayama and Ogus' local triviality theorem and apply it in the context of
quasi-toroidalizations of smoothings. 

In~\autoref{sec:loctropNND} we introduce  the tropical ingredients of our proof: the notion of 
\emph{local tropicalization} of an analytic germ  contained in $(\CC^n,0)$ and the notion of 
\emph{Newton non-degeneracy}. 

\autoref{sec:edgedeform} presents the explicit deformations of splice type systems 
appearing in Stage \ref{stage ii} above. 

The paper concludes with~\autoref{sec:stepsproof}, 
 in which we give a detailed proof outline in $28$ steps 
 of the six stages discussed earlier to establish the Minor fiber conjecture.

\vfill

\pagebreak

\section{Splicing, splice type singularities and the Milnor fiber conjecture} \label{sec:splicetypeMF}

The operation of splicing and the construction of splice type singularities are central components of the Milnor fiber conjecture. In this section, we review 
Seifert's classification of Seifert fibered integral homology spheres, henceforth denoted by $\ZHS$'s 
(see \autoref{prop:seifzhs}), Neumann's realization of those  $\ZHS$'s as links of isolated complete intersections of \emph{Pham-Brieskorn-Hamm type} (see \autoref{prop:isolsingBH}) and the cut-and-paste operation of \emph{splicing} of $\ZHS$'s along knots (see \autoref{def:splicing}).
In addition, we recall the genesis of \emph{splice diagrams} (see \autoref{def:splicediag}) 
as graphs introduced by Siebenmann to encode $\ZHS$'s which are graph manifolds (see \autoref{def:splicediagZHS}), and Eisenbud and Neumann's characterization of splice diagrams encoding all singularity links which are $\ZHS$'s (see \autoref{thm:charZHSlinks}). We describe how these results motivated Neumann and Wahl to define \emph{splice type systems} and \emph{splice type singularities} associated to splice diagrams which satisfy the so-called \emph{determinant} and \emph{semigroup conditions} (see \autoref{ssec:spltypesing}).  The section concludes 
with a discussion of
Neumann and Wahl's \emph{four-dimensional splicing} operation  (see \autoref{def:fourdimsplice}) 
and with the statement of their \emph{Milnor fiber conjecture} (see \autoref{conj:MFC}).

\medskip
\subsection{Seifert fibered integral homology spheres}  \label{ssec:SeifZHS}$\:$ 
\medskip

We start this subsection by explaining the notions of {\em integral homology sphere} 
(or $\ZHS$, see \autoref{def:ZHS}), 
of {\em meridian} and {\em longitude} of a knot in a $\ZHS$ (see \autoref{def:merparzhs}) 
and of {\em Seifert fibration} (see \autoref{def:seifert}). 
In addition, we discuss various results  that predate the 
notion of \emph{splice type singularity}, from the appearance of 
\emph{Poincar\'e's homology sphere} as the link of the $E_8$ surface singularity 
(see \autoref{ex:quasihomE8})  to \emph{Seifert's classification of Seifert fibered 
links of singularities which are moreover integral homology spheres} (see \autoref{prop:seifzhs}). 
\medskip

In the sequel, we denote by $\boxed{\partial_{top} W}$ the {\em boundary} 
of a smooth or topological manifold 
with boundary. In contrast, we use  $\boxed{\partial W}$ to denote the 
{\em algebro-geometric boundary} of a toroidal variety $(W, \partial W)$ (see~\autoref{def:toroimorph}).  
If $V \hookrightarrow W$ is a properly embedded submanifold with boundary of a manifold 
with boundary, we use $\boxed{N_W(V)}$ to denote a topologically 
closed {\em tubular neighborhood}  
of $V$ in $W$. Note that $N_W(V)$  has the structure of a disk bundle over $V$, whose fibers have dimension equal to the 
codimension  of $V$ in $W$, which we denote by $\boxed{\codim_{W} (V)}$.
Its intersection with $\partial_{top} W$ is a tubular neighborhood of the boundary 
$\partial_{top} V$ inside $\partial_{top} W$.

The next class of three-dimensional manifolds is central to this paper:

\begin{definition}   \label{def:ZHS}
    An \textbf{integral homology sphere}\index{integral homology sphere}, 
    briefly written $\boxed{\ZHS}$, 
    is a closed smooth three-manifold  
    which has the total integral homology group of a three-dimensional 
    sphere. A $\ZHS$ is called \textbf{trivial} if, and only if,  it is homeomorphic to the unit three-dimensional sphere  $\bS^3$. 
\end{definition}

\begin{remark}\label{rm:firstPropertiesZHS} 
   Important properties follow from the $\ZHS$ condition.
   Indeed, if $M$ is a $\ZHS$, then $H_0(M, \Z) \simeq \Z$. Thus, $M$ must be connected. 
   In addition, as $ H_3(M, \Z) \simeq \Z$, we see that $M$ is also orientable. 
\end{remark}

\begin{remark}\label{rm:characterizationOfZHS}
     Fixing an orientation on a closed orientable three-manifold $M$ determines a  well-defined     
     Poincar\'e duality isomorphism $H_2(M, \Z) \simeq H^1(M, \Z)$.
     Since $H^1(M, \Z)$  $\simeq  \Hom(H_1(M, \Z), \Z)$,   
     by the universal coefficients theorem,  we conclude that 
     $M$ is a $\ZHS$ if, and only if,  it is a connected and orientable three-manifold with 
     $H_1(M, \Z)=0$. 
\end{remark}

Throughout, we assume that all integral homology spheres are {\em oriented}, i.e., 
they are endowed with fixed orientations.  Such manifolds $M$ admit a well-defined notion of 
{\bf linking number}\index{linking number} $\boxed{lk(K_1, K_2)}$ between 
any two disjoint oriented knots 
$K_1, K_2$ on them: it is the intersection number between $K_1$ and an oriented 
surface $S_2 \hookrightarrow M$ with boundary $K_2$. The fact that such a surface exists and 
that this intersection number is independent of the choice of $S_2$ is a direct consequence of 
the vanishing of $H_1(M, \Z)$. The 
linking number is symmetric in its two arguments, a property which we will frequently exploit.

Given any integral homology sphere, the boundary of a tubular neighborhood of a knot in it 
is canonically trivialized, up to isotopy, as the next result shows. 
For more details, we refer to~\cite[page 21]{EN 85} and~\cite[\S 6]{S 80}. 

\begin{proposition}  \label{prop:trivbd}
   Let $M$ be an oriented $\ZHS$ and let $K$ be an oriented knot in $M$. Let $N_M (K)$ 
   be a tubular neighborhood of $K$ in $M$. Then, there exist embedded oriented circles $\mu$ 
   and $\lambda$ on $\partial_{top} N_M (K)$, well-defined up to isotopy, such that $lk(\mu, K) =1$, 
   $lk(\lambda, K)=0$ and the homology classes of $\lambda$ and $K$ in $H_1(N_M (K), \Z)$ 
   coincide. Moreover, the classes of $\mu$ and $\lambda$ in $H_1(\partial_{top} N_M (K), \Z)$ 
   form a basis of this lattice. 
\end{proposition}

The previous statement determines the notions of meridian and longitude of oriented knots in integral homology spheres, which we now recall:

\begin{definition}   \label{def:merparzhs}
   Let $M$ be an oriented $\ZHS$ and let $K$ be an oriented knot in $M$. The oriented curves 
   $\mu$ and $\lambda$ characterized in~\autoref{prop:trivbd} are called a 
   \textbf{meridian}\index{meridian} and a \textbf{longitude}\index{longitude} of $K$, respectively. 
\end{definition}

Replace the previous paragraph with the following one:
Sometimes, any oriented simple closed curve on $\partial_{top} N_M (K)$ whose homology class gives a basis of $H_1(\partial_{top} N_M (K), \Z)$ when completed by that of 
$\mu$ is called a {\em longitude} of $K$ (the curves characterized in \autoref{def:merparzhs} 
being then called  {\em topologist's longitudes}). As we will not consider these more general types of 
longitudes, we refrain from using this terminology.

Meridians and longitudes are essential to defining three-dimensional \emph{splicings}, 
as seen in~\autoref{prop:splicezhs} and~\autoref{def:splicing} below. 
They are denoted by $m_i$ and $\ell_i$ in~\autoref{fig:3DSplicing}. 

\medskip
The first example of non-trivial $\ZHS$ was given by Poincar\'e in his 1904 paper \cite{P 04}: it is the famous \textbf{Poincar\'e homology sphere}\index{Poincar\'e homology sphere}. He defined it using a Heegaard diagram. Notably, it can also be defined as the \emph{link}  of the $\boxed{E_8}$ surface singularity, i.e., the germ at the origin of the complex affine surface in $\CC^3_{x,y,z}$ 
defined by the equation
  \begin{equation*}   
      x^2 + y^3 + z^5 =0.
  \end{equation*}
  It is not at all obvious that these two three-manifolds are homeomorphic. 
  This was established only after the introduction of \emph{Seifert fibered} three-manifolds 
  by Seifert in his 1933 paper \cite{S 33}. The crux of the proof is to show that both manifolds 
  are  Seifert fibered integral homology spheres and that their Seifert fibrations have the same 
  numerical invariants (see \autoref{ex:quasihomE8}). For further details on 
  the first studies of the Poincar\'e homology sphere, we refer to Gordon's work \cite[Section 6]{G 99}.  
  For other characterizations, the reader may consult
  Kirby and Scharlemann's paper~\cite{KS 79}, or  Saint Gervais' website~\cite{SG 17}. 
  
Next, we review the notion of  \emph{Seifert fibration} on closed oriented three-manifolds. For further details, the reader may consult Orlik's book~\cite{O 72} or Neumann and Raymond's paper \cite{NR 78}.
  
  \begin{definition}  \label{def:seifert}
 A \textbf{Seifert fibration}\index{Seifert fibration} on  a closed oriented three-manifold  is an 
         orientable foliation by circles. Its \textbf{base} is the space of leaves endowed 
         with the quotient topology. Its \textbf{fibre map} is the quotient map. A manifold 
         endowed with a Seifert fibration is called \textbf{Seifert fibered}. 
  \end{definition}  

\begin{remark}\label{rem:SeifertFibrations} 
   It can be shown that the base $S$ of a Seifert fibration on a closed oriented 
   three-manifold $M$ is an orientable closed surface and that the fibre map $\psi\colon M \to S$ 
   is a locally trivial  circle bundle away from  a finite set of points of $S$. Those points correspond 
   to the so-called \textbf{special fibers}  of the Seifert fibration. 
\end{remark}

\begin{remark}\label{rem:originalDefSeifertFib}
   Following Seifert's original approach from \cite[Section 1]{S 33}, 
   Seifert fibrations are often  defined as  maps $\psi \colon M \to S$ which are locally trivial 
   on $S$ away from  the neighborhood of a finite set of points and which  have prescribed 
   models in the  neighborhoods of the special fibers  (see, for instance,~\cite[Section 5.2]{O 72}). 
  These models can be described using the  holonomy  of the foliation along a special fiber $C$. 
  Turning once around $C$ yields a  diffeomorphism of a transversal slice, which is isomorphic 
  to a finite-order rotation of a disk.  
  Such a rotation may be encoded by a rational number $q/p \in (0, 1) \cap \Q$, 
  with $p$ and $q$ coprime. The integer $p \geq 2$ is the order of the rotation, that is,  
  the degree of  the quotient map $\psi$ restricted to a transversal slice of $C$. For this reason, 
  we call it the \textbf{degree} of the point $\psi(C) \in S$. It can also be interpreted  
  as the number  of times the leaves situated in the neighborhood of $C$ turn around  $C$. 
\end{remark}
 
The basic numerical invariants of a Seifert fibration are the pairs $(p,q)$ associated  to its special fibers and the topological type of the base surface $S$. These, combined with the 
{\em rational Euler number} of the fibration 
(see, e.g.,~\cite[Section 1]{NR 78}, \cite[Section 1]{N 83} or \cite[Section I.3]{JN 83}),  
determine the fibration up to a homeomorphism of $M$  preserving the foliation and the orientation.   
For general Seifert fibrations, the Euler number is rational and it changes sign  if the orientation 
on $M$ is reversed. When $\psi \colon M \to S$ is a locally trivial circle bundle, we have 
no special fibers and the rational Euler number of the fibration agrees with 
the usual  Euler number of the bundle; thus, it is an integer.

Seifert fibrations of non-trivial integral homology spheres are well-understood, as the following 
theorem of Seifert confirms. For details, we refer the reader to  
works of Seifert \cite[Theorem 12]{S 33}, Neumann and Raymond \cite[Section 4]{NR 78}  or 
Eisenbud and Neumann \cite[Chapter II.7]{EN 85}): 
 
 \begin{proposition}  \label{prop:seifzhs}
   If  $M$ is a non-trivial $\ZHS$ that admits a Seifert fibration, then this fibration is unique up to  isotopy. Furthermore, its base is a two-dimensional sphere and it has at least three special fibers, with pairwise coprime degrees. Conversely, given $n\geq 3$ and a sequence $(p_1, \dots, p_n)$ of pairwise coprime positive integers with $p_i\geq 2$ for all $i$, there exists a unique Seifert fibered $\ZHS$ up to homeomorphisms, whose base is a two-dimensional sphere and whose special fibers have degrees $p_1, \dots, p_n$. With either orientation, the Euler number of this fibration is non-zero.   
 \end{proposition}

The previous proposition allows to define integral homology spheres from sequences of pairwise coprime positive integers:

\begin{definition}   \label{def:notationseifzhs}
   Fix $n\geq 3$ and let $(p_1, \dots, p_n)$ be a sequence of pairwise coprime positive integers 
   with $p_i\geq 2$ for all $i\in \{1,\ldots, n\}$ . The oriented 
    three-dimensional manifold $\boxed{\Sigma(p_1, \dots, p_n)}$ is the unique oriented 
    $\ZHS$ which admits a Seifert fibration with a \emph{negative} Euler number and whose sequence 
    of degrees of special fibers is $(p_1, \dots, p_n)$, up to permutation. 
\end{definition}

\begin{remark}\label{rem:pisCanBeOne}
         Note that in both~\autoref{prop:seifzhs} and~\autoref{def:notationseifzhs} 
         we assume $p_i\geq 2$ for all $i\in \{1,\ldots, n\}$. We can extend~\autoref{def:notationseifzhs} 
         to allow for $p_i \geq 1$, by simply removing all terms of the sequence with value one, 
         defining the corresponding Seifert fibered $3$-manifold and identifying indices $i$ 
         with $p_i=1$ with non-special fibers of the fibration. For instance, 
         $\Sigma(1, 1, 3, 8, 35) = \Sigma(3, 8, 35)$  and  the first two elements of 
         the sequence $(1, 1, 3, 8, 35)$  witness  two non-special fibers  
         of $\Sigma(3, 8, 35)$  (see~\autoref{rem:PBHAreSeifertFibered} below). Allowing 
         some $p_i$'s to take value $1$ is important 
         in the construction of integral homology spheres from splice diagrams 
         (see~\autoref{def:splicediagZHS} below). 
\end{remark}

\begin{example} \label{ex:quasihomE8}
    Consider the polynomial $f:= x^2 + y^3 + z^5$ defining the complex 
    surface $X$ whose germ at the origin is the $E_8$ singularity.
    The polynomial $f$ is homogeneous relative to the weight vector 
    $w := (3 \cdot 5, 2 \cdot 5, 2 \cdot 3)$. Therefore,  the surface $X$ is invariant under 
    the following natural action of the group $(\CC^*_t, \cdot)$ on $\CC^3_{x,y,z}$:
       \begin{equation}   \label{eq:invsurfE8}
             t \cdot (x,y,z) := (t^{3 \cdot 5} x, t^{2 \cdot 5} y , t^{2 \cdot 3} z) . 
       \end{equation}
   Thus, it is invariant under the action of  
    the circle $(\bS^1, \cdot)$ of $(\CC^*_t, \cdot)$. Similarly, 
    all Euclidean spheres $\bS^5_{\varepsilon}$ centered at the 
    origin of $\CC^3_{x,y,z}$ (of radius $\varepsilon>0$) are invariant under this action 
    of the circle. As a consequence, the intersections $X \cap \bS^5_{\varepsilon}$     
    are also invariant. 
Note that the manifolds $X \setminus \{0\}$ and $\bS^5_{\varepsilon}$  
 intersect transversally, as the orbits of 
the $(\R_t^{*}, \cdot)$-action on $X$ induced by the above $(\CC_t^{*}, \cdot)$-action  
are transversal 
to the spheres $\bS^5_{\varepsilon}$. Therefore, these  intersections are representatives 
of the link $\partial(X,0)$ of the singularity $(X,0)$. In particular, this 
 shows that there exists an action  of $(\bS^1, \cdot)$ on this link with no fixed points.
Its orbits  determine a Seifert fibration on $\partial(X,0)$.

A closer look at the action~\eqref{eq:invsurfE8} confirms that  the previous Seifert  fibration 
has exactly three special fibers (the intersections with the planes of coordinates), 
    with degrees  $2, 3$ and $5$. Moreover, $\partial(X,0)$ is a $\ZHS$. 
    This fact may be proved in several ways:

     \begin{itemize}
        \item  By seeing it as a ramified cover of $\bS^3$ of degree $5$, ramified over the 
             trefoil knot, and using Seifert's characterization 
              \cite[Addendum to Theorem 17, page 413 of the 
              English version of Seifert and Threlfall's book]{S 33}) of such covers which 
              are integral homology spheres. 
         \item   By using Brieskorn's criterion \cite[Satz 1, page 6]{B 66} 
            (see also Dimca's \cite[Theorem 4.10, page 94]{D 92}), described first 
           in a letter of Milnor to Nash (see \cite[page 47]{B 00}), allowing to determine 
           when the link of a Pham-Brieskorn hypersurface singularity of arbitrary dimension 
           (see \autoref{rem:PBH}) is an integral homology sphere.
         \item By computing the weighted dual graph of the minimal good resolution of $(X,0)$, 
             which is a tree of components of genus zero (it is the so-called $E_8$-tree 
             of Lie groups theory), 
            and by proving that the associated intersection form is unimodular, which implies 
            that the link is indeed an integral homology sphere (see \cite[Proposition 3.4, page 52]{D 92}). 
            The weighted dual graph may be computed either using the Jung-Hirzebruch method, 
            as explained by Laufer \cite[pages 23--27]{L 71} or using the $(\CC_t^{*}, \cdot)$-action, 
            as explained by Orlik and Wagreich in \cite[Section 3]{OW 71} (see also 
            \cite[pages 64--67]{D 92} and \cite[Theorem 4.2]{M 00}). 
         \item By using the facts that the intersections $Z(x|_X), Z(y|_X), Z(z|_X)$  
             of $X$ with the three coordinate planes 
            are irreducible germs of curves and that their strict transforms by the 
            minimal good resolution $\pi: \tilde{X} \to X$ of $X$ 
            intersect transversally the exceptional divisor $E$ at its components 
            associated with the leaves of the 
            dual tree. Then, Neumann and Wahl's \cite[Proposition 5.1]{NW 05bis} 
            implies that the duals of those components (in the intersection lattice 
            $H_2(E, \Z)$ of $E$ endowed with its intersection form inside $\tilde{X}$) 
            generate the discriminant group $H_2(E, \Z)^{\vee}/H_2(E, \Z)$ of $E$, which identifies 
            canonically to $H_1(\partial(X,0), \Z)$. The irreducibility and 
            the transversality properties mentioned above imply that those duals are equal to 
            the opposites of the exceptional parts of the total transforms of 
            $Z(x|_X), Z(y|_X), Z(z|_X)$ by $\pi$. 
            Thus they have integral coefficients, that is, they belong to $H_2(E, \Z)$.  
            This implies that the discriminant group is trivial,  therefore $H_1(\partial(X,0), \Z)$ is also 
            trivial. 
        \end{itemize}

    \autoref{prop:seifzhs} ensures now that  $\partial(X,0)$ 
    is the Seifert fibered integral homology sphere  $\Sigma(2,3,5)$. 
\end{example}

Work of~Neumann \cite{N 77} characterizes the integral homology sphere 
$\Sigma(p_1, \dots, p_n)$ as a singularity link:

\begin{theorem}   \label{thm:hammbrieslinks}
     Fix $n\geq 3$ and let $(p_1, \dots, p_n)$ be a sequence of pairwise coprime positive integers with $p_k\geq 2$ for all $k\in \{1,\ldots, n\}$. Let $(c_{i,j})_{1 \leq i \leq n-2, 1 \leq j \leq n}$ be a matrix of complex numbers 
     all of whose maximal minors are non-zero. Then, the subspace of $\CC^n$ 
     defined by the system of equations:
         \begin{equation} \label{eq:BHsyst} 
               \left\{ \begin{array}{cccccc}
                    c_{1,1} z_1^{p_1} &  + &  \cdots   &  + &  c_{1,n} z_n^{p_n} & =0, \\
                      \vdots &   \vdots & \vdots & \vdots & \vdots & \vdots \\
                    c_{n-2,1} z_1^{p_1} &  + &  \cdots   &  + &  c_{n-2,n} z_n^{p_n} & =0,
              \end{array}  \right.
         \end{equation}
      is an irreducible surface with an isolated singularity at $0$ whose link, oriented as the boundary of a neighborhood of $0$, is (orientation-preserving) homeomorphic to the integral homology sphere $\Sigma(p_1, \dots, p_n)$. 
\end{theorem}

The condition that all the maximal minors of the matrix of coefficients are non-zero is equivalent to the condition that the previous system defines an isolated complete intersection singularity at the origin of $\CC^n$. This is a direct consequence of the following more general result of 
Hamm (see~\cite[\S 5]{H 69} and \cite{H 72}):

\begin{proposition}   \label{prop:isolsingBH}
     Fix $n\geq 3$ and let $(p_1, \dots, p_n)$ be a sequence of positive integers with $p_i\geq 2$ for all $i\in \{1,\ldots, n\}$, and 
   fix $k \in \{1, \dots, n-1 \}$. 
     Consider a $k\times n$-matrix $(c_{i,j})_{i,j}$ 
       with complex entries.
     Then, the system of equations
         \begin{equation} \label{eq:BHsystgen} 
               \left\{ \begin{array}{cccccc}
                    c_{1,1} z_1^{p_1} &  + &  \cdots   &  + &  c_{1,n} z_n^{p_n} & =0, \\
                      \vdots &   \vdots & \vdots & \vdots & \vdots & \vdots \\
                    c_{k,1} z_1^{p_1} &  + &  \cdots   &  + &  c_{k,n} z_n^{p_n} & =0,
              \end{array}  \right.
         \end{equation}
      defines an isolated complete intersection singularity at $0$ in $\CC^n$ if, and only if,  all maximal minors of the input matrix $(c_{i,j})_{i,j}$ are non-zero.
\end{proposition}

\begin{remark}   \label{rem:PBH}
  Notice that each equation of~\eqref{eq:BHsystgen} defines a so-called 
\textbf{Pham-Brieskorn hypersurface singularity}\index{singularity!Pham-Brieskorn} 
(see Bries\-korn's paper \cite[pages 47--49]{B 00} 
for an explanation of this terminology). For this reason, isolated complete intersection singularities (ICIS) defined by these systems are sometimes called  
\textbf{Pham-Brieskorn-Hamm singularities}\index{singularity!Pham-Brieskorn-Hamm}.
\end{remark}

\begin{remark}\label{rem:ExoticSpheres}
    If a Pham-Brieskorn-Hamm singularity of complex dimension at least three has an 
    integral homology sphere link (i.e., its link has the integral homology of  a sphere of the 
    same dimension), then this link is homeomorphic to a sphere. Indeed, as proved by Milnor 
    \cite[Theorem 5.2]{M 68} for isolated singularities of hypersurfaces 
    and extended by Hamm \cite[Kor. 1.3]{H 71} to ICIS, their links are simply connected. 
    In turn, by a  theorem of Smale \cite{S 61}, a simply connected integral homology sphere 
    of dimension  at least five is homeomorphic to a sphere. Brieskorn discovered in \cite{B 66} 
    (see also \cite{B 00}) that for hypersurfaces, such links could be exotic spheres. 
    Subsequent work by Hamm \cite{H 72} extended the study of such exotic spheres to all 
    Pham-Brieskorn-Hamm singularities. 
\end{remark}

\begin{remark}\label{rem:PBHAreSeifertFibered} 
       Notably, the link of a Pham-Brieskorn-Hamm surface singularity  $(X,0)$ defined by the system
       \eqref{eq:BHsyst} is always Seifert-fibered, even when it is not an integral 
       homology sphere. This fact can be proven using the same group-action methods    
       from~\autoref{ex:quasihomE8}. Indeed, the surface $X$ is invariant under 
       the  action of $(\CC^*, \cdot)$ on $\CC^n_{z_1, \dots, z_n}$ given by  
          \begin{equation}\label{eq:groupAction}
                \CC^* \times \CC^n  \to \CC^n \qquad               
               (t, (z_1, \dots, z_n))\mapsto   (t^{p_2 \cdots p_n} z_1, \dots, t^{p_1 \cdots p_{n-1}} z_n) .
          \end{equation}
       Furthermore, the special fibers are obtained as  the intersections 
       of  $X \cap  \bS^{2n-1}_{\varepsilon}$ with some hyperplanes of coordinates.
        If the integers $p_i$ are pairwise coprime (as required for~\autoref{thm:hammbrieslinks}), 
        it follows that the degree of the  fiber $X \cap  \bS^{2n-1}_{\varepsilon} \cap Z(z_i)$ 
        equals $p_i$, for every $i \in \{1, \dots, n\}$. In particular, we see that this fiber is special if, 
        and only if,  $p_i >1$. 
\end{remark}

\medskip
\subsection{From three-dimensional splicing to splice type singularities}  
   \label{ssec:splicingtosing}$\:$ 
\medskip

In this subsection we explain how to build new integral homology spheres from old ones 
by \emph{splicing} them along oriented knots (see \autoref{def:splicing}). 
Then, we introduce \emph{splice diagrams} (see \autoref{def:splicediag}), 
which are particular decorated trees encoding the result of successive splicings 
of Seifert fibered integral homology spheres along some of their fibers 
(see \autoref{def:splicediagZHS}). We continue by explaining Eisenbud and Neumann's characterization 
of splice diagrams describing the $\ZHS$ which appear as links of isolated complex surface 
singularities (see \autoref{thm:charZHSlinks}). This characterization uses the notion of 
{\em edge determinant condition} (see \autoref{def:edgedet}). We conclude 
by explaining Neumann and Wahl's \emph{semigroup condition} on the decorations 
of splice diagrams (see \autoref{def:semgpcond}), which we use in~\autoref{ssec:spltypesing}  
to define  \emph{splice type singularities}.

\medskip
\autoref{prop:seifzhs} above characterizes the non-trivial integral homology spheres which are 
Seifert fibered. A natural question arises: are there other $\ZHS$'s? It turns out that there 
are many more! In order to explain  
this fact it is useful to introduce the following terminology:

\begin{definition}   \label{def:cutmanif}
    Let $M$ be a compact manifold (with or without boundary) and let 
    $K \hookrightarrow M$ be a properly embedded submanifold. 
    A \textbf{classical cut\index{classical cut} $\boxed{\cC_{K} M}$ of $M$ along $K$}   is the closure 
    inside $M$ of the complement of a compact tubular neighborhood of $K$ in $M$. 
\end{definition}

\begin{remark}\label{eq:cutsThroughKnots} 
       Note that if $K$ is a knot in a three-manifold $M$, then the boundary 
      $\partial_{top} (\cC_{K} M)$ of $\cC_{K} M$ is a two-dimensional torus. 
      If $M$ is moreover an oriented $\ZHS$, 
      then $\partial_{top} (\cC_{K} M)$ contains preferred isotopy classes of curves, namely, those of 
      the meridians and longitudes of $K$ in the sense of~\autoref{def:merparzhs}. 
\end{remark}

As stated in the next proposition, new $\ZHS$'s can be obtained from 
a pair of $\ZHS$'s with prescribed embedded knots by  gluing  the corresponding 
classical cuts appropriately, as seen in~\autoref{fig:3DSplicing}. For details, we refer 
the reader to~\cite[Section 1.1]{EN 85}.
\begin{proposition}  \label{prop:splicezhs}
     Let $M_1$ and $M_2$ be two oriented $\ZHS$'s and let 
     $K_i \subset M_i$ be oriented knots in them. Consider classical cuts 
     $\cC_{K_i} M_i$ of $M_i$ along $K_i$ in the sense of~\autoref{def:cutmanif} 
     and let $M$ be  the manifold obtained by gluing $\cC_{K_1} M_1$ and $\cC_{K_2} M_2$ 
     through a diffeomorphism of the tori $\partial_{top} (\cC_{K_i} M_i)$  (for $i=1,2$)
     which permutes their meridians and longitudes. Then, the manifold $M$ is also a $\ZHS$. 
\end{proposition}

\begin{figure}[tb]
  \includegraphics[scale=0.65]{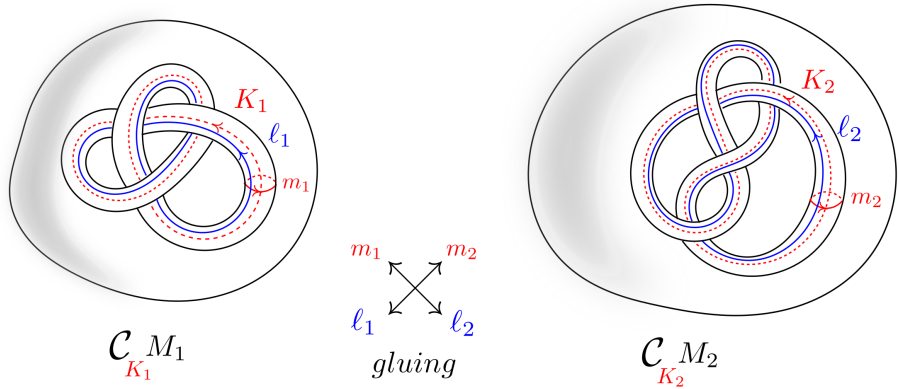}
  \caption{Splicing the integral homology spheres $M_1$ and $M_2$ along the knots 
      $K_1$ and $K_2$ (see \autoref{def:splicing}).\label{fig:3DSplicing}}
  \end{figure}

As mentioned by Gordon in~\cite[Section 6]{G 99}, 
the previous property had been noticed by Dehn in his 1907 paper \cite{D 07} for a pair of three-dimensional spheres. 
The following terminology  describing
the operation performed in~\autoref{prop:splicezhs} is due to 
Siebenmann \cite{S 80}:

\begin{definition}   \label{def:splicing}
       Let $M_1$ and $M_2$ be two oriented $\ZHS$'s and let 
     $K_i \subset M_i$ be oriented knots in them.  Then, the oriented three-manifold 
     $\boxed{(M_1, K_1) \oplus (M_2, K_2)}$  obtained by the procedure described 
     in~\autoref{prop:splicezhs} is called the \textbf{splice of $M_1$ and $M_2$ 
     along  the knots $K_1$ and $K_2$}\index{splicing!three-dimensional}. 
\end{definition}

The splicing operation is sketched in~\autoref{fig:3DSplicing}. The meridians are denoted by $m_1$ and $m_2$, whereas the longitudes are indicated by $\ell_1$ and $\ell_2$. The curves
in the figure are schematic, i.e., they should not  be interpreted as  
linear projections of knots in the standard sphere. Otherwise, the knot $l_i$ 
would not be a longitude of $K_i$, as their linking number would not be zero.
\smallskip

In the same article~\cite{S 80} in which he had introduced the {\em splicing} terminology, 
  Siebenmann considered the special class of 
$\ZHS$'s obtained from several Seifert-fibered 
ones by splicing them recursively along fibers of their respective Seifert fibrations. 
He encoded the resulting oriented $\ZHS$'s by special types of decorated trees   
called \emph{splice diagrams}, which we now discuss. We start by recalling 
some standard terminology from graph theory:

\begin{definition} \label{def:trees}
          A \textbf{tree}\index{tree} $T$ is a finite acyclic connected graph. 
          The \textbf{valency}\index{valency of a vertex}  of a vertex $v$ 
          is the number of edges incident to it, which we denote by $\boxed{\valv{v}}$. 
          When $T$ has at least two vertices, 
          the \textbf{leaves}\index{tree!leaf of} of $T$ are those vertices of valency one, 
          and the \textbf{nodes}\index{tree!node of}\index{node!of a tree}
          of $T$ are the remaining vertices. If $T$ is a singleton, 
          its unique vertex is taken to be a leaf. An edge joining two nodes is called 
          \textbf{internal}\index{tree!internal edge of}. 
\end{definition}

\begin{definition}  \label{def:splicediag}
          A \textbf{splice diagram}\index{splice diagram}\index{tree!splice diagram} 
          is a finite tree without vertices of valency two, 
          such that for each node $v$, every incident edge $e$ is decorated by a positive integer 
          $\boxed{\du{v,e}}$ in the neighborhood of $v$ 
          and such that around each node, the integers decorating adjacent  
          edges are pairwise coprime. 
          A \textbf{star-shaped splice diagram}\index{splice diagram!star-shaped} 
          is a splice diagram with a single node. 
\end{definition}

Siebenmann's work~\cite{S 80} associates an oriented integral homology sphere 
to every splice diagram by an explicit procedure, which we now recall:

\begin{definition}   \label{def:splicediagZHS}
          Let $\Gamma$ be a splice diagram. Its \textbf{associated oriented integral homology sphere} 
          $\boxed{\Sigma(\Gamma)}$ is constructed as follows:
             \begin{itemize}
                  \item For each node $v$ of $\Gamma$, let $\boxed{\Gamma^v}$ 
                      be the star-shaped splice diagram obtained by taking the union 
                      of the compact edges of $\Gamma$ containing $v$ and by keeping 
                      only their decorations around $v$.
                  
                  \item  If $(p_1(v), \dots, p_{\valv{v}}(v))$ is the sequence of decorations 
                     on the edges of $\Gamma^v$ (arbitrarily ordered), consider the 
                     Seifert-fibered integral homology sphere 
                     $\boxed{\Sigma(\Gamma^v)}:= \Sigma(p_1(v), \dots, p_{\valv{v}}(v))$ 
                     (see~\autoref{def:notationseifzhs} and~\autoref{rem:pisCanBeOne}), 
                     with its fibers oriented arbitrarily, but in a continuous way. 
                     The manifold $\Sigma(\Gamma^v)$ has a set of $\valv{v}$  distinguished 
                     fibers in bijection with the set of edges of $\Gamma$ adjacent to $v$. 
                    
                  \item Given two adjacent nodes  $u$ and $v$ of $\Gamma$,  splice  
                      $\Sigma(\Gamma^{u})$ and $\Sigma(\Gamma^{v})$  along the oriented 
                      fibers  corresponding to the unique edge of $\Gamma$ joining $u$ and $v$. 
                     
                  \item Perform the previous splicing simultaneously on the disjoint union of all 
                      oriented Seifert-fibered integral homology spheres with oriented fibers 
                      $\Sigma(\Gamma^{v})$, indexed by all nodes  $v$ of $\Gamma$. 
                      
                  \item The resulting oriented integral homology sphere is $\Sigma(\Gamma)$. 
             \end{itemize}
\end{definition}

\begin{remark} 
  Siebenmann's construction is more general and allows negative edge weights on splice diagrams.    
  We focus on the case of positive weights since this restriction 
   is enough for describing the singularity links which 
  are $\ZHS$'s (see~\autoref{thm:charZHSlinks}). 
\end{remark}

\autoref{thm:hammbrieslinks} shows that all  integral homology spheres associated to star-shaped splice diagrams occur as  links of normal surface singularities. 
The notion of edge determinant, introduced formally by Neumann and Wahl 
in \cite[Section 1]{NW 05} (although it appears already in~\cite[page 82]{EN 85}), 
 allows to characterize which integral homology spheres may be realized 
 as such links (see \autoref{thm:charZHSlinks}):

\begin{definition}   \label{def:edgedet}
   Let $\Gamma$ be a splice diagram. If $u$ and $v$ are two adjacent nodes of $\Gamma$,    
   then the \textbf{edge determinant}\index{edge!determinant} of the edge $[u,v]$ is the number 
   obtained by subtracting from the 
   product of the two decorations on $[u,v]$  the product of the remaining 
   decorations in the neighborhoods of $u$ and $v$. 
   We say that $\Gamma$ 
   \textbf{satisfies the edge determinant condition}\index{condition!edge determinant} 
   if the edge determinant of every internal edge of $\Gamma$ is positive. 
\end{definition}

We illustrate this definition with the running example from~\cite[Section 1]{NW 05}:
\begin{figure}[tb]  
         \includegraphics[scale=0.75]{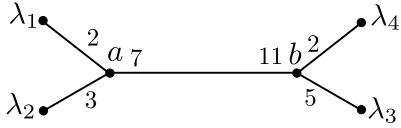}
         \caption{The splice diagram used in 
           {Examples}~\ref{ex:basicdetcond}  and~\ref{ex:semigpsatisf}.
             \label{fig:simpleExampleOverview}
}               
  \end{figure}

  \begin{example}   \label{ex:basicdetcond}
      Consider the splice diagram with two nodes and four leaves seen in 
    ~\autoref{fig:simpleExampleOverview}. It satisfies the edge determinant condition, as 
     the edge determinant of its single internal edge is $7 \cdot 11 - (2 \cdot 3) \cdot (5 \cdot 2) = 11 > 0$. 
 \end{example}

In~\cite[Theorem 9.4]{EN 85}, Eisenbud and Neumann  
gave an explicit description of all integral homology spheres that can be realized as surface 
singularity links. Here is the precise statement:

\begin{theorem}   \label{thm:charZHSlinks}
    The links of normal surface singularities which are $\ZHS$'s     
    are precisely the oriented three-manifolds $\Sigma(\Gamma)$ associated     
    to  splice diagrams which satisfy the edge determinant condition. Moreover, 
    the diagram $\Gamma$ is completely determined by the link if for every non-internal 
    edge $e$ of $\Gamma$ joining a node $v$ to a leaf $\lambda$ we have $d_{v,e}\geq 2$.
\end{theorem}

Note that~\autoref{thm:hammbrieslinks}  shows that the Seifert fibered  $\ZHS$'s 
are not only surface singularity links, but they occur
as links of isolated complete intersection singularities. This observation leads to the following natural analog of~\autoref{q:whichGraphsAreLinks} stated in~\autoref{sec:introd}:

\begin{question}\label{q:qhichZHSAreLinks}
         Which integral homology spheres of the form $\Sigma(\Gamma)$ can occur as links 
         of isolated complete intersection singularities?
\end{question}

Although a complete answer to this question remains unknown, a partial answer was given by Neumann and Wahl in~\cite{NW 05bis}. Indeed, they showed that $\Sigma(\Gamma)$ is the link  of an isolated complete intersection singularity whenever $\Gamma$ satisfies a supplementary hypothesis called the \emph{semigroup condition}, which we now recall. We start with the auxiliary notions of 
{\em linking number between two vertices}  and of {\em degree of a node} of $\Gamma$:

\begin{definition}  \label{def:linkingNumberDegree}  
      Let $\Gamma$ be a splice diagram. 
      For every pair of vertices $u$ and $v$ of $\Gamma$,   
      their \textbf{linking number}\index{linking number} $\boxed{\wtuv{u}{v}}$ 
      is the product of edge weights 
      adjacent to the shortest path $[u, v]$ on $\Gamma$ joining $u$ and $v$. In particular, 
      the \textbf{degree}\index{node!degree of}  of a node $v$ is the product 
      $\boxed{\du{v}}:=\wtuv{v}{v}$ of all edge decorations adjacent to $v$. 
\end{definition}

\begin{remark}\label{rem:whyLinkingName}   
The name \emph{linking number} used for the integers $\wtuv{u}{v}$ is motivated by the fact 
  that they agree with  the linking numbers (inside the integral homology sphere 
  $\Sigma(\Gamma)$ from~\autoref{def:splicediagZHS}) of the knots corresponding to the generic fibers of the 
  Seifert fibered manifolds $\Sigma(\Gamma^u)$ and $\Sigma(\Gamma^v)$. For more details, we refer to~\cite[Theorem 10.1]{EN 85}. Note that the edge determinant of an internal edge $[u,v]$ is positive if and only if $\du{u} \du{v} > \wtuv{u}{v}^2$.
\end{remark}

\begin{definition}      \label{def:semgpcond} 
   Let $\Gamma$ be a splice diagram. Fix a node $v$ and 
  an edge $e$ of $\Gamma$ adjacent to it.
   We say that $\Gamma$ \textbf{satisfies the semigroup condition at $v$ in the direction of $e$} if 
     $\du{v}$ belongs to the subsemigroup of $(\N, +)$ generated by the positive integers     
      $\wtuv{v}{\lambda}$, where $\lambda$ varies among the leaves of $\Gamma$ 
       seen from $v$ in the direction of $e$ (i.e., such that $e$ lies in the shortest 
      path $[v,\lambda]$).  
      If this condition is verified for all pairs $(v,e)$, then we say that $\Gamma$ 
      \textbf{satisfies the semigroup condition}\index{condition!semigroup}. 
\end{definition}

\begin{remark}\label{rem:homogeneity}
The semigroup condition is essential to extend the construction of Brieskorn-Pham-Hamm systems from star-shaped diagrams to arbitrary ones: it specifies how to replace the power of a single variable (indexed by the corresponding leaf) by a monomial in the variables indexed by leaves seen from $v$ in the direction of $e$. Polynomials constructed in this way will be  homogeneous relative to suitable 
weight vectors, described in~\autoref{rm:HomogConditions}~(\ref{item:whomog}) below. 
\end{remark}

\begin{example}\label{ex:semigpsatisf}
  As in~\autoref{ex:basicdetcond}, we consider the splice diagram $\Gamma$ from~\autoref{fig:simpleExampleOverview}. Note that 
  $\du{a} = 2 \cdot 3 \cdot 7 = 42$, $\wtuv{a}{\lambda_3}= 2 \cdot 3 \cdot 2=12$, 
$\wtuv{a}{\lambda_4} = 2 \cdot 3 \cdot 5=30$,  
$\du{b} = 2 \cdot 5 \cdot 11=110$, $\wtuv{b}{\lambda_1}= 2 \cdot 5 \cdot 3=30$,  and
$\wtuv{b}{\lambda_2} = 2 \cdot 5 \cdot 2=20$. Therefore, $\Gamma$ satisfies the  semigroup condition, as $\du{a} \in \N\langle \wtuv{a}{\lambda_3}, \wtuv{a}{\lambda_4}\rangle$ 
and $\du{b} \in \N\langle \wtuv{b}{\lambda_1}, \wtuv{b}{\lambda_2}\rangle$. More precisely, 
\begin{equation}\label{eq:ExampleSemigroupID}           
        \du{a} =  42 =  12 + 30 = \wtuv{a}{\lambda_3} +  \wtuv{a}{\lambda_4}  \quad \text{ and } \quad 
           \du{b} = 110 =  30 + 4 \cdot 20 = 
           \wtuv{b}{\lambda_1} + 4 \, \wtuv{b}{\lambda_2}. \qedhere
\end{equation}
Note that the semigroup condition is always satisfied at a node $v$ in the direction of an edge 
   joining $v$ to a leaf. For instance, $ \du{a} =  42 = 2 \cdot (3 \cdot 7) = 
      2 \cdot \wtuv{a}{\lambda_1} \in \N\langle \wtuv{a}{\lambda_1}\rangle$. 
  \end{example}

The following result is due to Neumann and Wahl   
(as a consequence of~\cite[Theorems 2.6 and 7.2]{NW 05}):   

\begin{theorem}   \label{thm:semgrpsplice}
     Let $\Gamma$ be a splice diagram with $n$ leaves which satisfies the determinant 
     and  semigroup conditions. 
    Then, there exists an isolated complete intersection singularity 
    embedded in $\CC^n$    whose oriented link is orientation-preserving homeomorphic to  $\Sigma(\Gamma)$. 
\end{theorem}

In fact, Neumann and Wahl's result referenced above is more general, since it concerns splice diagrams whose edge weights around vertices are not necessarily pairwise coprime.  
The reader interested in learning more about them and the 
associated \emph{splice quotient singularities} may consult Wahl's  
surveys~\cite{W 06, W 22}. 

The proof of~\autoref{thm:semgrpsplice} is constructive. Indeed, given any splice diagram $\Gamma$ 
satisfying the semigroup condition, Neumann and Wahl build a family of systems of formal power series in $n$ variables  which define equisingular 
isolated complete intersection singularities with link $\Sigma(\Gamma)$ 
(see \autoref{thm:splicesingzhs}).
The explicit construction of such {splice type systems} and 
the associated {splice type singularities} will be discussed in~\autoref{ssec:spltypesing} below. 
The largest  class known up to date of complete intersection isolated surface singularities with integral homology sphere links remains that of splice type.

\medskip
\subsection{Splice type singularities}   
\label{ssec:spltypesing}
$\:$ 
\medskip

In this subsection we recall Neumann and Wahl's construction
of \emph{splice type systems}~\cite{NW 05,NW 05bis} associated to 
splice diagrams satisfying both the determinant and semigroup conditions 
(see {Definitions}~\ref{def:splicediag},~\ref{def:edgedet} and~\ref{def:semgpcond}). 
Such systems define the so-called \emph{splice type singularities} 
(see \autoref{def:splicesystem}). For a description of how Neumann and Wahl were led to this
construction, we refer the reader to Wahl's paper \cite{W 22}. 
\medskip 

Let $\boxed{\nodesT{\Gamma}}$ be the set of nodes of 
the splice diagram $\Gamma$ and $\boxed{\leavesT{\Gamma}}$ be its set of leaves. 
We denote by $\boxed{n}$ the number of leaves of $\Gamma$. 
Following~\autoref{def:splicediagZHS} we let $\boxed{\Gamma^{v}}$ be the \textbf{star} 
of a vertex $v$ of $\Gamma$, i.e. the collection of all edges adjacent to $v$, 
with inherited weights around $v$. It contains precisely $\valv{v}$ edges, 
i.e., as many as the valency of $v$.
In addition to the notion of linking number between pairs of vertices  
introduced in~\autoref{def:linkingNumberDegree}, it will often be convenient to work 
with the following related notion, first introduced in \cite[Section 1]{NW 05}:

\begin{definition}  
   The \textbf{reduced linking number}\index{linking number!reduced} $\boxed{\wtuv{v}{u}'}$ 
    is defined as the product of all weights adjacent to the path $[u,v]$ excluding those 
    around $u$ and $v$. In particular, $\wtuv{v}{v}' =1$ for each node $v$ of $\Gamma$.
\end{definition}

  \begin{remark}   \label{rem:reducedlink}
        Given a node $v$ and a leaf $\lambda$ of $\Gamma$, it is immediate to check that 
        $\wtuv{v}{\lambda} \,{\du{v,\lambda}} = \wtuv{v}{\lambda}'\,\du{v}$. This implies that 
        the semigroup condition from~\autoref{def:semgpcond} for the pair $(v,e)$ 
        is satisfied if, and only if,  $\du{v,e}$ belongs to the subsemigroup of $(\N, +)$ generated 
        by $\wtuv{v}{\lambda}'$, where $\lambda$ varies among the leaves of $\Gamma$ 
        which are seen from $v$ in the direction of $e$. 
  \end{remark}

In what follows, we recall some standard notations from toric geometry. 
They are not required to  define splice type singularities 
(and were not used in the foundational papers of  Neumann and Wahl)
but they are essential for our proof of the Milnor fiber conjecture.   

Each leaf $\lambda$ of $\Gamma$ yields a variable $\boxed{z_{\lambda}}$. Let $\boxed{\Mwp{\leavesT{\Gamma}}}$ be the  
lattice of exponent vectors of monomials in the variables $z_{\lambda}$. 
We denote by $\boxed{\Nwp{\leavesT{\Gamma}}}$ its dual lattice of weight vectors  
of the variables $z_{\lambda}$. We write the associated pairing using dot product notation, 
i.e. $\boxed{\wu{}\cdot m} \in \Z$ whenever $\wu{} \in \Nwp{\leavesT{\Gamma}}$ and $m\in \Mwp{\leavesT{\Gamma}}$. The canonical basis $\{\boxed{\wu{\lambda}}: \lambda \in \leavesT{\Gamma}\}$ 
of $\Nwp{\leavesT{\Gamma}}$ and the dual basis $\{\boxed{m_{\lambda}}: \lambda \in \leavesT{\Gamma}\}$ 
of  $\Mwp{\leavesT{\Gamma}}$ identify both lattices with $\Z^{n}$.   
Each node $v$ of $\Gamma$ has an associated  weight vector
    \begin{equation}\label{eq:wu}
          \boxed{\wu{v}} :=\sum_{\lambda \in \leavesT{\Gamma}} \wtuv{v}{\lambda} \, \wu{\lambda} 
         \in     \Nwp{\leavesT{\Gamma}}. 
     \end{equation}

If $v$ is a node of $\Gamma$ and $e \in \Gamma^{v}$, we 
denote by $\boxed{\nodesTevRoot{v,e}}$ the set of leaves $\lambda$ 
of $\Gamma$ seen from $v$ in the direction of $e$. 
The diagram  $\Gamma$ satisfies the semigroup condition if, and only if, for each node $v$, edge $e\in \Gamma^v$ and leaf $\lambda \in \nodesTevRoot{v,e}$, 
  there exists  $\boxed{\wtNveL{v}{e}{\lambda}}\in \N$ such that: 
\begin{equation}\label{eq:PowersadmisibleMon}
  \du{v} = \sum_{\lambda \in \nodesTevRoot{v,e}} \wtNveL{v}{e}{\lambda} \, \wtuv{v}{\lambda}\,, \quad \text{ or equivalently } \quad \du{v,e} = \sum_{\lambda \in \nodesTevRoot{v,e}} \wtNveL{v}{e}{\lambda} \, \wtuv{v}{\lambda}'.
\end{equation}
This last equivalence is a direct consequence of~\autoref{rem:reducedlink}.

We use the coefficients from~\eqref{eq:PowersadmisibleMon} to define an  element of $\Mwp{\leavesT{\Gamma}}$ for each pair $(v,e)$:
\begin{equation}\label{eq:admissibleMonExp}
     \boxed{\wtNve{v}{e}} := \sum_{\lambda \in \nodesTevRoot{v,e}}
    \wtNveL{v}{e}{\lambda} \,m_{\lambda} \in  \Mwp{\nodesTevRoot{v,e}}\subset \Mwp{\leavesT{\Gamma}}.
\end{equation}
Following~\cite{NW 05}, we refer to it as an 
\textbf{admissible exponent vector}\index{admissible!exponent vector} 
for $(v,e)$. By~\eqref{eq:PowersadmisibleMon}, it  satisfies
\begin{equation}\label{eq:admMonwuvalue} 
       \wu{v}\cdot \wtNve{v}{e} = \du{v} \quad \text{ for each edge }\; e \in \Gamma^{v}.
\end{equation}
In turn, each admissible exponent vector  $\wtNve{v}{e}$ defines an 
\textbf{admissible monomial}\index{admissible!monomial}:  
\begin{equation}\label{eq:admissibleMon}
      \boxed{\zexp{\wtNve{v}{e}}} := \prod_{\lambda \in \nodesTevRoot{v,e}}
      z_{\lambda}^{\wtNveL{v}{e}{\lambda}}.
\end{equation}

The next definition is a reformulation of a notion introduced by Neumann 
and Wahl   in \cite[Section 2]{NW 05}:

\begin{definition}   \label{def:splicesystem}
      Let $\Gamma$ be a splice diagram which satisfies both the determinant and 
       semigroup conditions, and assume that the set of $n$ leaves of $\Gamma$ is totally ordered. 
       For each node $v$ and adjacent edge $e$ of it, fix an admissible exponent vector 
       $\wtNve{v}{e}\in \Mw{\Gamma} $ defined in~\eqref{eq:admissibleMonExp}. 
        \begin{itemize}
             \item A \textbf{strict splice type system}\index{splice type!system!strict} 
                  for $\Gamma$ is a finite family  
                 of $(n-2)$ polynomials of the form
                \begin{equation}\label{eq:surface}
                         \boxed{ \fvi{v}{i}(\zu)}   :=\sum_{e\in {\starT{
                               }{v}}} \!\!\!\!    \boxed{\cvei{v}{e}{i}}  \; \zexp{\wtNve{v}{e}} \quad \text{for all }  
                                   i\in \{1,\dots, \valv{v}-2\} \, \text{ and } \text{ each node } v\text{ of }\Gamma.
                 \end{equation}                  
                    We  require the coefficients $\cvei{v}{e}{i}$ to  satisfy the 
                    \textbf{Hamm determinant condition}\index{condition!Hamm determinant}.     
                   Namely, for any node $v\in \Gamma$, and any fixed ordering 
                   of the  edges in $\Gamma^{v}$,  all the maximal minors of the matrix 
                   of coefficients $(\cvei{v}{e}{i})_{e,i} \in \CC^{\valv{v} \times (\valv{v}-2)}$ 
                   must be non-zero.

             \item A \textbf{splice type system}\index{splice type!system} $ \boxed{\sG{\Gamma}} $  
             associated to $\Gamma$ is a finite family  of power series  of the form 
                \begin{equation}\label{eq:surfaceSeries}
                     \boxed{ \fgvi{v}{i}(\zu)}   := \fvi{v}{i}(\zu)  + \gvi{v}{i}(\zu)
                       \quad \text{for all }   i\in \{1,\dots, \valv{v}-2\} \, \text{ and any fixed  node } 
                       v\text{ of }\Gamma,
                 \end{equation}
                where the collection $(\fvi{v}{i})_{v,i}$ is a strict splice type system 
                for $\Gamma$ and each $\gvi{v}{i}$ is a convergent power series 
                near the origin satisfying the following condition for each exponent vector $m$ 
                in its support:
                       \begin{equation}\label{eq:gviConditions}
                              \wu{v}\cdot m > \du{v}. 
                        \end{equation}
       \item A \textbf{splice type singularity}\index{splice type!singularity} 
          associated to  $\Gamma$ is the subgerm of 
          $(\CC^n, 0)$ defined  by a splice type system $\sG{\Gamma}$.   
     \end{itemize}  
\end{definition}

\begin{remark}\label{rm:HomogConditions} 
     The following observations regarding~\autoref{def:splicesystem} are in order:
\begin{enumerate}
   \item   \label{item:whomog}
        By equations \eqref{eq:admMonwuvalue} and \eqref{eq:gviConditions}, 
        each polynomial $\fvi{v}{i}$ is 
       $\wu{v}$-homogeneous, where 
       $\wu{v}$ is the weight vector from \eqref{eq:wu}, and each monomial appearing 
       in $\gvi{v}{i}$ has higher $\wu{v}$-weight.   
      
   \item  The first appearance of splice type systems can be traced back to~\cite{NW 02}. 
        In that paper, the edge weights around nodes were not assumed to be pairwise coprime, 
        but the edge determinant and semigroup conditions were still required. 
       Neumann and Wahl proved that under a supplementary condition 
       (called the \emph{congruence condition}), 
      it is possible to pick the series $\fgvi{v}{i}(\zu)$ in an equivariant way  under the action of certain
       finite abelian groups. This construction then leads to defining \emph{splice quotient singularities} 
       as the quotients 
       of the associated splice type singularities by those abelian groups. 
       These singularities and their defining systems are studied thoroughly in \cite{NW 05bis} 
       (see also Wahl's surveys~\cite{W 06, W 22}).

   \item Neumann and Wahl proved in \cite{NW 05bis} that the set of splice-type subgerms 
      of $\mathbb{C}^n$ corresponding to a given splice diagram satisfying the determinant 
      and the semigroup condition is independent of the choice of admissible exponents. 
      For a detailed proof, we refer the reader to~\cite[Theorem 9.1]{CPS 21}. 
   \end{enumerate}
          \end{remark}

The following two examples illustrate~\autoref{def:splicesystem}:

\begin{example}\label{ex:splicesystbasic}
     Consider the splice diagram from~\autoref{fig:simpleExampleOverview}. As shown 
     in   {Examples}~\ref{ex:basicdetcond} and~\ref{ex:semigpsatisf}, 
     $\Gamma$ satisfies the determinant 
     and semigroup conditions. The explicit semigroup membership identities 
     from~\eqref{eq:ExampleSemigroupID} yield the following associated strict splice type system:
         \begin{equation*}\label{eq:numExampleSurface1}
               f_a:= z_1^2 - z_2^3 + z_3 z_4 \quad  \text{ and } \quad  
              f_b:=    z_1 z_2^4 + z_3^5 - z_4^2.\qedhere
         \end{equation*}
      Another possible choice for the $\wu{b}$-homogeneous function $f_b$ is 
      $z_1^3 z_2 + z_3^5 - z_4^2$, obtained by replacing the admissible monomial  
      $z_1 z_2^4$ by the other possible admissible monomial $z_1^3 z_2$ for $(b, [b, a])$.  
      This second monomial is admissible because 
      $\du{b} = 110 =  3 \cdot 30 + 1 \cdot 20 = 
           3 \cdot \wtuv{b}{\lambda_1} + 1 \cdot \wtuv{b}{\lambda_2}$). 
 \end{example}

 \begin{figure}[tb]
    \includegraphics[scale=0.75]{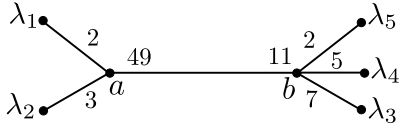}
        \caption{The splice diagram of Examples \ref{ex:2NodesNW} and \ref{ex:2NodesNW3fold}. 
         \label{fig:numExampleSimpleEndCurve} }
  \end{figure}

\begin{example}\label{ex:2NodesNW}
  Consider the splice diagram  $\Gamma$ from~\autoref{fig:numExampleSimpleEndCurve} 
  with nodes $a$ and $b$.
  Then, $\du{a} = 294$, $\du{b}= 770$ and $\wtuv{a}{b} = 420$, so $\du{a}\,\du{b}>\wtuv{a}{b}^2$. 
  Therefore, the edge determinant 
  condition holds for $[a,b]$ by~\autoref{rem:whyLinkingName}. Furthermore, 
      \begin{equation*}
               49 =   0 \cdot   (2 \cdot 5) +  1 \cdot  ( 2 \cdot 7)   +  1 \cdot (5 \cdot 7) \quad \text{ and } \quad
               11  =   1  \cdot (3) + 4 \cdot  (2) = 3  \cdot (3) + 1 \cdot  (2),
      \end{equation*}
so  the semigroup condition is also satisfied. Associated admissible exponent vectors 
  are $\wtNve{a}{[a,b]} = (0,0,0,1,1)$ and  $\wtNve{b}{[a,b]} = (1,4,0,0,0)$ or $(3,1,0,0,0)$. The following polynomials determine a strict splice type system for $\Gamma$:
  \begin{equation*}\label{eq:numExampleSurface}
    \begin{cases}  
         \fvi{a}{1}:= \;\;\;\, z_1^{2}\;\; \;-\;\;\;2\; z_2^3\; +\;\;\; z_4\,z_5,  \\
         \fvi{b}{1}:=\;\;\;\; z_1z_2^4 +  z_3^7 + \;\; z_4^5\, - \;2155\; z_5^2,    \\
         \fvi{b}{2}:=  33\, z_1z_2^4 +  z_3^7 + 2\,z_4^5 - \; 2123 \, z_5^2. \,  
    \end{cases}
  \end{equation*}
 An alternative system is obtained by replacing  
 the admissible monomial $z_1z_2^4$ in $\fvi{b}{1}$ and $\fvi{b}{2}$ with $z_1^3z_2$.
\end{example}

   \begin{remark}   \label{rem:tropNND2D}
    Our paper \cite{CPS 21} describes standard tropicalizing 
    fans of splice type singularities 
    in the sense of~\autoref{def:tropfans}, and shows that splice type systems are 
    Newton non-degenerate complete intersection 
    presentations of them in the sense of~\autoref{def:NNDcomplint}. These results 
    are essential tools to prove analogous facts for their edge-deformations, introduced 
    in~\autoref{sec:edgedeform}  (see also~\autoref{rem:tropNND3D}). 
    The weight vectors $(\wu{v})_v$ indexed by the nodes of $\Gamma$
    generate the positive rays of the 
    standard tropicalizing fan of the splice type singularity. Its remaining rays 
    are generated by the basis vectors 
    $(\wu{\lambda})_{\lambda}$ of $\Nwp{\leavesT{\Gamma}}$. 
    Moreover, {\em the associated splice diagram   
    appears as a transversal section of the 
    local tropicalization of a splice type singularity} (see \cite[Theorem 1.2]{CPS 21}). 
    This gives the first tropical interpretation 
    of splice diagrams, whenever they  satisfy the determinant and 
    semigroup conditions. 
    We make use of this fact in Step~(\ref{item:tropfanX}) of~\autoref{sec:stepsproof}. 
  \end{remark}

  \begin{remark}   \label{rem:irredinters}
      Let $(X,o)$ be a splice type singularity associated to the splice diagram $\Gamma$. 
      For each leaf $\lambda$ of $\Gamma$, one may consider the hyperplane section of $(X,o)$ by 
      the hyperplane of coordinates defined by $z_{\lambda} =0$. As a particular case of 
      Neumann and Wahl's theorem \cite[Theorem 7.2 (6)]{NW 05bis}, this hyperplane section 
      is an irreducible germ of curve, therefore its associated link is a knot inside the link 
      of $(X,o)$. This fact will be used in \autoref{def:fourdimsplice}.
  \end{remark}

In \cite[Theorems 2.6 and 7.2]{NW 05}, Neumann  and 
Wahl prove the following explicit form of \autoref{thm:semgrpsplice}:

\begin{theorem} \label{thm:splicesingzhs}
      Let $\Gamma$ be a splice diagram which satisfies both the determinant and the 
      semigroup conditions. Then the link of any splice type singularity associated to $\Gamma$ 
      is orientation-preserving homeomorphic to $\Sigma(\Gamma)$. 
\end{theorem}

\medskip
\subsection{Neumann and Wahl's Milnor fiber conjecture}  \label{ssec:mfconj}$\:$ 
\medskip

In this subsection we explain Neumann and Wahl's {\em four-dimensional splicing operation} 
(see \autoref{def:fourdimsplice}) 
and we give a more precise formulation of the {\em Milnor fiber conjecture} that the one 
given in the Introduction (see \autoref{conj:MFC}). 
We conclude by stating a corollary of our proof of this conjecture (see \autoref{thm:invMF}).
\medskip

Throughout this subsection, we fix  a splice diagram  $\Gamma$ with $n$ leaves satisfying the determinant and the semigroup conditions (see~ 
{Definitions}~\ref{def:edgedet} and \ref{def:semgpcond}). Furthermore, we assume that $\Gamma$ is not star-shaped and we fix two adjacent nodes $\boxed{a,b}$ of it. As illustrated in~\autoref{fig:splittingByEdge}, we let $\boxed{\Gamma_a}$ and $\boxed{\Gamma_b}$  be the splice diagrams  obtained by cutting $\Gamma$ at an interior point $\boxed{\roottree}$ of $[a,b]$. 
Denote by $\boxed{\roottree_a} \in \Gamma_a$ and $\boxed{\roottree_b} \in \Gamma_b$ 
the corresponding leaves of $\Gamma_a$ and $\Gamma_b$. We view them as roots of the 
two trees. 
We let $\boxed{n_a}$ and $\boxed{n_b}$ be the number of leaves of $\Gamma_a$ and $\Gamma_b$, respectively. Therefore, $n=n_a+n_b - 2$.

\begin{figure}[tb]
  \includegraphics[scale=0.45]{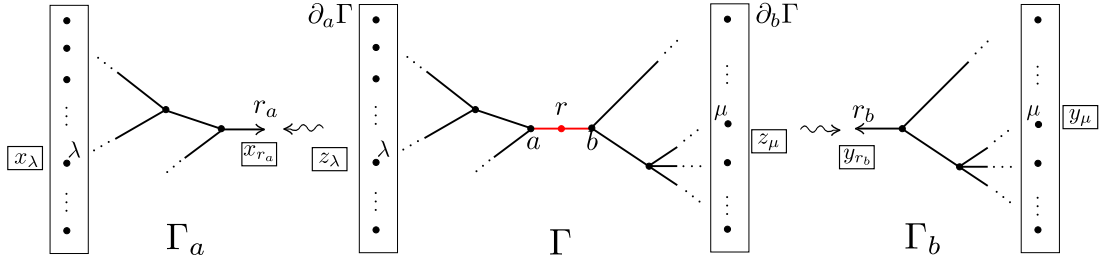}
  \caption{Splitting  the splice diagram $\Gamma$ along any interior point $\roottree$ on the central edge $[a,b]$ yields the diagrams $\Gamma_a$ and $\Gamma_b$, with roots $\roottree_a$ and $\roottree_b$, respectively. The  variables associated to the leaves on each diagram are labeled from left to right by  $x_{\lambda}$, $z_{\lambda}$, $z_{\mu}$ and $y_{\mu}$, respectively. \label{fig:splittingByEdge}
  }
  \end{figure}

It is a simple matter to check that $\Gamma_a$ and $\Gamma_b$ also satisfy the determinant and 
semigroup conditions. Thus, we may use the three splice diagrams $\Gamma$, $\Gamma_a$ 
and $\Gamma_b$ to  build three splice type systems. We let  $\boxed{X}$, $\boxed{X_a}$ and 
$\boxed{X_b}$ be the germs at the origin defined by each system in $\CC^n$, $\CC^{n_a}$ and 
$\CC^{n_b}$, respectively.  We denote by $x_{\lambda}$ the variables of the ambient space 
$\CC^{n_a}$ of $X_a$ and by $y_{\mu}$ those of the ambient space $\CC^{n_b}$ of $X_b$, 
where $\lambda$ varies in the set $\leavesT{\Gamma_a}$ of leaves of $\Gamma_a$ and 
$\mu$ varies in the set $\leavesT{\Gamma_b}$ of leaves of $\Gamma_b$. In particular, 
there are two variables, $x_{r_a}$ and $y_{r_b}$, which correspond to the roots of the two trees.

Since the germs $X$, $X_a$, and $X_b$ are isolated complete intersections, they  have well-defined Milnor fibers $\boxed{F}$, $\boxed{F_a}$, $\boxed{F_b}$, which are compact oriented four-dimensional manifolds with boundary. Furthermore, their boundaries are orientation-preserving diffeomorphic to the links of the associated singularities.  The Milnor fiber conjecture of Neumann and Wahl describes a concrete topological operation to build $F$ from $F_a$ and $F_b$. In what follows, we review this construction.

Consider the restriction of the coordinate function $x_{\roottree_a}$ to $X_a$. 
This holomorphic function has an isolated critical point at $0 \in X_a$. 
Therefore, it defines an {\em open book}\index{open book} (a terminology introduced by Winkelnkemper 
\cite{W 73}, also called an {\em open book decomposition}) 
on the link $\partial(X_a, 0)$ of $(X_a, 0)$: it is the 
{\em Milnor open book}\index{Milnor!open book} induced 
by the argument of the holomorphic function (see \cite[Section 6.5]{LNS 20}). Since 
the link $\partial(X_a, 0)$ is diffeomorphic to the boundary of the Milnor fiber $F_a$, 
we obtain an open book on this boundary~\cite{H 71}. Denote by $\boxed{G_a} \hookrightarrow F_a$ a compact surface with boundary obtained by pushing a page of this open book inside $F_a$, while keeping the boundary fixed. \autoref{fig:openBook} depicts this construction in lower dimension. 
Note that the boundary of $G_a$ is connected, because the hyperplane sections of $X_a$ 
by coordinate hyperplanes are irreducible (see \autoref{rem:irredinters}). 

We let $N_{F_a}(G_a)$ be a tubular neighborhood of $G_a$ in $F_a$. 
Consider the associated classical cut $\cC_{G_a} F_a$ of $F_a$ along $G_a$, 
as in~\autoref{def:cutmanif}.  Note that the normal bundle of  $G_a$ inside $F_a$ 
is trivial because it is a disk bundle over a connected surface with non-empty boundary. 
Therefore, the tubular 
neighborhood $N_{F_a}(G_a)$ is diffeomorphic to $G_a \times \D$, where $\D$ denotes a compact 
two-dimensional disk. This implies that the 
\textbf{longitudinal boundary} of $N_{F_a}(G_a) $, which we  define as 
\begin{equation}\label{eqhorizontalBoundary}
\boxed{\partial_{long}N_{F_a}(G_a)} := \cC_{G_a} F_a \cap N_{F_a}(G_a) 
\hookrightarrow \partial_{top}( \cC_{G_a} F_a)
\end{equation}
is diffeomorphic to $G_a \times \bS^1$. 

\begin{figure}[tb]
  \includegraphics[scale=0.5]{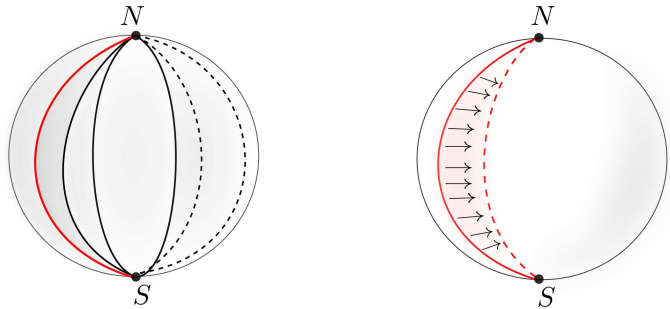}
  \caption{From left to right: collection of meridians forming an open book on the sphere $\bS^2$ with a distinguished page (in red), and pushing of this  page inside the interior of the ball 
  bounded by $\bS^2$, featured as a dashed arc. The binding is given by the north and south poles.\label{fig:openBook}}
\end{figure}

The next definition recalls
Neumann and Wahl's {\em four-dimensional splicing operation}  
in this context (see~\cite[Section 6]{NW 05} for further details). 
The construction is depicted in~\autoref{fig:4DSplicing}. 

 \begin{figure}[tb]
   \includegraphics[scale=0.7]{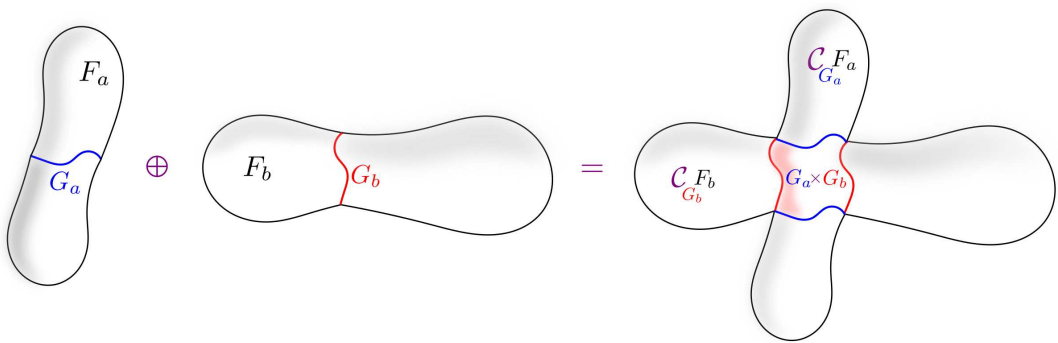}
   \caption{Splicing of two four-dimensional manifolds $F_a$ and $F_b$ with integral 
        homology sphere boundaries along properly embedded surfaces called $G_a$
        and $G_b$, respectively (see \autoref{def:fourdimsplice}).\label{fig:4DSplicing}}
 \end{figure}

 \begin{definition}   \label{def:fourdimsplice}\index{splicing!four-dimensional}
     Let $(F_a, G_a)$ and $(F_b, G_b)$ be the pairs defined above. 
     The \textbf{manifold 
     $\boxed{(F_a, G_a) \oplus (F_b, G_b)}$ obtained by splicing $F_a$ and $F_b$ 
     along $G_a$ and $G_b$} is constructed from the disjoint union 
        \[\cC_{G_a} F_a \:  \:  \sqcup  \:  \:  (G_a \times G_b)  \:  \:   \sqcup   \:  \:  \cC_{G_b} F_b\]
     by identifying $G_a \times \bS^1 \simeq \partial_{long}N_{F_a}(G_a) \hookrightarrow \cC_{G_a} F_a$ 
     with $G_a \times \bS^1 \simeq G_a \times \partial_{top} G_b \hookrightarrow G_a \times G_b$ 
     and, similarly, 
     $G_b \times \bS^1 \simeq \partial_{long}N_{F_b}(G_b) \hookrightarrow \cC_{G_b} F_b$ 
     with $ \bS^1 \times G_b \simeq \partial_{top}G_a \times G_b \hookrightarrow G_a \times G_b$.
 \end{definition}

 A basic, yet crucial, property of this operation is that it induces the 
 $3$-dimensional splicing operation of~\autoref{def:splicing} at the level of boundaries. As no proof for this fact was given in~\cite{NW 05}, we include one below.

 \begin{proposition}\label{pr:property4dSplicing}
     Let $(F_a, G_a)$ and $(F_b, G_b)$ be pairs as above. 
     Then, the boundary of the manifold obtained by splicing $F_a$ and $F_b$ 
     along $G_a$ and $G_b$ is the three-dimensional manifold obtained 
     by splicing their boundaries. More precisely, we have an orientation 
     preserving diffeomorphism:
        \[  \partial_{top} \left(  (F_a, G_a) \oplus (F_b, G_b) \right) \simeq 
                (\partial_{top}  F_a, \partial_{top}  G_a) \oplus (\partial_{top}  F_b, \partial_{top}  G_b).      \]
 \end{proposition}
 
 \begin{proof}  
     Let $(F, G)$ be one of the pairs $(F_a, G_a)$ and $(F_b, G_b)$. The next  
   reasoning is to be followed along using~\autoref{fig:ProffOfProp4d}, suggestive of 
   an analogous situation in one dimension lower.

   As explained above, the  boundary $K$ of the surface $G$ is a knot in $\partial_{top} F$, 
   because the hyperplane sections of splice type singularities 
   by coordinate hyperplanes are irreducible. 
   We must show that the trivialization of the circle bundle 
     $ \partial_{top} N_{\partial_{top} F} (K) \to K$ induced by the chosen trivialization 
   of the circle bundle $ \partial_{long} N_F(G) \to G$ coincides up to isotopy with 
   the trivialization described in~\autoref{prop:trivbd}. Thus, we must check that the  boundary 
   of a constant section of $ \partial_{long} N_F(G) \to G$ relative to this 
   trivialization has linking number $0$ with $K$ inside $\partial_{top} F$.

  Consider a page 
   $G_K \hookrightarrow \partial_{top} F$ of the given open book on $\partial_{top} F$ with 
   binding $K$. As $G$ is obtained by pushing $G_K$ inside $F$ while preserving its boundary, 
   $G \cup G_K$ is  the boundary of an oriented compact three-manifold $M$ diffeomorphic 
   to a handelbody and 
   embedded in $F$. We choose the tubular neighborhood $N_{\partial_{top} F} (K)$ to be transversal 
   to $M$. Therefore, the intersection $\partial_{top} N_{\partial_{top} F} (K) \  \cap \ M$ is 
   a section of the circle bundle $ \partial_{top} N_{\partial_{top} F} (K) \to K$. Slightly turning 
   this intersection inside each fiber yields another section $G'$ which is disjoint from $M$. 
   Therefore, its boundary 
   $K' := G' \ \cap \  \partial_{top} F \hookrightarrow \partial_{top} F$ is disjoint from 
   $G_K = M \ \cap \ \partial_{top} F$. This implies that $lk_{\partial_{top} F}(K', K) =0$, 
   as we wanted to show.
 \end{proof}

 \begin{figure}[tb]
   \includegraphics[scale=0.5]{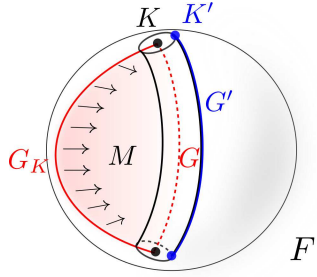}
   \caption{A tubular neighborhood of $G$ inside $F$, and copies $G'$ (in blue) 
   and $G_{K}$ (in red)  of $G$  (in dashed red) with $G_K\subset \partial_{top} F$ 
   satisfying $\partial_{top} G= \partial_{top} G_K=K$ and $\partial_{top} G' = K'$. 
   The pink shaded area $M$
   bounded by $G_K$ and $G$ avoids $G'$ 
   (see the proof of \autoref{pr:property4dSplicing}). \label{fig:ProffOfProp4d}}
 \end{figure}

\autoref{def:fourdimsplice}  allows us to present a more precise version of 
Neumann and Wahl's {\em Milnor fiber conjecture} 
of \cite[Section 6]{NW 05} than the one given in~\autoref{sec:introd}:

\begin{conjecture}   \label{conj:MFC}\index{conjecture!Milnor fiber}
    Let $X$ be a splice type singularity whose splice diagram $\Gamma$ is not star-shaped. 
    Fix an internal edge $[a,b]$ of $\Gamma$. 
    Let $\Gamma_a$ and $\Gamma_b$ be the rooted splice diagrams obtained by cutting 
    $\Gamma$ at an interior point of $[a,b]$. 
    Denote by $X_a$ and $X_b$ the  splice type singularities associated to $\Gamma_a$ 
    and $\Gamma_b$. Let $F$, $F_a$ and $F_b$ be Milnor 
    fibers of $X$, $X_a$ and $X_b$, respectively. Consider a surface 
    $G_a \hookrightarrow F_a$ obtained as above 
    from the  open book defined on $\partial (X_a, 0)$ by the variable associated to the 
    root of $\Gamma_a$. Consider an analogous surface $G_b \hookrightarrow F_b$. 
    Then, $F$ is homeomorphic to the result of splicing  $F_a$ and $F_b$ 
     along $G_a$ and $G_b$.
\end{conjecture}

As was mentioned in~\autoref{sec:introd}, the formulation of this conjecture was motivated by the Casson invariant conjecture from~\cite{NW 90}.
In \cite[Section 6]{NW 05}, Neumann and Wahl proved  that 
the Casson invariant conjecture\index{conjecture!Casson invariant} for splice type singularities follows from the Milnor fiber conjecture. So far, the latter has only been confirmed in special cases. Indeed,
Neumann and Wahl \cite[Section 8]{NW 05} 
showed it for hypersurface singularities defined by equations of the form $z^n + f(x,y)=0$, whereas Lamberson work \cite{L 09} discusses a   generalization of this class of singularities, whose links are obtained from $\bS^3$ by iterated cyclic branched covers along suitable links.

Note that~\autoref{conj:MFC} presumes that all splice type singularities with a fixed 
splice diagram have homeomorphic Milnor fibers, since for a fixed $X$, the singularity $X_a$ 
can be chosen to be \emph{any} splice type singularity with diagram $\Gamma_a$. Remarkably, this subtle yet previously unknown fact is a direct consequence of the proof of the conjecture  
outlined in this paper. More precisely, we have:

\begin{theorem}   \label{thm:invMF}
    The Milnor fibers of any two splice type singularities arising from the same splice diagram 
    are diffeomorphic.
\end{theorem}

\begin{proof}
  The proof outline of the Milnor fiber conjecture discussed in~\autoref{sec:stepsproof} 
  allows us to reduce to the case when the splice diagram $\Gamma$ is star-shaped. 
  For such diagrams, our  description of Milnor fibers through roundings shows 
  that the Milnor fiber of such a splice type singularity does not depend of the higher 
  order terms of the defining splice type system, but only on the initial 
  Pham-Brieskorn-Hamm system (see~\autoref{rem:PBH}). 

                    It remains to check that the Milnor fibers 
                    of those singularities do not depend on the matrix of coefficients satisfying 
                    the Hamm determinant condition. But this is a consequence of the fact that 
                    those singularities are quasi-homogeneous. Indeed, quasi-homogeneity ensures 
                    that any Euclidean ball centered at the origin becomes a Milnor ball for all 
                    such systems simultaneously. This proves the statement. 
\end{proof}

\section{The main ideas of our proof}   \label{sec:princproof}

In this section we give an informal description of the main ideas involved in our proof of~\autoref{conj:MFC}. In~\autoref{ssec:canMiln}, we explain  how to canonically decompose 
Milnor fibers into pieces using real oriented blowups of embedded resolutions of smoothings.  
In ~\autoref{ssec:ourusequasitor} we extend this construction to 
{\em quasi-toroidalizations} of suitable smoothings using the notion of \emph{rounding} of 
a complex log space and explain how to construct quasi-toroidalizations using 
tropical geometry techniques. Finally,~\autoref{ssec:loctrop} gives some basic intuitions 
about {\em local tropicalization} and {\em Newton non-degeneracy}, two notions  
which are key players in 
our construction of quasi-toroidalizations of smoothings.

\medskip
\subsection{Canonical Milnor fibrations through real oriented blowups} 
\label{ssec:canMiln}
$\:$
\medskip

In this subsection we explain how A'Campo's\index{real oriented blow up!A'Campo} 
notion of {\em real oriented blow up} 
yields canonical representatives of the Milnor fibrations over the circle of a given 
smoothing of an isolated singularity, once an embedded resolution of the smoothing 
is fixed.
\medskip

Throughout, we  let $(X,o)$ be an isolated singularity of arbitrary dimension and we let
$f \colon (Y, o) \to (\CC, 0)$ be a smoothing of  $(X,o)$. Consider an 
\emph{embedded resolution} of $f$, that is, a modification $\pi \colon \tilde{Y} \to Y$ 
which restricts to an isomorphism outside $o$, such that 
$\tilde{Y}$ is smooth and the zero level set $Z(\tilde{f})$ of the lifting 
$\tilde{f} := f \circ \pi$ of $f$ to $\tilde{Y}$ is 
a normal crossings divisor. It is natural to ask how the  non-zero levels of 
$\tilde{f}$ degenerate to  $Z(\tilde{f})$. By definition, these  levels are identified via $\pi$ 
with the Milnor fibers of $f$. This produces a  decomposition of those  
Milnor fibers into compact pieces, each piece  consisting of the points which degenerate to a fixed  irreducible component of $Z(\tilde{f})$. These pieces are manifolds with corners, 
whose boundaries degenerate to the singular locus of $Z(\tilde{f})$. 

This decomposition into pieces is analogous to Mumford's plumbed decomposition 
 of the link of an isolated surface singularity obtained 
by looking at the way the link degenerates onto 
the exceptional divisor of a good resolution (see~\cite[Section 1]{M 61}).  As in that prototypical case, the decomposition 
of a given Milnor fiber of $f$ is not canonical, because it depends on the choices 
of embedded resolution,  of suitable coordinate 
systems near the singular locus of the exceptional divisor and also of 
a level  $f^{-1}(\lambda)$  of $f$ with  $0<| \lambda |\ll 1$.

Once the embedded resolution is fixed, the non-canonical aspect of the construction 
can be repaired via the
operation of \emph{real oriented blowup}, introduced by A'Campo in his
study of monodromies of germs $f\colon (\CC^{n+1}, 0) \to (\CC,
0)$~\cite[Section 2]{A 75}. This operation may be performed starting
from any normal crossings divisor $D$ (seen as a reduced hypersurface)
in a complex manifold $W$.  Its effect is to determine a {\em canonical} cut
of $W$ along $D$, which contrasts with the dependency of a classical cut 
(see \autoref{def:cutmanif}) on the choice of a tubular neighborhood. 
 This canonical cut produces a real analytic manifold with corners
$W_D$, together with a map 
   \[ \tau_{W, D}\colon W_D \to W.\] 
 This {\em real oriented blowup} map is
proper and a homeomorphism on $W \setminus D$.  It sends the
topological boundary $\partial_{top} W_D$ of $W_D$ onto $D$ and,
furthermore, the corner locus of $W_D$ is the preimage of the singular
locus of $D$ under $\tau_{W,D}$.  In this way, the {\em algebro-geometric boundary} $D$ 
of the pair $(W, D)$ (in algebraic geometry it is customary to say that a boundary 
is a divisor) is replaced by the {\em topological boundary} $\partial_{top} W_D$ 
of the piecewise-smooth manifold
$W_D$. \autoref{fig:realBlowup} shows a ``real analog'' of this
procedure for a divisor $D$ with two components in a smooth surface.

\begin{example}
    When $W = \CC$ and $D$ is the origin, the real oriented blowup is the map 
    $\tau_{\CC, \{0 \}} \colon [0, + \infty) \times \bS^1 \to \CC$ obtained through 
    the use of polar coordinates: 
    $(r, e^{i \theta}) \to r e^{i \theta}$. The origin of $\CC$, seen 
    as an algebro-geometric boundary, is replaced by the topological boundary circle $\bS^1$ 
    of the cylinder $[0, + \infty) \times \bS^1$. This example 
    will be thoroughly discussed in~\autoref{ssec:loggeomround}.
\end{example}

\begin{figure}
  \begin{center}
    \includegraphics[scale=0.6]{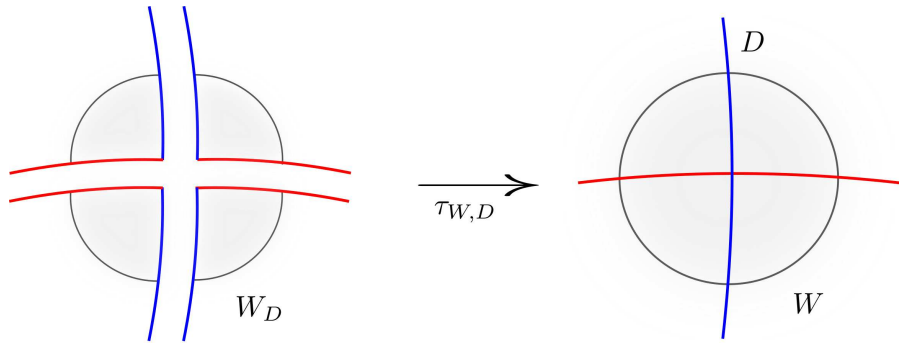}
  \end{center}
  \caption{Real analog of A'Campo's real oriented blowup of $W$ along the 
       divisor $D$.\label{fig:realBlowup}}
\end{figure}

Let us come back to the lifted morphism $\tilde{f}\colon \tilde{Y} \to \CC$ defined earlier. Performing  both the real oriented blowup of $\tilde{Y}$ along $Z(\tilde{f})$ and of $\CC$ at 
the origin allows us to lift  the function $\tilde{f}$ in a canonical way to  those new spaces. Moreover, the restriction $\partial (\tilde{f})\colon  \partial  \tilde{Y}_{Z(\tilde{f})}   \to  \bS^1$  of this lift 
to the boundaries of the source and target spaces gives \emph{a canonical representative of the Milnor fibration of $f$ above a circle}, relative to its embedded resolution $\pi$ 
(see A'Campo's \cite[Section 2]{A 75} and \autoref{cor:milntubelog} of a 
theorem of Nakayama and Ogus).

Since the source space $ \partial  \tilde{Y}_{Z(\tilde{f})} $ is endowed with a canonical surjection
\begin{equation}\label{eq:cutMap}
  \tau_{\tilde{Y}, Z(\tilde{f}) }  \colon   \partial  \tilde{Y}_{Z(\tilde{f})}  \twoheadrightarrow  Z(\tilde{f}),
\end{equation}
we see that it inherits a canonical decomposition into  pieces that are manifolds with corners. 
Each piece lies above an irreducible component of $Z(\tilde{f})$. This yields the desired decomposition of all fibers of the canonical Milnor fibration $\partial (\tilde{f})$.

\medskip
\subsection{Quasi-toroidalizations} $\:$   \label{ssec:ourusequasitor}
\medskip

In this subsection we introduce a class of maps called \emph{quasi-toroidalizations}, 
associated to smoothings of isolated singularities, which are more general than 
the embedded resolutions $\pi\colon \tilde{Y}\to Y$ considered in \autoref{ssec:canMiln}, and which 
play a central role in our proof of~\autoref{conj:MFC}. They may have mildly singular total spaces 
and special fibers which are not crossing normally. Remarkably, the construction 
of canonical Milnor fibrations from modifications explained in~\autoref{ssec:canMiln} 
can be applied to quasi-toroidalizations as well. This extension uses the operation of 
\emph{rounding}, a generalization of real oriented blowups introduced by Kato and Nakayama 
\cite{KN 99} in the context of \emph{logarithmic geometry} (in the sense of Fontaine and Illusie), 
or \emph{log geometry} for short. The latter will be reviewed in~\autoref{ssec:logspacesmorph}. 
Quasi-toroidalizations are relevant for us, as the modifications induced by the natural 
fans we use in order to subdivide the local tropicalizations of our germs are quasi-toroidalizations. 
We could of course subdivide  those fans even further, in a non-canonical way, until we reach an  
embedded resolution morphism. However, this process would fail to describe our morphisms 
explicitly in terms of the given splice diagrams.

\medskip

Let us start with the notions of toroidal varieties and morphisms:

\begin{definition}\label{def:toroimorph}
  Let $(W, \partial W)$ be a pair consisting of an equidimensional complex analytic space $W$ 
  and a reduced complex hypersurface $\partial W$ in it. We say that $(W,\partial W)$ is a 
  \textbf{toroidal variety}\index{variety!toroidal}\index{toroidal!variety} or a 
  \textbf{toroidal pair}\index{toroidal!pair} if it is  locally analytically isomorphic to 
  the pair consisting of a toric variety and its \textbf{toric boundary}, i.e., the complement 
  of its dense algebraic torus. Such a local isomorphism is a \textbf{toric chart} of the toroidal pair.  
  The hypersurface $\partial W$ is called a \textbf{toroidal boundary}\index{toroidal!boundary} for $W$.
  
  If $(W, \partial W)$ is a toroidal pair, then the {\bf toroidal stratification} of $W$ is obtained 
    by gluing together using the toric charts the preimages of the various toric orbits. 
  
     A \textbf{toroidal morphism}\index{toroidal!morphism} 
     is a complex analytic morphism $(V, \partial V) \to (W, \partial W)$ between     
     toroidal varieties that is locally analytically a monomial map when restricted 
     to convenient toric charts.
\end{definition}

\begin{remark}\label{rem:toroidalvstoric} 
   The notion of toroidal variety generalizes that of a toric variety, since every pair consisting of a toric variety and its toric boundary is automatically toroidal. Note that $(W, \partial W)$ is a toroidal pair if and only if $W\smallsetminus \partial W \hookrightarrow W$ is a \emph{toroidal embedding} in the sense of~\cite{KKMS 73}.
\end{remark}

\begin{definition}\label{def:quasi-toroidal} 
   Fix  an isolated singularity $(X,o)$ of arbitrary dimension and a smoothing
   $f\colon (Y,0) \to (\CC,0)$ of it. A  \textbf{quasi-toroidalization}\index{quasi-toroidalization} of $f$ 
   is  a modification $\pi \colon \tilde{Y} \to Y$  satisfying the following conditions:
  \begin{enumerate}
      \item there exists a reduced hypersurface $\partial \tilde{Y}$ in $\tilde{Y}$ such that  
           the lifting  $\tilde{f} \colon (\tilde{Y}, \partial \tilde{Y}) \to (\CC, 0)$ 
           of $f$ to $\tilde{Y}$ is a toroidal morphism,
      \item the zero-locus $Z(\tilde{f})$ 
            of  $\tilde{f}$ in $\tilde{Y}$ is included in the toroidal boundary 
            $\partial \tilde{Y}$, and
       \item locally around each point $x$ of $\partial \tilde{Y}$, we can find a 
           (local) toroidal stratum $S$ for which the (local) irreducible components 
           of $\partial \tilde{Y}$ containing $x$ that are  not components of 
           $Z(\tilde{f})$ are exactly those irreducible components of $\partial \tilde{Y}$ 
           containing $S$. 
  \end{enumerate}
\end{definition}

\noindent The third condition may seem strange at 
first sight. It originates in the observation that, unlike normal crossing divisors 
in manifolds (or smooth varieties), 
closed subdivisors of toroidal boundaries are not necessarily toroidal 
(see {Examples}~\ref{ex:quasitorcone} and~\ref{ex:quasitor4D}). 
Its exact formulation is explained in \autoref{rem:motivqtor}.

\begin{remark}\label{rem:whyQT}
    The relevance of quasi-toroidalizations for our work lies in the following crucial observation.     
    Mimicking the real oriented blowup 
     construction of the previous subsection in  this more general context via Kato and Nakayama's     
     rounding operation (see \autoref{ssec:KNrounding}) and restricting to $Z(\tilde{f})$, 
      produces a  morphism
\begin{equation}\label{eq:restrictionBoundary}
  \partial (\tilde{f}) \colon  \partial  \tilde{Y}_{|_{Z(\tilde{f})}}   \to  \bS^1
\end{equation}
      which is a representative of the Milnor fibration of $f$. This is a consequence 
      (see \autoref{cor:milntubelog}) of a 
      more general local triviality theorem for roundings proved by Nakayama and 
     Ogus \cite[Theorem 3.7]{NO 10}, stated as~\autoref{thm:logehresm}  below. 
\end{remark}

In order to use quasi-toroidalizations to determine the topology of the Milnor fiber 
of a splice type singularity $(X,0)\hookrightarrow (\CC^n,0)$, we must first pick an appropriate smoothing $f\colon (Y,0)\to (\CC,0)$. Notice that, unlike the quasi-toroidalization $\pi\colon \tilde{Y}\to Y$,  the Milnor fiber is independent (up to diffeomorphism) of the choice of $f$ because $X$ is an isolated complete intersection, which implies by an important result of Tyurina \cite[Theorem 8.1]{T 70}  
(see also \cite[Chapter 6]{L 84} or  \cite[Theorem 1.16]{GLS 07}) that its miniversal 
deformation has an irreducible (even smooth) base. Thus, we may pick a smoothing that 
is well-adapted to proving~\autoref{conj:MFC}.
 We construct such a smoothing by deforming the splice type system defining $X$ in a way compatible with the  given internal edge $[a,b]$ (see \autoref{def:deformation}).  
The deformed system defines a three-dimensional germ $(Y,0) \hookrightarrow 
(\CC^{n+1}, 0)$.

The {\em local tropicalization} of this deformed system (a notion discussed in 
~\autoref{ssec:loctrop} and~\autoref{sec:loctropNND}) is supported on a three-dimensional fan $\cF$ 
contained in the cone of weights $(\Rp)^{n+1}$ defining $\CC^{n+1}$ as an  
affine toric variety. 
This fan has the following crucial property:

\begin{proposition}\label{prop:QuasiTfromTropical} 
      Consider the  toric birational morphism 
     $\pi_{\cF}  \colon \tv_{\cF} \to \CC^{n+1}$ defined by the fan $\cF$ and let  
    $\pi\colon \tilde{Y} \to Y$ be its restriction to  the strict transform of $Y$ by $\pi_{\cF}$. 
    Then, the map $\pi$ is a quasi-toroidalization  of $f$. Furthermore,
   the dual complex of the exceptional divisor $E:=\pi^{-1}(0)$ is naturally isomorphic to a    
   subdivision of the splice diagram $\Gamma$ of $(X,0)\hookrightarrow (\CC^n,0)$,  
   obtained by adding an interior point $r$ of the 
    edge $[a,b]$ as an extra node and subdividing $[a,b]$ accordingly.
\end{proposition}

The second part of~\autoref{prop:QuasiTfromTropical} confirms that the quasi-toroidalization 
$\pi \colon  \tilde{Y} \to Y$  of $f$ is adapted to the proof of~\autoref{conj:MFC}. As we discuss in Step (\ref{item:decompdivf}) of~\autoref{sec:stepsproof},  this property yields a decomposition of  the exceptional divisor $E$ into three pieces: two divisors $D_a$ and $D_b$  coming from the $a$- 
and $b$-sides, respectively, and  an irreducible central divisor 
$D_r$ corresponding to the new vertex $r$. 
Moreover, the special fiber $Z(\tilde{f})$ of 
$\tilde{f}$ is reduced, making $\tilde{f}$ analogous to a semistable degeneration in the sense of~\cite{KKMS 73}. Its component $D_r$ is a cartesian product of two projective curves. In turn, this last 
fact then allows us to prove that the central piece of 
the Milnor fiber which connects the $a$-side and $b$-side has
  the desired product structure. This central piece is obtained by intersecting  the preimage 
  $\tau_{\tilde{Y}, Z(\tilde{f}) }^{-1}(D_r)$ of the analog in our context 
  of the map ~\eqref{eq:cutMap} 
  with  a fiber of the restriction map $\partial(\tilde{f})$ from~\eqref{eq:restrictionBoundary}. 
 The product structure results from 
the reducedness of the special fiber $Z(\tilde{f})$ combined with  a result of Achinger and Ogus 
\cite[Corollary 4.1.9]{AO 20}.  

In a similar way, the $a$-side piece of any Milnor fiber of $(X,0)$ is recovered by intersecting the preimage $\tau_{\tilde{Y}, Z(\tilde{f}) }^{-1}(D_a)$ and
the fibers of the restriction map $\partial (\tilde{f})$ from~\eqref{eq:restrictionBoundary}. 
This piece can then be identified with a Milnor fiber  of a smoothing $f_a \colon (Y_a, 0) \to (\CC, 0)$ 
of a convenient splice type singularity $(X_a, 0)$ associated to the $a$-side rooted subtree 
$\Gamma_a$ of $\Gamma$, cut (as explained in \autoref{ssec:mfconj})  
along a pushed page of the Milnor open book 
defined by the root coordinate $x_{r_a}$. The total space $Y_a$
is determined from  the system defining $(Y,0)$ by its pullback under a suitable monomial map $\varphi_a\colon \CC^{n_a+1}\to \CC^{n+1}$ (see Steps (\ref{item:toricmorpha}) and (\ref{item:deformsysta}) of~\autoref{sec:stepsproof}). In turn, the smoothing $f_a$ is obtained by restricting $f\circ \varphi_a$ to $Y_a$.  This identification of portions of Milnor fibers 
of smoothings of distinct singularities is done using log geometry techniques. More concretely,  we prove that the corresponding pieces
 of the canonical Milnor fibrations obtained through rounding are homeomorphic (see Steps (\ref{item:commuttrianga}) through (\ref{item:cutmiln}) of~\autoref{sec:stepsproof}). 

Basic to the proof of this homeomorphism is the following reinterpretation of cutting the Milnor fibers of $(X_a,0)$ in the direction of the root $r_a$  
of $\Gamma_a$. It can be achieved by  cutting a level of the smoothing 
$f_a \colon (Y_a, 0) \to (\CC, 0)$ of $(X_a, 0)$ along the coordinate hypersurface $Z(x_{r_a})$ associated to the root of $\Gamma_a$. In order to perform such a cut in the logarithmic setting via rounding  we must  cut a suitable modification $\tilde{Y}_a$ of $Y_a$
along $Z(\tilde{f}_a \tilde{x}_{r_a})$. Note that the latter is precisely the total transform 
of the intersection of $Y_a$ with the coordinate hypersurface $Z(f_ax_{r_a})$. This construction is illustrated in~\autoref{fig:MilnorFibersCutv2}.  The map $(\tilde{f_a})^{\dagger}_{\log} \colon (\tilde{Y_a}, \cO^{\star}_{\tilde{Y_a}}(-Z(\tilde{f_a}\tilde{x}_{r_a})))_{\log} \to (\D, \cO_{\D}(- \{ 0 \}))_{\log}$ in the figure is the rounding of  a log enhancement of $\tilde{f_a}$ (see Definitions 
\ref{def:indlogmorph} and \ref{def:rounding}).

\begin{figure}[tb]
  \includegraphics[scale=0.6]{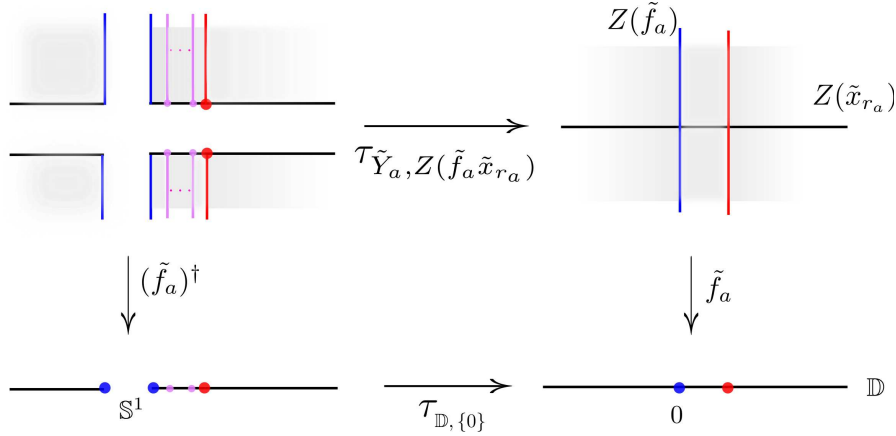}
  \caption{Local drawing of the rounding operation away from $\tilde{X_a}$. 
  The canonical representative of the Milnor fiber (in red) has been cut using $Z(\tilde{x}_{r_a})$. 
  The special fiber $Z(\tilde{f}_a)$ is drawn in blue. 
  The origin has been replaced by the circle $\bS^1$ under the real oriented blow up $\tau_{\D, \{ 0 \}}$ 
  of \autoref{ssec:loggeomround}. This circle is depicted on the left as a pair of points. 
  The two remaining fibers (in purple) indicate the local triviality of the Milnor fibration.
    \label{fig:MilnorFibersCutv2}}
\end{figure}

\vfill
\pagebreak

\medskip
\subsection{Local tropicalization and Newton non-degeneracy} $\:$ \label{ssec:loctrop}
\smallskip

In this subsection we discuss the essential role played 
 in our proof of~\autoref{conj:MFC} by both the \emph{local 
tropicalizations} of analytic subgerms of $(\CC^{n+1}, 0)$ and the 
\emph{Newton non-degeneracy condition}. 
More details on these two notions can be found in~\autoref{sec:loctropNND}. 
\medskip

Let $(Y,0)\hookrightarrow (\CC^{n+1},0)$ be a proper equidimensional 
subgerm, without irreducible components contained 
in the toric boundary of $\CC^{n+1}$.  
Given any  fan $\cF$ of $(\Rp)^{n+1}$, we can consider the toric morphism 
$\pi_{\cF}  \colon \tv_{\cF} \to \CC^{n+1}$ defined by $\cF$ and its restriction 
   \[ \pi\colon \tilde{Y}\to Y\] 
to the strict transform $\tilde{Y}$ of $Y$ by 
$\pi_{\cF}$. 
Note that the map $\pi_{\cF}$ is not proper if the support of $\cF$ is strictly 
included in $(\Rp)^{n+1}$. 
However, the restricted map $\pi$ \emph{is proper}  
whenever the support of $\cF$ contains the {\em local tropicalization} of the 
embedding $(Y,0)\hookrightarrow (\CC^{n+1},0)$ (see \autoref{pr:localTropicalCompactification}).

Local tropicalizations\index{local tropicalization} were developed by the last two authors 
in~\cite{PPS 13} as a tool to study singularities. They are a local version of global 
tropicalizations (or ``non-Archimedean amoebas'') of subvarieties of the algebraic torus 
$(\CC^*)^{n+1}$. Namely, the local tropicalization of a subgerm 
of $(\CC)^{n+1}$ is the support of a fan contained in $(\R_{\geq 0})^{n+1}$. 
We used this notion in~\cite{CPS 21} to prove several properties of splice type surface singularities, including their Newton non-degeneracy property and the first tropical interpretation of splice 
diagrams, whenever they satisfy the determinant and semigroup conditions.

The statement regarding the properness of $\pi\colon \tilde{Y}\to Y$ is a consequence of the following 
local analog  of Tevelev's result~\cite[Proposition 2.3]{T 07}. For a proof, we refer the reader to~\cite[Proposition 3.15 (1)]{CPS 21}.

\begin{proposition} \label{pr:localTropicalCompactification}
Let $(Y,0)$ be any reduced complex analytic subgerm  of $\CC^{n+1}$ without irreducible 
components contained in the toric boundary $\partial \CC^{n+1}$. Let $\cF$ be a fan whose 
support is contained in $(\R_{\geq 0})^{n+1}$. Then,  the strict transform morphism 
$ \pi\colon \tilde{Y}\to Y$ is proper if, and only if,  the support of $\cF$ 
 contains the local tropicalization  of $Y$.
\end{proposition}

As in the global setting (see \cite[Theorem 3.2.3]{MS 15}), local tropicalizations admit an alternative more algebraic description using {initial ideals} relative to non-negative weight vectors. Namely, as discussed in~\autoref{def:posloctrop}, the \emph{local tropicalization} $\Trop Y$ of a germ $(Y,0)\hookrightarrow (\CC^{n+1}, 0)$ defined by an ideal $I$ of the local ring $\CC\{z_0,\ldots, z_{n}\}$ of $(\CC^{n+1}, 0)$ is the closure of the set of non-negative weight vectors $\wu{}$ such that the $\wu{}$-initial ideal $\initwf{\wu{}}{I}$ contains no monomials.

This  viewpoint is particularly useful  when working with explicit equations 
defining $(Y,0)$ inside $(\CC^{n+1}, 0)$.
 For instance, it allowed us to determine the local tropicalizations of splice 
type singularities in \cite{CPS 21}. Similar methods can be used to 
compute the local tropicalization of 
an edge-deformation $(Y,0) \hookrightarrow (\CC^{n+1}, 0)$ in the sense of 
\autoref{def:deformation}. Namely, the support of the fan $\cF$ alluded to in~\autoref{prop:QuasiTfromTropical} is $\Trop Y$.

\begin{remark}\label{rem:tropicalizingFan} 
      Notice that $\Trop Y$ has no canonical fan structure. Particularly useful to us are those 
      fan structures where the initial ideals of $I$ are constant along the relative interiors of all its cones. 
      A fan $\cF$ with this property and support equal to $\Trop Y$ is called a 
      {\em standard tropicalizing fan}\index{tropicalizing fan!standard}  
     (see \autoref{def:tropfans}). 
\end{remark}

The use of standard tropicalizing fans is convenient when dealing with  
{\em Newton non-degenerate germs} (see \autoref{def:NNDgerms}):

\begin{proposition}\label{pr:transverseBoundary} 
     Assume that $\cF$ is a standard tropicalizing fan 
     of a Newton non-degenerate germ $(Y,0)\hookrightarrow (\CC^{n+1},0)$, 
     and let $\tilde{Y}$ be  the  strict transform  of $Y$ under the toric morphism 
     $\pi_{\cF}\colon \tv_{\cF}\to \CC^{n+1}$. Then, $\tilde{Y}$  is transversal to the 
     toric boundary $\partial\tv_{\cF}$ of $\tv_{\cF}$ in the sense of \autoref{def:bdrytransv}. 
\end{proposition}

\begin{remark}\label{rem:NND} 
    \autoref{pr:transverseBoundary} is the crucial ingredient allowing us to prove 
    \autoref{prop:QuasiTfromTropical}, concerning our special 
    smoothings $f : (Y,0) \to (\CC, 0)$ of splice-type singularities 
    (see \autoref{thm:NNDtoroidal}). As explained in~\autoref{ssec:ourusequasitor}, once we know that 
   $\pi$ is a quasi-toroidalization of $f$, a consequence (see \autoref{cor:milntubelog}) 
   of a general local triviality theorem of Nakayama and Ogus 
    applied to $f \circ \pi$ yields a canonical 
   representative of the Milnor fibration of $f$.
\end{remark}

\section{Logarithmic ingredients}   \label{sec:logingred}

In this section we give an overview of the logarithmic tools needed to prove~\autoref{conj:MFC}. 
\autoref{ssec:quasitor} discusses in further detail than in \autoref{ssec:ourusequasitor} 
the notion of quasi-toroidalization of a smoothing. In turn, \autoref{ssec:loggeomround}  provides a first glimpse of both the rounding operation and the 
notion of a log structure by means of the classical passage to polar coordinates.  \autoref{ssec:logspacesmorph} reviews basic definitions of log spaces and  morphisms between them that are needed to introduce the rounding operation of Kato and Nakayama. The latter is the subject of~\autoref{ssec:KNrounding}. \autoref{ssec:loctrivquasitor} discusses Nakayama and Ogus' local triviality theorem. This result allows us to get canonical representatives of Milnor fibrations over the circle using quasi-toroidalizations of smoothings.

\medskip
\subsection{Quasi-toroidal subboundaries and quasi-toroidalizations of smoothings} 
 \label{ssec:quasitor}  $\:$
 \medskip

We begin this subsection by defining 
\emph{boundary-transversal subvarieties} of toroidal varieties. 
Then, we introduce the notions of \emph{quasi-toroidal subboundaries} of toroidal varieties 
and of \emph{quasi-toroidalizations} of smoothings. These last two notions play a central role in both~\autoref{cor:quasitorrelcoh} and  the  local triviality theorem of Nakayama and Ogus (see~\autoref{thm:logehresm}).

\medskip

Recall that {\em toroidal varieties} were introduced in~\autoref{def:toroimorph}.  
We define now a special type of complex analytic subvarieties of toroidal varieties, that are relevant for proving~\autoref{conj:MFC}:

  \begin{definition}   \label{def:bdrytransv}
   Let $(W, \partial W)$ be a toroidal variety. A reduced closed equidimensional subvariety $V$ of $W$ 
   is called \textbf{boundary-transversal}\index{boundary-transversality}, 
   or \textbf{$\partial$-transversal} for short, 
   if  the following conditions are satisfied for each stratum $S$ of the toroidal stratification  of $W$:
      \begin{enumerate}
            \item  \label{manifcond} 
                the  analytic space $V \cap S$ 
                is a  (possibly empty) equidimensional complex manifold;
          \item \label{codimcond} 
            if $V \cap S \neq \emptyset$, then $\codim_V(V\cap S) =\codim_W (S)$.
      \end{enumerate}
\end{definition}

As~\autoref{thm:NNDtoroidal} below shows, our main example of $\partial$-transversal 
  subvarieties are strict transforms of Newton non-degenerate germs  
  $(X,0) \hookrightarrow \CC^n$   
   by toric birational morphisms defined by standard tropicalizing fans 
   of $(X,0) \hookrightarrow \CC^n$.

\begin{remark}   \label{rem:classicalNotionTransverse} 
     Notice that when $W$ is a complex manifold and $S$ is a 
     submanifold of it, then conditions (\ref{manifcond}) and 
     (\ref{codimcond}) of ~\autoref{def:bdrytransv} recover in a neighborhood of $S$ the classical 
     notion of {\em transversality} of two submanifolds of an ambient manifold (meaning that 
     at each of their intersection points, the sum of their tangent spaces is equal to 
     the tangent space of the ambient manifold). 
     Indeed, assume that  $S$ is a submanifold of $W$ and $V$ 
     is a reduced subvariety of $W$ such that $V \cap S$ is smooth and 
     $\codim_V(V \cap S) = \codim_W (S)$. Then 
       $V$ is smooth in a neighborhood of $S$ and transversal to it. 
       Condition (\ref{manifcond}) regarding the smoothness of the intersection is essential, 
       as shown by the example of the pair $(W, S):= (\CC^2, Z(x))$ and $V := Z(y^2 - x^3)$: the 
       analytic space $V \cap S$ is the doubled origin $\mathrm{Spec} (\CC[y]/(y^2))$, 
       therefore it is not a manifold. 
       Condition (\ref{codimcond}) regarding equality of codimensions is also crucial, 
       as shown by the example 
       of the pair $(W, S):= (\CC^2, 0)$ and $V := Z(x)$, since  $\codim_V(V \cap S) = 1$ and 
       $ \codim_W (S)= 2$. 
\end{remark}

\begin{remark}   \label{rem:schoncompar}
    The notion of $\partial$-transversality in the toric case is closely related to that of 
    {\em sch{\"o}n compactifications} of subvarieties of tori, a concept introduced by 
    Tevelev in~\cite[Definition 1.3]{T 07}. An equivalent definition, more suitable for 
    our purposes was given by  Maclagan and Sturmfels in \cite[Definition 6.4.19]{MS 15} 
    (see also~\cite[Proposition 6.4.7]{MS 15}):
       \begin{quote}
          {\em Let $V$ be an equidimensional subvariety of an algebraic torus $T$, and let $\tv_{\cF}$ 
         be a normal toric variety with dense torus $T$. The compactification $\overline{V}\subset \tv_{\cF}$ 
         is {\bf sch\"on}\index{sch\"on compactification} if, and only if,  
         $\overline{V}$ intersects each orbit $\cO_{\tau}$ of $\tv_{\cF}$ 
          and, furthermore, these intersections are smooth with 
          $\codim_{\overline{V}}(\overline{V}\cap \cO_{\tau}) = \codim_{\tv_{\cF}}(\cO_{\tau})$.}
      \end{quote}
   Notice that the equality of codimensions in this last definition agrees with condition (\ref{codimcond})   
   of~\autoref{def:bdrytransv}. In particular,    if  $\overline{V}\subset \tv_{\cF}$ is a 
   sch\"on compactification,  then  $\overline{V}$ is 
   $\partial$-transversal in the toroidal variety $(\tv_{\cF}, \partial \tv_{\cF})$.  
   Moreover,  a $\partial$-transversal subvariety of a toric variety is a sch\"on compactification 
   of its intersection with the dense torus if, and only if,  it meets each torus orbit.
 \end{remark}

\smallskip 
Boundary-transversal subvarieties of toroidal varieties admit an inherited toroidal structure whose associated  toroidal stratification is compatible with the ambient one. This is summarized in the following folklore result that can be easily established by working locally in toric charts: 

\begin{proposition}  \label{prop:restoroid}
      Let $(W, \partial W)$ be a toroidal variety and let $V$ be a $\partial$-transversal subvariety of $W$. Consider the set $\partial V := V \cap \partial W$. 
      Then: 
       \begin{enumerate}
           \item The pair $(V, \partial V)$ is a toroidal variety.
           \item 
               The  strata  of the toroidal stratification of $(V, \partial V)$
                 are the connected components of the 
                  intersections $V \cap S$, where $S$ varies among the strata 
                   of the toroidal stratification of $W$. 
            \item  The embedding $(V, \partial V) \hookrightarrow (W, \partial W)$  
                is a toroidal morphism of toroidal varieties. 
      \end{enumerate}
\end{proposition}

Of particular interest to us are  special 
subvarieties of toroidal boundaries obtained by taking unions of certain irreducible 
components satisfying a special condition, as we now describe:

   \begin{definition}  \label{def:face-subdiv}
       Let $(W, \partial W)$ be a toroidal variety and let  $\cD_W$ be a subdivisor of the toroidal 
       boundary $\partial W$ of $W$. We say that $\cD_W$ is a 
        \textbf{quasi-toroidal subboundary of $(W, \partial W)$}\index{quasi-toroidal subboundary}  
       if in the neighborhood of any point of $\partial W$, the complementary divisor 
               $\partial W - \cD_W:= \overline{\partial W \smallsetminus \cD_W}$ consists of the local 
               irreducible components of $\partial W$ containing a fixed stratum of the 
               toroidal stratification.       
   \end{definition}

   We illustrate this definition with two examples:
   \begin{example}   \label{ex:quasitorcone}
        Consider the quadratic cone 
       $W := Z(z^2 -xy)  \hookrightarrow \CC^3_{x,y,z}$. It is a normal 
       affine toric surface, whose  boundary $\partial W$ is the union  $Z(z, xy)$ of the  
       $x$-axis $L'$ and the $y$-axis $L$. Then, $L$ is a quasi-toroidal subboundary of 
       the toroidal surface $(W, L+ L')$. However,  $(W, L)$ is not a toroidal pair, 
       because the boundary of a toric variety is always locally reducible at a singular point.
   \end{example}

      \begin{figure}[tb] 
     \includegraphics[scale=0.6]{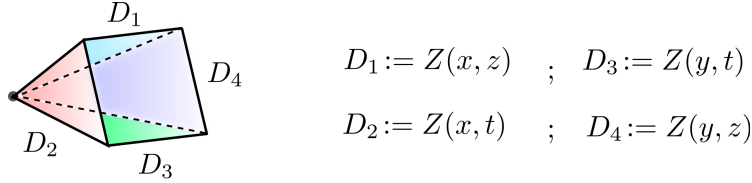}
     \caption{Boundary strata of the toroidal variety $(W,\partial W)$, where  $W:=Z(xy-zt)\subset \CC^4$. 
     The toric boundary $\partial W$ has   \label{fig:QuasiToroidalEx2}  four irreducible components 
     (see \autoref{ex:quasitor4D}).}
     \end{figure}

   \begin{example}   \label{ex:quasitor4D}
     We consider the normal affine toric hypersurface 
     $W :=Z (x\,y-z\,t)\hookrightarrow \CC^4_{x,y,z,t}$ whose boundary $\partial W$ 
     is the union of the coordinate subspaces $D_1:=Z(x,z)$, $D_2:=Z(x,t)$, 
     $D_3:= Z(y,t)$, $D_4:= Z(y,z)$, as seen in~\autoref{fig:QuasiToroidalEx2}.  
     Then, $D_1 + D_4$ is a quasi-toroidal subboundary because $D_2$ and $D_3$ 
     are the only components of $\partial W$ containing $Z(t)\cap W$. In turn, 
     $D_1+D_3$ is not a quasi-toroidal subboundary since the only stratum contained 
     in both $D_2$ and $D_4$ is the origin, but this point is contained in both 
     $D_1$ and $D_3$ as well.
  \end{example}
  
   \begin{remark} \label{rem:motivqtor} 
          We were led to~\autoref{def:face-subdiv} by trying to determine which 
           subdivisors of boundaries of toroidal spaces produce associated divisorial log structures that are 
           relatively coherent in the sense of Nakayama and Ogus \cite[Definition 3.6]{NO 10}   
            (see~\autoref{prop:toroidcharrelcoh}).
  \end{remark}

        Quasi-toroidal subboundaries are essential ingredients to  define   
        {\em quasi-toroidalizations of smoothings}:
        
     \begin{definition}  \label{def:quasitorsmoothing}
        Let $f \colon (Y, o) \to (\CC, 0)$ be a smoothing. 
        A \textbf{quasi-toroidalization}\index{quasi-toroidalization} of $f$ is
         a modification $\pi \colon \tilde{Y} \to Y$  such that there exists 
         a divisor $\partial \tilde{Y}$ of $\tilde{Y}$ satisfying the following properties:
        \begin{enumerate}
           \item the pair  $(\tilde{Y}, \partial \tilde{Y})$ is a toroidal variety;
           \item  the morphism $\tilde{f}  \colon (\tilde{Y}, \partial \tilde{Y}) \to (\CC, 0)$ is toroidal; 
           \item the zero-locus $Z(\tilde{f})$ 
                of the lifting of $f$ to $\tilde{Y}$ is a quasi-toroidal subboundary 
                 $(\tilde{Y}, \partial \tilde{Y})$.
       \end{enumerate}
    \end{definition}
    
Note that  Definitions~\ref{def:face-subdiv} and \ref{def:quasitorsmoothing} 
are reformulations 
of parts of ~\autoref{def:quasi-toroidal}. Quasi-toroidalizations of smoothings 
feature in Steps (\ref{item:torbirY}) and (\ref{item:torbirYa})    of our proof of the Milnor fiber conjecture.

\medskip
\subsection{Introduction to logarithmic structures and rounding through polar coordinates} $\:$  \label{ssec:loggeomround}
\medskip

In this subsection we introduce \emph{log structures} in the sense 
of Fontaine and Illusie \cite{K 88} and the operation of \emph{rounding} due to Kato and Nakayama 
\cite{KN 99} by ways of a unifying example, namely,  the standard morphism of 
{\em passage to polar coordinates}\index{polar coordinates}:
    \begin{equation} \label{eq:polcoord}
             \boxed{\tau_{\CC, \{0\}}} \colon   [0, + \infty) \times \bS^1   \to    
             \CC   \qquad  (r, e^{i \theta})\mapsto r e^{i \theta}.
     \end{equation}
    We do not claim that this was the original motivation behind the development of these two notions. 
    The reader interested in learning how Fontaine and Illusie discovered 
    log structures may consult~\cite{I 94}.
    \smallskip

Our first objective is to define the map $\tau_{\CC, \{0 \}}$ from~\eqref{eq:polcoord} 
in a coordinate-free fashion, in order to extend it to any pair consisting of a complex space and a hypersurface in it, rather than solely for $(\CC, 0)$.

 Since $\tau_{\CC, \{0 \}}$ is a homeomorphism outside the circle $\bS^1\simeq \tau_{\CC, \{0 \}}^{-1}(0)$ 
 bounding $[0,+\infty)\times \bS^1$, we may view  $\tau_{\CC, \{0 \}}$ as an analog of the 
 usual blowup of the real plane $\CC$ at the origin. While in the usual blowup the origin is replaced by 
 the set of real lines passing through it, in the passage to polar coordinates this point  is replaced by the 
 set of \emph{oriented} lines (which may be canonically identified with the set of 
 half-lines, see \autoref{fig:blowup}). For this reason, the map $\tau_{\CC, \{0 \}}$ 
 is known also under the 
 name of a \emph{real oriented blowup}\index{real oriented blow up}. 
 The analogy between the two blowups may be enhanced  by seeing them both as closures 
 of graphs of maps which are undetermined at the origin. While for the usual 
 blowup the map is the real projectivization $\CC= \R^2 \dashrightarrow \mathbb{P}(\R^2)$, 
 it is the \emph{argument function} for the real oriented blowup:
 
  \begin{definition}\label{def:logCoordinatesPolar}
     Let $z$ be the standard coordinate function on $\CC$. 
     The \textbf{argument function}\index{argument function} 
     $\boxed{\mbox{arg}} \colon \CC^*\to \bS^1$ is defined by 
     $\mbox{arg}(z):=z/|z|$.
 \end{definition}
 
\begin{remark}
   The argument function given above is a variant of the standard notion of argument of a 
   non-zero complex number, which takes values in $\R / 2 \pi \Z$, and is defined by 
   $ r e^{i \theta} \to \theta  \mod 2 \pi$. Our notation follows the choice made by Ogus 
   in~\cite[Section V.1.2]{O 18}.
\end{remark}

 \begin{figure}[tb] 
     \includegraphics[scale=0.5]{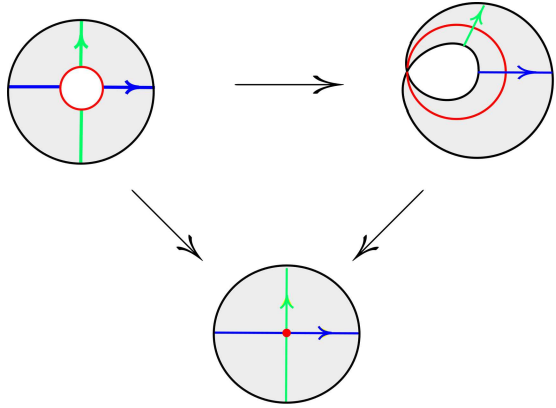}
     \caption{The left arrow is the real oriented blowup while the right one is the usual blowup of the center of a disc. The horizontal arrow is the natural factorisation of the left arrow through the second one.}\label{fig:blowup}
     \end{figure}

 \begin{remark} \label{rem:evolrealblowup}
         The construction of  \emph{real blowups} or \emph{real oriented blowups} 
          was extended by  A'Campo ~\cite[Section 2]{A 75} to arbitrary normal 
          crossings divisors in complex manifolds  
          (see also~\cite[pages 404--405]{KN 94},~\cite[Section I.3]{M 84},~\cite[Section 2.2]{P 77} 
          and~\cite{P 98}). 
          It was later extended by Kawamata \cite[Section 2]{K 02} to toroidal boundaries 
          of special types of toroidal varieties  and by Hubbard, 
          Papadopol and Veselov~\cite[Section 5]{HPV 00} to arbitrary closed analytic 
          subsets of real analytic manifolds. In another direction, A'Campo's definition 
          was extended by Kato and Nakayama to arbitrary log complex spaces 
          (see \autoref{def:rounding} below). It is this last viewpoint which is of interest 
          for us, therefore we explain now how to see the simplest real oriented 
          blow up $\tau_{\CC, \{0\}}$ above as an operation performed on a log complex space. 
   \end{remark}

Note first that the lift $\tau_{\CC, \{0 \}}^*( \mbox{arg})$ of the argument function to 
$ (0, + \infty) \times \bS^1$ can be uniquely extended by continuity 
to  $[0, + \infty) \times \bS^1$. The resulting map 
   \[   \tau_{\CC, \{0 \}}^* (\mbox{arg}) :  [0, + \infty) \times \bS^1 \to \bS^1 \]
is simply the second projection. Therefore, each point 
$P  \in \bS^1= \partial_{top}([0,+\infty)\times \bS^1)$ may be seen 
as a possible place to compute the limit of $\mbox{arg}(z)$ as $z$ converges to $0$. 
As we will now explain, the choice of such a point allows to also define the limit of 
$\mbox{arg}(h(z))$ as $z$ converges to $0$, for all non-zero germs of holomorphic functions 
$h$ at the origin.

Let $\boxed{\cO_{\CC, 0}}$ be the local ring of the complex curve $\CC$ at $0$, 
consisting of the germs of holomorphic functions on $\CC$ at $0$.  
Then $\cO_{\CC, 0} \setminus \{0\}$ is a commutative monoid for multiplication, 
in the following sense:

\begin{definition} \label{def:monoid}
    A {\bf monoid}\index{monoid} is a set endowed with an associative binary operation which 
    has a neutral element. The monoid is {\bf commutative} if the operation is so. 
\end{definition}

Denote by 
$\boxed{\cO_{\CC, 0}^{\star}}$ the subgroup of units of the monoid 
$\cO_{\CC, 0} \setminus \{0\}$, consisting of the germs of 
holomorphic functions which are non-zero at $0$.

Any germ $h \in \cO_{\CC, 0} \setminus \{0\}$  
can be  written in a unique way as $h=z^m\cdot v$ for some $m \in \N$ and 
$v \in \cO_{\CC, 0}^{\star}$. Thus, the following  relation holds in a sufficiently small punctured neighborhood of $0$ in $\CC$: 
     \begin{equation}   \label{eq:argh}
          \mbox{arg}(h) = \frac{h}{|  h |} =  \left(\frac{z}{|  z |} \right)^m    \frac{v}{|  v |} = 
        \mbox{arg}(z)^m \frac{v}{|  v |}.  
    \end{equation}
As a consequence of the fact that $\tau_{\CC, \{0 \}}^*( \mbox{arg})$ extends by continuity 
to $[0, + \infty) \times \bS^1$, we see that  the same is true for 
 the lift $\tau_{\CC, \{0 \}}^* (\mbox{arg}(h))$. 
By abusing notation, we denote this extension also by  $\tau_{\CC, \{0 \}}^* (\mbox{arg}(h))$: 
         \[
                           \xymatrix{
                               [0, + \infty) \times \bS^1
                                  \ar[rrd]^{\:\:\:  \tau_{\CC, \{0 \}}^* (\mbox{arg}(h))} 
                                  \ar[d]_{\tau_{\CC, \{0 \}} }& 
                                                & \\
                                  \CC \ar@{-->}[rr]_{\mbox{arg}(h)}  &  & \bS^1. }  
           \]

If $h_1, h_2 \in \cO_{\CC, 0} \setminus \{0\}$, then on any punctured neighborhood of the origin 
on which they are both non-zero, we have:  
   \[  \mbox{arg}(h_1) \cdot  \mbox{arg}(h_2) =  \mbox{arg}(h_1 \cdot h_2). \]   
 As a consequence, the relation
     \[ \tau_{\CC, \{0 \}}^*( \mbox{arg}(h_1)) \cdot \tau_{\CC, \{0 \}}^*( \mbox{arg}(h_2)) =  
        \tau_{\CC, \{0 \}}^*(\mbox{arg}(h_1 \cdot h_2)) \]   
 is true over a neighborhood of the boundary $\bS^1= \partial_{top}([0,+\infty)\times \bS^1)$ 
 of $[0,+\infty)\times \bS^1$. 
We get:

 \begin{proposition}\label{pr:monoidsFromPoints}
   Consider a point $P \in \bS^1= \partial_{top}([0,+\infty)\times \bS^1)$. 
    Then, the map 
      \[  \begin{array}{ccc} 
              (\cO_{\CC, 0} \setminus \{0\}, \cdot) &   \to   & (\bS^1, \cdot) \\
                  h   &  \to & \tau_{\CC, \{0 \}}^* (\mathrm{arg}(h))(P)
           \end{array} \]
is a \emph{ morphism of multiplicative monoids}
 extending the standard morphism of groups $(\cO_{\CC, 0}^{\star}, \cdot)  \to    (\bS^1, \cdot)$ given by $h\mapsto   \mbox{arg}(h(0))$. 
 \end{proposition}
 
 That is, each point of the topological boundary of the real oriented blowup 
    $[0,+\infty)\times \bS^1$ of $\CC$ at $0$ may be seen as a morphism of monoids 
    from $(\cO_{\CC, 0} \setminus \{0\}, \cdot)$ to  $(\bS^1, \cdot)$. 
  This statement yields the promised intrinsic, ``coordinate-free'', 
  extension of the  map $\tau_{\CC, \{0 \}}$ from~\eqref{eq:polcoord} to arbitrary pairs of complex varieties and hypersurfaces in them:

  \begin{definition}\label{def:roundingPrelim}  
      Let $(W,D)$ be a pair consisting of a reduced complex variety $W$, 
     and a hypersurface $D\subset W$ (which may be also seen as a reduced 
     Weil divisor).  For every point $x \in W$, denote by $\mathcal{M}_{W,D,x}$ 
     the multiplicative monoid of germs at $x$ of holomorphic  functions on $W$       
     which are non-zero outside $D$. Consider the set
        \begin{equation}\label{eq:Md}
              \begin{split}
                        W_D:=  \{  (x,P)\colon x\in W \text{ and }   P\colon 
                   (\mathcal{M}_{W,D,x} , \cdot) \to  (\bS^1, \cdot) \text{ is a morphism of monoids such that } \\
                   P(v) = \mbox{arg}(v(x))  \text{ for every } v \in \cO_{W, x}^{\star} \},
              \end{split}
         \end{equation}
   The \textbf{rounding map}\index{rounding!map} $\tau_{W,D}\colon W_D\to W$ 
   is given by the first projection.
 \end{definition}

   \begin{example}  \label{ex:roundbasic}
  When $(W,D)=(\CC,\{0\})$, the rounding map 
 $\tau_{\CC,\{0\}} \colon \CC_{\{0\}}\to \CC$ 
  becomes the change to polar coordinates map $\tau_{\CC, \{0 \}}$ from \eqref{eq:polcoord}. 
  When $(W,D)=(\CC^2, Z(xy))$, the rounding map
     \[  \tau_{\CC^2, Z(xy)} \colon \CC^2_{Z(xy)}\to \CC^2 \] 
 is simply the cartesian product of the 
 rounding maps $\tau_{\CC,\{0\}} \colon \CC_{\{0\}}\to \CC$ of the factors of $\CC^2$:
    \[  \begin{array}{ccc} 
               [0,+\infty)\times \bS^1 \times [0,+\infty)\times \bS^1 &   \to   & \CC^2 \\
                  (r_1, e^{i \theta_1}, r_2, e^{i \theta_2})   &  \to &   r_1e^{i \theta_1} r_2 e^{i \theta_2}
           \end{array}. \]
   \end{example}
  
\smallskip

Each monoid $(\cM_{W, D, x}, \cdot)$ from~\autoref{def:roundingPrelim} is the stalk at 
the point $x\in W$ of the sheaf of monoids $\cM_{W, D}$ on $W$ whose sections on 
  an open subset $U$ of $W$ containing $x$ are the holomorphic functions 
  on $U$ which are non-zero 
  outside $D$. Note that the sheaf $\cM_{W,D}$ comes with a canonical morphism of 
  sheaves of monoids 
     \[ (\cM_{W, D}, \cdot)  \to( \cO_W, \cdot)  \]
to the sheaf  $\cO_M$ of germs of holomorphic functions on $W$: it is simply the inclusion 
morphism. This morphism  identifies the corresponding subgroups of units. This is precisely 
  the defining property of a \emph{log structure}\index{log!structure} in the sense of Fontaine 
    and Illusie \cite{K 88} (see~\autoref{def:logspace} below). 
  The previous log structure is called the 
  {\em divisorial log structure}\index{log!structure!divisorial} induced by $D$ 
  (see~\autoref{def:divlogstruct} below).

      The notations of~\autoref{def:roundingPrelim} will not be used any further. 
      We chose them because they were sufficiently simple not to hinder the understanding 
      of the meaning of a divisorial log structure. 
  We will introduce other notations for divisorial log structures and for rounding 
  maps in {Definitions}~\ref{def:divlogstruct} and~\ref{def:rounding}, 
  believing that they are more adapted for a functorial manipulation of log structures:
    \[    \cO_W^{\star}(-  D) :=  \cM_{W, D}, \:  \:   \tau_{\cO_W^{\star}(-  D)} :=  \tau_{W,D}. \]

  \begin{remark}\label{rem:logGeomvsAG} 
      It is worth pointing out some differences between scheme-theoretic algebraic geometry 
      and  log geometry in the sense of Fontaine and Illusie. First, the algebraic basis of algebraic 
            geometry consists of the study of {\em rings}, their ideals and modules, whereas 
            the algebraic basis of log geometry involves {\em monoids}, 
            and the corresponding notions of ideals and modules 
            (see \cite[Sections I.1.2 and I.1.4]{O 18}).
        Second, assume we are given an algebraic variety $W$ and a hypersurface $D$ on it. 
        Then,
     \begin{enumerate}
        \item   Algebraic geometry assigns to this pair a  sheaf of ideals, whose sections 
            on an open subset of $W$ consist of the regular functions vanishing 
            \emph{at least on $D$}.
        \item  Fontaine and Illusie's log geometry assigns to $(W,D)$ a sheaf of monoids, 
           whose sections   on an open subset of $W$ consists of the regular functions vanishing 
          \emph{at most on $D$}.
     \end{enumerate}
  \autoref{rem:NotationsDivLogStr} is a consequence of this observation. 
  \end{remark}

\medskip
\subsection{Complex log spaces and their morphisms}   \label{ssec:logspacesmorph}  $\:$ 
\medskip

In~\autoref{ssec:loggeomround} we motivated the concept of 
a divisorial log structure through a coordinate-free version of the classical 
change to polar coordinates in $\CC$. In this subsection, we explain basic 
general definitions about {\em log spaces} and their {\em morphisms}, including 
{\em pre-log and log structures} (see \autoref{def:prelog}), 
{\em pullbacks and pushforwards of log structures} (see Definitions \ref{def:pb} 
and \ref{def:pf}),  \emph{divisorial log structures} (see \autoref{def:divlogstruct}), 
{\em toroidal log structures} (see \autoref{def:logtoricgen}), \emph{strict log morphisms} (see \autoref{def:strict}) and 
\emph{log enhancements} of suitable analytic morphisms of pairs (see \autoref{def:indlogmorph}).  
For further details, we refer the reader to Ogus' textbook~\cite{O 18}. 
\medskip

Kato's foundational paper \cite{K 88} on the subject develops log structures in the category 
of schemes, inspired by ideas of Fontaine and Illusie (see also~\cite[Definition III.1.1.1]{O 18}). 
Log structures in the complex analytic setting are discussed in~\cite[Section 1]{KN 99}. 
We will give the definitions 
for arbitrary ringed spaces, which will be assumed to be locally ringed.

The starting point for defining log structures is the notion of a {\em pre-logarithmic structure} 
(recall that {\em monoids} were introduced in \autoref{def:monoid}):

   \begin{definition}   \label{def:prelog}
        A \textbf{pre-logarithmic space}\index{pre-logarithmic!space}\index{space!pre-logarithmic} 
        $W$  is a  ringed space $\boxed{\underline{W}}$ 
        (called the \textbf{underlying ringed space} of the pre-logarithmic space), 
        endowed with a 
        sheaf of monoids $\cM_W$ and a morphism of sheaves of monoids 
        \[ \boxed{\alpha_W \colon \cM_W \to (\cO_{\underline{W}}, \cdot)}. \]  
        The pair $\boxed{(\cM_W, \alpha_W)}$ is called a \textbf{pre-logarithmic 
        structure}\index{pre-logarithmic!structure} on $\underline{W}$, or \textbf{pre-log structure} for short. 
        To simplify notation, we often write  $\cO_W$ instead 
        of $\cO_{\underline{W}}$. The pre-logarithmic space $W$ is called \textbf{complex} 
        (respectively, \textbf{complex analytic}) 
        if the underlying ringed space $\underline{W}$ is complex (respectively, complex analytic).
    \end{definition}

   A {\em log structure} is a pre-log structure satisfying a supplementary condition:

     \begin{definition}  \label{def:logspace}   
        A pre-logarithmic space $(W, \cM_W, \alpha_W)$ is called a 
        \textbf{logarithmic space}\index{space!logarithmic}, 
        or a \textbf{log space}\index{log!space}\index{space!log} for short 
        (and the associated pre-log structure is then called a 
        {\bf log structure}\index{log!structure})  
        if the morphism $\alpha_W$  induces an isomorphism 
        $\alpha_W^{-1}(\cO_W^{\star}) \simeq \cO_W^{\star}$. Here, $\boxed{\cO_{W}^{\star}}$ 
        denotes the sheaf 
        of units of $(W, \cO_W)$.       
        A complex (analytic) space endowed with a logarithmic structure is called a 
        \textbf{log complex (analytic) space}\index{log!complex space}. 
      \end{definition}

      \begin{remark}  \label{rem:isomequiv}
           The condition that $\alpha_W$  induces an isomorphism 
            $\alpha_W^{-1}(\cO_W^{\star}) \simeq \cO_W^{\star}$ is equivalent to the condition 
            that it induces an isomorphism $\cM_W^{\star} \simeq \cO_W^{\star}$ between the sheaves 
            of unit subgroups of the sheaves of monoids $\cM_W$ and $\cO_W$. 
      \end{remark}

     \begin{remark}\label{rem:logNotation}
        If a log structure on a complex space $W$ can be inferred from the context, 
        we  simplify notation and write  $\boxed{W^{\dagger}}$  for the corresponding 
        log space. The notation ``$W^{\dagger}$'' is borrowed from the book \cite{G 11}, 
        which surveys the Gross-Siebert program to study mirror symmetry with log geometry techniques. 
    \end{remark}

    \begin{remark}  \label{rem:namelog}
          Fontaine and Illusie's main motivations for introducing the notion of 
           a log space (in the context of schemes) can be found in~\cite{I 94}. 
           The terminology refers  to the fact that 
           a log structure gives rise to a canonical notion of \emph{sheaf of 
           differential forms with logarithmic poles}. 
           The term ``logarithmic'' hints also to the fact that the composition law in $\cM_W$ 
           can be viewed additively, i.e.,  $\alpha_W$ becomes an exponential map turning 
           sums into products.
   \end{remark} 
 
   Every ringed space can be endowed with two canonical log structures, which we now describe:

    \begin{definition} \label{def:tauttriv}
        Let  $(\underline{W}, \cO_{\underline{W}})$ be a ringed space. Its 
        \textbf{tautological log structure}\index{log!structure!tautological} 
        is given by the identity  morphism on $\cO_{\underline{W}}$ and its  
        \textbf{trivial log structure}\index{log!structure!trivial} 
        by the embedding $\cO_{\underline{W}}^{\star} \hookrightarrow \cO_{\underline{W}}$. 
    \end{definition}
    
\begin{remark}\label{rem:logCatFixedV} 
   Log  structures on a fixed ringed space form a category. 
   More precisely,  morphisms $\phi \colon (\cM,\alpha) \to (\cN, \beta)$  
   are morphisms of sheaves of monoids     
   $\varphi \colon \cM \to \cN$ compatible with the evaluation morphisms 
   $\alpha$ and $\beta$, i.e., $\alpha = \beta \circ  \varphi$.
    The trivial log structure is the initial object in this category, 
    whereas  the tautological log structure is its final object. 
\end{remark}

By definition,  any log structure on a ringed space $(\underline{W},\cO_{\underline{W}})$ 
is a pre-log structure. Thus, we have a natural inclusion functor:
\begin{equation}\label{eq:forgetful}
  \iota\colon \{\text{log structures on } \underline{W}\} \to \{\text{pre-log structures on } \underline{W}\}.
\end{equation}
Furthermore, $\iota$ admits a left adjoint $j$ by~\cite[Proposition III.1.1.3]{O 18}. More precisely, 
given a pre-log structure $(\cM_W,\alpha_W)$ on $\underline{W}$, 
its image $ \boxed{\cM_W^a} $ under $j$ (``$a$'' being the initial of ``associated'', see 
\autoref{def:standardLogStructure} below) 
is the push-out of the diagram of sheaves over $\underline{W}$:
\[
\xymatrix{
  \alpha_W^{-1}(\cO^{\star}_W) \ar[r] \ar[d]_{\alpha_W} &\cM_W\\
  \cO_W^{\star} &
  }
\]
where $\alpha_W^{-1}(\cO_W^{\star})$ is the inverse image sheaf under $\alpha_W$.  
It comes with a natural map $\boxed{\alpha_W^a} \colon \cM_W^a\to \cO_W$ sending 
$(s,t)$ to $\alpha_W(s)t$ for each $s\in \cM_W$ and $t\in \cO^{\star}_W$.
Thus, any pre-log structure on $W$ comes with a natural log structure, namely, 
its image under $j$ (see  \cite[(1.3)]{K 88} and
\cite[Proposition III.1.1.3]{O 18} for details). 

\begin{definition}\label{def:standardLogStructure} 
     We call $(\cM_W^a,\alpha_W^a)$ the \textbf{log structure associated 
     to the pre-log structure}\index{log!structure!associated} $(\cM_W, \alpha_W)$.
\end{definition}
\smallskip

    Log structures may be {\em pulled back} and {\em pushed forward}   (see \cite[Section 1.4]{K 88} 
    and \cite[Definition III.1.1.5]{O 18}): 
   
   \begin{definition}   \label{def:pb}
    Let $f \colon V \to W$ be a morphism of ringed spaces. Fix a log structure 
    $(\cM_W,\alpha_W)$  on $W$.
     The \textbf{pullback\index{log!structure!pullback} $\boxed{f^* \cM_W}$  of $\cM_W$ by $f$} 
      is the log structure on $V$ associated to the pre-log structure obtained 
      as the composition $f^{-1}(\cM_W) \xrightarrow{\alpha_W} f^{-1}(\cO_W) \to \cO_V$.  
      Here, $f^{-1}(\cM_V)$ is the inverse image sheaf, i.e., the sheafification of the presheaf 
      $U\mapsto \lim_{U'\supseteq f(U)} \cM_V(U')$ on $V$ where $U'\subset W$ 
      and $U\subset V$ are open.         
   \end{definition}
   
   \begin{definition}   \label{def:pf}
    Let $f \colon V \to W$ be a morphism of ringed spaces. Fix a log structure $(\cM_V,\alpha_V)$ 
     on $V$. 
     The \textbf{pushforward\index{log!structure!pushforward} 
     $ \boxed{f_* \cM_V} $  of $\cM_V$ by $f$} is 
     the fiber product of the morphisms of sheaves of monoids 
     $\cO_W \to f_+(\cO_V)$ and $f_+(\cM_V) \to f_+(\cO_V)$ on $W$, 
     endowed with the projection $\boxed{\rho_2}\colon f_* \cM_V \to  \cO_W$: 
          \[  \xymatrix{
                               f_* \cM_V
                                  \ar[rr]^{\rho_2}
                                  \ar[d]_{\rho_1} & 
                                                &   \cO_W   \ar[d] \\
                                    f_+(\cM_V)  \ar[rr]&  & f_+(\cO_V). }  \]
     Here, $f_+(\cM_V)$ and $f_+(\cO_V)$ denote the direct image sheaves 
     of $\cM_V$ and $\cO_V$ by $f$. 
   \end{definition}
   
   The pair $(f_* \cM_V , \rho_2)$ is a log structure on $W$.

     \begin{remark} \label{rem:restrictions}
                    If $f\colon V\hookrightarrow W$ is a closed immersion of analytic spaces, 
                 we say that $f^*\cM_W$ is  \textbf{the restriction\index{log!structure!restriction of} 
                 of $\cM_W$ to $V$}. 
                 For this reason, we often denote it by $\boxed{\cM_{W | V}}$.

                 This operation of restriction is thoroughly used in our proof of the Milnor fiber 
                 conjecture (see Steps (\ref{item:startlogpart}), (\ref{item:isomtorlogf}), 
                 (\ref{item:startlogparta}), (\ref{item:isomtorlogfa}) and (\ref{item:isomtoriclog}) 
                 of~\autoref{sec:stepsproof}). In turn, the operation of pushforward is used in 
                 \autoref{def:divlogstruct} below.
    \end{remark}

   In order to turn pre-log and log spaces into categories, morphisms must be appropriately 
   defined. We start with morphisms between  pre-log spaces, which are defined 
   using inverse image sheaves:

  \begin{definition}  \label{def:morphprelog}
        A \textbf{morphism $\phi \colon V \to W$ between pre-log spaces} is a pair 
      \[(\underline{\phi} \colon \underline{V}  \to \underline{W}, 
   \:  \phi^{\flat}\colon \underline{\phi}^{-1} (\cM_W) \to \cM_V),\]
   where 
   $\underline{\phi}$ is a morphism of ringed spaces and 
   $\phi^{\flat}$ is a morphism of sheaves of monoids on $V$, making the
    following diagram commute
      \begin{equation}\label{eq:logMorphisms}  
           \xymatrix{
              \underline{\phi}^{-1} (\cM_W)  \ar[r]^-{\phi^{\flat}} \ar[d]_{\underline{\phi}^{-1}\alpha_W}                       
                           & \cM_V \ar[d]^{\alpha_V} \\
               \underline{\phi}^{-1} (\cO_W)   \ar[r]
                          & \cO_V.}
      \end{equation}        
      The pre-log structure on $\underline{\phi}^{-1} (\cM_W)$ is given 
      by the composition $\alpha_V\circ \phi^{\flat}\colon \underline{\phi}^{-1} (\cM_W)\to \cO_V$.
   \end{definition}

\begin{definition} \label{def:morlogspaces}
   A \textbf{morphism of log spaces}, or 
   \textbf{log morphism}\index{log!morphism} for short, 
    is simply a morphism between the underlying pre-log spaces. 
   That is, the \textbf{category of log spaces}\index{log!category} is the full subcategory of 
   the category of pre-log spaces whose objects are the log spaces. 
\end{definition}

\begin{example}\label{ex:trivialStructures}   
     If two ringed spaces $V$ and $W$ are endowed with their trivial log structures 
   in the sense of~\autoref{def:tauttriv}, then a log morphism 
   $\phi \colon V \to W$ is simply a morphism of ringed spaces. 
\end{example}

Next, we define special morphisms of log spaces, namely, those that can be obtained by restricting log structures (see~\cite[Section III.1.2]{O 18}). They play a central role in the construction of roundings, as we will see in~\autoref{thm:torsorfibre} below.

     \begin{definition}  \label{def:strict}
         A morphism of log spaces $f \colon V \to W$ is called 
         \textbf{strict}\index{strict log morphism}\index{log!strict morphism} 
         if it establishes an isomorphism $f^* \cM_W \simeq \cM_V$. 
     \end{definition}

     As we saw in~\autoref{ssec:loggeomround} through the example of the passage to polar coordinates, 
     special types of log structures on complex analytic varieties may be built 
     using reduced divisors (see \autoref{def:roundingPrelim}). 
     We reformulate now that definition using the operation of pushforward:

 \begin{definition}  \label{def:divlogstruct}
         If $D$ is a reduced divisor on a complex analytic variety $W$, 
         its \textbf{associated divisorial log structure}\index{log!structure!divisorial}  
                    $ \boxed{\cO_W^{\star}(-  D)} $
         is the pushforward of the trivial log structure on 
         $W \setminus D$ by the inclusion  $W \setminus D \hookrightarrow W$. 
         More precisely, its monoid of sections on an open set $U$ of $W$ consists of the 
         holomorphic functions defined on $U$ {\em which do not vanish outside $D$}. 
         If $V \hookrightarrow W$ is an embedding, then we write 
         $\boxed{\cO_{W| V}^{\star}(-  D)}$ for the restriction of $\cO_W^{\star}(-  D)$ to $V$, 
         following~\autoref{rem:restrictions}.
    \end{definition}
    
    For the role of divisorial log structures in the proof of~\autoref{conj:MFC}, we refer to 
     Steps (\ref{item:isomlogf}) and (\ref{item:isomtorlogf}) of~\autoref{sec:stepsproof}.

\begin{remark}\label{rem:NotationsDivLogStr} 
   The notation ``$\cO_W^{\star}(-  D)$'' is not standard. We chose it by analogy  
   to the classical notation   
    ``$\cO_W(-D)$'' for the sheaf of holomorphic functions vanishing at least along $D$ 
    (keeping in mind that, as we emphasized in~\autoref{rem:logGeomvsAG}, 
    sections of $\cO_W^{\star}(-  D)$ are not allowed to vanish outside 
    $D$, unlike for $\cO_W(-D)$).
    Other notations used in the literature are ``$\cM_{(W \setminus D) | V}$'' 
    (see \cite[Section III.1.6]{O 18}) and ``$\cM_{(W,D)}$'' (see \cite[Example 3.8]{G 11} 
    or \cite[Example 1.6]{A 21}).  It is worth pointing out that, unlike what happens to the sheaf 
    $\cO_W(-D)$, no new object arises from $\cO_W^{\star}(-  D)$ if we consider 
    non-reduced divisors. In short, $\cO_W^{\star}(-  D)$ depends only on the support of $D$.
\end{remark}

Toroidal varieties (see~\autoref{def:toroimorph}) can be equipped with canonical divisorial log structures as follows:

   \begin{definition}  \label{def:logtoricgen}
        A \textbf{toroidal log structure}\index{log!structure!toroidal} is a divisorial log structure of the form  
        $\cO_W^{\star}(- \partial W)$, where $(W, \partial W)$ is a toroidal variety. 
        A variety endowed with a toroidal log 
        structure is called \textbf{log toroidal}\index{variety!log toroidal}. 
   \end{definition}

   In the same way as divisors determine log structures, particular kinds of morphisms between 
   varieties endowed with divisors determine log morphisms. Indeed, let $V$ and $W$ 
   be two complex  analytic varieties and let $\cD_V$ and $\cD_W$ be two reduced divisors 
   on them. Let $f\colon V \to W$ be a complex morphism such that 
   the following inclusion holds:
       \[  f^{-1} (\cD_W) \subseteq \cD_V. \]
   Then, the pullback by $f$ of any section of $\cO_W^{\star}(- \cD_W)$ is a section 
   of $\cO_V^{\star}(-\cD_V)$. Since this pullback commutes with the tautological 
   inclusion morphisms 
   $\cO_V^{\star}(-\cD_V) \to \cO_V$ and $\cO_W^{\star}(-\cD_W) \to \cO_W$, it induces  
   a log morphism  between the corresponding log toroidal varieties.
   The following terminology summarizes this construction:

   \begin{definition}  \label{def:indlogmorph}
       Let $V, W$ be two complex analytic varieties and $\cD_V, \cD_W$ be two reduced divisors 
       on them. Let $f\colon V \to W$ be a complex morphism such that 
       $f^{-1} (\cD_W) \subseteq \cD_V$. Then, the log morphism 
       \[\boxed{f^{\dagger}}\colon (V, \cO_V^{\star}(- \cD_V))\to 
       (W, \cO_W^{\star}(- \cD_W))\] obtained by pullback via $f$ is called the 
        \textbf{log enhancement\index{log!enhancement}\index{enhancement of a morphism} 
        of $f$ associated to the divisors $\cD_V$ and $\cD_W$}. 
   \end{definition}

   \begin{remark}  In our proof of~\autoref{conj:MFC} 
     we consider log enhancements of 
     morphisms of the form $\tilde{f}\colon \tilde{Y} \to \D$, where $f\colon Y \to \D$ is 
     a smoothing of a splice type singularity, 
     $\pi \colon \tilde{Y} \to Y$ is a quasi-toroidalization of $f$ in the sense 
     of~\autoref{def:quasitorsmoothing} and  $\tilde{f} := f \circ \pi$. 
     Such log enhancements feature in Steps (\ref{item:startlogpart}), 
     (\ref{item:roundlogenhancmt}), (\ref{item:isomtorlogf}), (\ref{item:startlogparta}), 
     (\ref{item:isomtorlogfa}) and (\ref{item:commuttrianga}) of our proof. 
   \end{remark}
   
\autoref{prop:restoroid} has the 
   following important consequence:  $\partial$-transversal 
   subvarieties of toroidal varieties acquire log-theoretic properties when intersecting 
   the input subvariety with a quasi-toroidal 
   subboundary of the ambient space. Indeed, we prove:
   
    \begin{proposition}  \label{prop:stricttransv}
         Let $(W, \partial W)$ be a toroidal variety and let $\cD_W$ 
         be a quasi-toroidal subboundary of it. 
          Consider a $\partial$-transversal subvariety $V$ 
         of $(W, \partial W)$ as in~\autoref{def:bdrytransv} and write 
         $\partial V := V \cap \partial W$ and $\cD_V := V \cap \cD_W$. Then:
            \begin{enumerate}
                \item  The subvariety $\cD_V$ of $V$ is a quasi-toroidal subboundary of  
                    $(V, \partial V)$. 
                \item The log enhancement of the embedding $V \hookrightarrow W$ 
                    as in~\autoref{def:indlogmorph} relative 
                    to the divisors $\cD_V$ and $\cD_W$ is strict in the sense of~\autoref{def:strict}.
            \end{enumerate}
     \end{proposition}

\subsection{Types of monoids and charts of log structures}   \label{ssec:chartslogstruct}  $\:$ 
\medskip

   In this subsection we introduce terminology for various types of commutative monoids 
  and we explain the notion of {\em chart} for a log structure, which is an analog of the 
  usual notion of chart in differential geometry.
\medskip

 \begin{definition} \label{def:groupif} Let  $(P, +)$ be a monoid. 
         The {\bf Grothendieck group}\index{Grothendieck group} 
         $\boxed{(P^{gp}, +)}$ generated by it 
         is the set of formal differences 
         $m_1 - m_2$ of elements of $P$ modulo the equivalence relation: 
        \[   m_1 - m_2 \equiv n_1 - n_2 \:   \Longleftrightarrow \: \mbox{ there exists } \:  
            p \in P \text{ satisfying }    m_1 + n_2 + p = m_2 + n_1 + p \]
        and endowed with the obvious addition:
            \[  (m_1 - m_2) + (m'_1 - m'_2) := (m_1 + m'_1) - (m_2 + m'_2). \] 
   The \textbf{group of units} $\boxed{P^{\star} }$ of the monoid $P$ is its maximal subgroup. 
  \end{definition} 
  
\noindent The group $P^{gp}$ is also called the \emph{groupification} or the \emph{group hull}  of $P$. 
   It is endowed with a natural morphism of monoids 
 $P \to P^{gp}$. The nature of this morphism  determines  special classes 
 of monoids (see \cite[Definition I.1.3.1]{O 18}). More precisely:
 
 \begin{definition}   \label{def:toricmon}
     A monoid\index{monoid} $(P, +)$  is called:
        \begin{enumerate}
            \item \textbf{integral}\index{monoid!integral} or \textbf{cancellative}\index{monoid!cancellative} 
               if the natural monoid morphism 
                $P \to P^{gp}$ is injective, 
                that is, if the implication 
                   \[  m + m' = m + m''  \:  \Longrightarrow 
                   \:  m' = m''  \]
                 holds for every $m, m', m'' \in P$; 
            \item \textbf{unit-integral}\index{monoid!unit-integral} 
                if the natural group morphism $P^{\star} \to P^{gp}$ is injective; 
            \item \textbf{saturated}\index{monoid!saturated} if it is integral and the  implication 
                  \[  q \,m  \in P \: \Longrightarrow \: m \in P \]
                  holds whenever 
                $m \in P^{gp}$ and $q \in \N^*$ (here, 
                 $\boxed{q\,m}:= \underbrace{m+\ldots +m}_{q \text{ times}}$);
            \item \textbf{fine}\index{monoid!fine} if it is integral and finitely generated; 
            \item \textbf{toric}\index{monoid!toric} 
               if it is fine and $P^{gp}$ is a lattice, that is, a free abelian group 
               of finite rank. 
        \end{enumerate}
 \end{definition}
 
\begin{remark} 
    Note that the toric monoids are exactly the monoids of characters of 
    affine toric varieties. Those varieties are normal if, and only if,  the toric monoid is saturated.  
\end{remark}

Just as local charts are essential to  do computations in differential  geometry,  the notion of a \emph{chart} of a log structure is crucial to study log structures locally. 
The definition of a chart is based on the construction of  log structures associated to pre-log structures (see~\autoref{def:standardLogStructure}).
A chart depends on the choice of a monoid.

Before formally defining charts of a log space $(W, \cM_W)$ (following~\cite[Definition (2.9)]{K 88} 
and \cite[Sections II.2.1, III.1.2]{O 18}), we need some auxiliary notation. For any monoid  
$P$, giving a morphism of monoids $P \to \Gamma(W, \cM_W)$ 
from $P$ to the monoid of global sections of the sheaf $\cM_W$ 
is equivalent to giving a morphism
of sheaves of monoids $P_W \to \cM_W$. Here, $\boxed{P_W}$ denotes the 
{\em constant sheaf associated to $P$}, that is, the sheaf associated 
to the presheaf that takes each open set of $W$ to $P$ and whose restriction maps 
are identities. Strictly speaking, it should be called the \emph{locally constant sheaf 
associated to $P$}, but tradition established the shorter name. 
By composing this morphism of sheaves with the structure map
$\alpha_W \colon \cM_W \to \cO_W$ of the log space $W$ we get a pre-log structure
$P_W \to \cO_W$. Its associated log structure  $\boxed{P_W^a}$ in the sense of 
\autoref{def:standardLogStructure} 
comes equipped with a morphism of log structures $P_W^a \to \cM_W$.

  \begin{definition}  \label{def:chart}
        Let $(W, \cM_W)$ be a log space and $P$ a monoid. 
        A \textbf{chart}\index{chart of a log space}\index{log!chart} 
        for $W$ \textbf{subordinate to $P$} is a morphism $P_W \to \cM_W$
        of sheaves of monoids such that the induced morphism $P_W^a \to \cM_W$
        of log structures is an isomorphism.

        If the monoid $P$ is finitely generated, then the chart
        is called \textbf{coherent}\index{coherent!chart}. If $P$ is fine/toric 
        (in the sense of~\autoref{def:toricmon}), then the chart is called 
        \textbf{fine/toric}\index{fine!chart}\index{toric!chart}.

        A log space which admits a coherent chart
        in a neighborhood of every point is called \textbf{coherent}\index{coherent!log space}. 
        If, moreover, such charts may be chosen to be fine/toric, then the log space 
        or structure is called \textbf{fine/toric}\index{fine!log space}\index{toric!log space}. 
  \end{definition}

  \begin{remark} 
      A simple check confirms that toroidal log structures 
      in the sense of~\autoref{def:logtoricgen} are toric, therefore coherent. 
      As with toric varieties, charts in neighborhoods of distinct points can 
      be subordinate to different monoids. 
      Indeed, if $W$ is a complex affine toric variety associated to a toric monoid $P$, 
      whose set of closed points is $\Hom(P, \CC)$, then the natural morphism of monoids 
      $P \to \Gamma(W, \cO_{W}^{\star}(-   \partial W))$ 
      is a chart whose domain is the whole variety $W$. Furthermore,  the 
      monoid $P/ P^{\star}$ can be reconstructed from the toric log space 
      $(W, \cO_{W}^{\star}(-   \partial W))$ as the  quotient of the monoid 
      of germs of sections $\cO_{W}^{\star}(-   \partial W)_o$ at the unique closed 
      orbit $o$ of $W$ by its subgroup of units $(\cO_{W}^{\star}(-   \partial W)_o)^{\star}$.
  \end{remark}

Our proof of~\autoref{conj:MFC} involves  divisorial log structures which are defined 
by quasi-toroidal subboundaries in the sense of 
  ~\autoref{def:face-subdiv}. The associated divisorial log structures are not necessarily 
   coherent, but they are \emph{relatively coherent} as defined by Nakayama and Ogus 
   in~\cite[Definition 3.6]{NO 10}. 
   In the context of toroidal varieties,  relatively coherent divisorial log structures
   correspond exactly to quasi-toroidal subboundaries, as our next result asserts: 
      
   \begin{proposition}   \label{prop:toroidcharrelcoh}
       Let $(W, \partial W)$ be a toroidal variety and $\cD_W$ be a subdivisor of $\partial W$. 
       Then, $(W, \cO_W^{\star}( -   \cD_W))$ is relatively coherent in 
       $(W, \cO_W^{\star}( -   \partial W))$ if, and only if,  $\cD_W$  is 
       a quasi-toroidal subboundary of $(W, \partial W)$. 
   \end{proposition}

\begin{remark}\label{rem:whyRelCoh}   
    Relative coherence plays a crucial role in  Nakayama and Ogus' local 
   triviality theorem (see~\autoref{thm:logehresm} below), as the source of a relatively log smooth 
   morphism is relatively coherent by hypothesis. We use this local triviality 
   and~\autoref{cor:quasitorrelcoh} to produce canonical representatives 
   of the Milnor fibrations over the circle associated to the quasi-toroidalizations of a given smoothing 
   (see Steps (\ref{item:roundlogenhancmt}) and (\ref{item:combprevsteps})). 
\end{remark}

\medskip
\subsection{Kato and Nakayama's rounding operation}   \label{ssec:KNrounding}  $\:$ 
\medskip

In~\autoref{ssec:loggeomround} we introduced \emph{rounding maps} by analogy with  
 the classical passage to polar coordinates (see \autoref{def:roundingPrelim}). In this subsection we give further details on this construction and discuss its functoriality properties. Throughout, 
 a {\em cartesian diagram} of topological spaces denotes a pullback or fiber product diagram in 
 the topological category.

      \medskip
    The following definition of the {\em rounding} of a log space is a slight reformulation 
    of Kato and Nakayama's generalization of the 
    real oriented blowup operation 
    given in~\cite[Section 1]{KN 99} for log complex analytic spaces 
    (see also~\cite[Definition V.1.2.4]{O 18}).     Alternative descriptions of this operation 
    can be found in~\cite[Section 1.2]{IKN 05} and~\cite[Section 1.1]{A 21}.
    A useful example to keep in mind  is the passage to polar coordinates on the log space 
    $(\CC,\cO^\star(-\{0\}))$, discussed in~\autoref{ssec:loggeomround}.
   For a  comparison with A'Campo's   
  classical real oriented blowups we refer the reader to~\cite{Gi 11}.

    \begin{definition} \label{def:rounding}
        Let $(W,\cM_W, \alpha_W)$ be a log complex space in the sense of~\autoref{def:logspace}. 
        We identify  the sheaves  $\cM_W^{\star}$ and $\cO_W^{\star}$ via the map  
        $\alpha_W$ (see \autoref{rem:isomequiv}). 
                          The \textbf{rounding}\index{rounding}     of $W$ is the set
                         \[ \boxed{W_{\log}}:= \left\{ (x,u), x \in W,   u \in \Hom(\cM_{W,x}, \bS^1)  , 
                               u(\alpha_{W,x}(f)) = \mbox{arg}(f(x)), \: \forall \: f \in 
                                   \cM_{W,x}^{\star}  = \cO_{W,x}^{\star} 
                               \right\}, \]
                               where $\mbox{arg}(s) = s/|s|$ for each $s\in \CC^*$ (see 
                               \autoref{def:logCoordinatesPolar}).  
                    The \textbf{rounding map}\index{rounding!map} is the function 
                        \[ \begin{array}{cccc}
                              \boxed{\tau_W} \colon &  W_{\log} & \to & \underline{W}  \\
                                & (x,u) & \to & x. 
                           \end{array}   \]
     The rounding $W_{\log}$ is endowed with the weakest topology making continuous 
     the rounding map $\tau_W$ and the set of maps 
            \[ \{\mbox{arg}(m)\in \Hom(\tau_W^{-1}(U), \bS^1): 
         U\subset \underline{W} \text{ open }, m\in \cM_W(U)\}, \]
     where: 
  \[  \begin{array}{cccc}
                   \boxed{\mbox{arg}(m)}\ \colon &   \tau_W^{-1}(U) & \to & \bS^1  \\
                                & (x,u) & \to & u(m_x). 
       \end{array} 
   \]
    \end{definition}

     \begin{remark}  \label{rem:terminbettirealiz} 
          The terminology ``\emph{rounding}'' was coined by Ogus   
        (see \cite{ACGHOSS 13, NO 10}) and refers to the fact that whenever $W$ 
        is a  \emph{fine} log space in the sense  of~\autoref{def:chart}, the fibers of the rounding 
        map $\tau_W$ are finite disjoint unions of compact tori, which are product 
        of circles, and thus, prototypical ``round'' geometric objects (see~\autoref{thm:torsorfibre}). 
        Alternative names in the literature are 
        ``\emph{Kato-Nakayama Space}''\index{space!Kato-Nakayama} 
        (see \cite{A 21,TV 18}) or ``\emph{Betti realization}''\index{Betti realization}, again a 
        terminology due to Ogus (see~\cite{O 18}).
     \end{remark}

    The next result discusses functoriality properties of the rounding operation. For a proof when $W$ is a log complex analytic space, we refer to \cite[Proposition V.1.2.5]{O 18}. The same proof is valid for  arbitrary log complex  spaces:
    
    \begin{theorem}  \label{thm:torsorfibre}
         Assume that $(W,\cM_W, \alpha_W)$ is a log complex space. 
             \begin{enumerate}
                   \item The rounding map $\tau_W$ is continuous.  It is a 
                       homeomorphism whenever the log structure of $W$ is trivial. 
                       
                   \item   \label{fibrdescr}
                       Let $x$ be a point of $\underline{W}$ and consider the abelian group 
                     \begin{equation}\label{eq:Tx}  
                       \boxed{T_x}:= \Hom({\cM}_{W,x}/{\cM}_{W,x}^{\star}, \bS^1).
                     \end{equation}
                     Then, $T_x$ acts naturally on the fiber $\tau_W^{-1}(x)$ by extending the natural 
                     action on $\Hom(\cM_{W,x}, \bS^1)$, i.e.:
                     \[ (\beta \cdot u)(m) = \beta(\overline{m}) u(m) \quad \text{ for } 
                             \beta\in T_x, \;u\in \Hom(\cM_{W,x}, \bS^1), \;m\in {\cM}_{W,x},
                     \]
                     where $\overline{m}$ is the coset of $m$ in $\cM_{W,x}/\cM_{W,x}^{\star}$.
                     This action 
                         defines a torsor if the monoid $\cM_{W,x}$ is unit-integral in the sense 
                         of~\autoref{def:toricmon}. 
                         In particular, $\tau_W$ is surjective if $\cM_W$ has only unit-integral stalks. 
                         This occurs, for instance, if $W$ is a fine log space.
                         
                   \item \label{functbr} 
                       The construction of $W_{\log}$ is functorial and the morphism 
                      $\tau_W$ is natural. More precisely, a morphism $f \colon V \to W$ 
                      of complex log spaces 
                      induces a morphism of topological spaces 
                      $\boxed{f_{\log}} \colon V_{\log} \to W_{\log}$, 
                      called the \textbf{rounding of $f$}\index{rounding}, which fits in a commutative diagram: 
                          \begin{equation}\label{eq:roundingOfF}  
                            \xymatrix{
                                V_{\log}     \ar[r]^{f_{\log}}   \ar[d]_{\tau_V} &   W_{\log} \ar[d]^{\tau_W} \\
                                \underline{V}\ar[r]_{\underline{f}}                                  &   \underline{W}.}
                          \end{equation}
                          Thus, the rounding operation is a \emph{covariant functor} from the category 
                          of log spaces to the category of topological spaces.
                          
                     \item   \label{cartdiag}
                         The diagram~\eqref{eq:roundingOfF} is cartesian (in the topological category) 
                         whenever the log morphism $f$ is strict in the sense of~\autoref{def:strict}. 
             \end{enumerate}
    \end{theorem}
    
 \begin{remark}\label{rem:ConnectedFibers}    
    Note that whenever $(W,\cM_W, \alpha_W)$ 
   is a fine log space in the sense  of~\autoref{def:chart}, the  monoid 
   ${\cM}_{W,x}/{\cM}_{W,x}^{\star}$ appearing in~\autoref{thm:torsorfibre}  (\ref{fibrdescr}) is fine. 
   Consequently, its Grothendieck group 
   $({\cM}_{W,x}/{\cM}_{W,x}^{\star})^{gp}$
   is finitely generated, thus a direct sum of a finite abelian group and a lattice. Therefore, 
   the group $T_x$ from~\eqref{eq:Tx}
    is a finite disjoint union of compact tori (that is, of groups isomorphic to 
    $(\bS^1)^n$ for some $n \in \N$). 
    As a consequence of~\autoref{thm:torsorfibre}, the fiber $\tau_W^{-1}(x)$ 
    is connected (that is, it is a single torus) if, and only if,
    the group $({\cM}_{W,x}/{\cM}_{W,x}^{\star})^{gp}$ is a lattice. This is always the case 
    when $(W,\cM_W)$ is a toric log space 
    in the sense  of~\autoref{def:chart} (see \cite[Proposition II.2.3.7]{O 18}). 
    Notice that even if the toric monoid is not saturated, its associated group is still a lattice: 
    it is the lattice of exponents of  monomials. 
 \end{remark}

\autoref{thm:torsorfibre}~(\ref{cartdiag})  has an important consequence:

    \begin{corollary}  Let $W$ be a complex log space and let $\underline{V}
    \hookrightarrow\underline{W}$ be a subspace of the underlying topological space. 
    Endow $\underline{V}$ with a  log structure obtained by restricting the log structure of $W$. 
    Then, $\tau_V$ is the restriction of 
      $\tau_W$ to the subspace $V_{\log}$.
    \end{corollary}

    The next result characterizes topological boundaries of roundings of log toroidal varieties in the sense of~\autoref{def:logtoricgen}.

    \begin{proposition}  \label{prop:topmantoroid}
        Assume that $W$ is a log toroidal variety.  Then, $W_{\log}$ is a 
        real semi-analytic variety homeomorphic to a topological manifold with boundary. 
        Its topological boundary $\partial_{top} (W_{\log})$ 
        is the preimage 
        of the toroidal boundary $\partial \underline{W}$ of $\underline{W}$ under 
        the rounding map $\tau_W$.        
    \end{proposition} 
    
    Furthermore, it can be shown that $W_{\log}$ is a ``manifold with generalized corners'' 
    in the sense of Joyce~\cite{J 16} (see also~\cite{GM 15,KM 15}).    
    The statement can be proven locally since open sets of affine toric varieties serve 
    as local models for toroidal varieties.  The topological part of the statement can be found 
    in~\cite[Lemma 1.2]{KN 08}, and its extension to the semi-analytic category 
    is straightforward. \autoref{thm:logehresm} in the next subsection complements 
    this result by extending it to morphisms.

    The next result is a slight generalization of~\autoref{thm:torsorfibre}~(\ref{cartdiag}). 
    It can be proved using  the classical {\em pullback lemma} of abstract category theory   
    (see~\cite[Lemma 5.8]{A 10} or \cite[Exercise III.4.8]{M 98}). It plays a crucial role in 
    Steps~(\ref{item:roundlogenhancmt}) and (\ref{item:startlogparta})
    of the proof of~\autoref{conj:MFC}.

     \begin{proposition}   \label{prop:strictparallel}
        Fix the following commutative diagram of log morphisms between log complex spaces 
           \begin{equation} \label{eq:cartdiaglog}  
                     \xymatrix{
                           X    \ar[r] \ar[d] &      Y \ar[d] \\
                            V   \ar[r] & W.}
            \end{equation}
           Assume  
           and that either its two vertical or its two horizontal arrows are strict 
           and that the underlying commutative diagram of topological spaces is cartesian.
           Then, the commutative diagram             \[  
                     \xymatrix{
                           X_{\log} \ar[r] \ar[d] &    Y_{\log}\ar[d] \\
                            V_{\log}   \ar[r]  &   W_{\log}}
                           \]
           obtained by rounding (\ref{eq:cartdiaglog}) is cartesian in the topological category.                 
     \end{proposition}

  \medskip
  \subsection{Nakayama and Ogus' local triviality theorem} 
  \label{ssec:loctrivquasitor}  $\:$
  \medskip

  In this subsection, we discuss Nakayama and Ogus' local triviality theorem 
  (see \autoref{thm:logehresm})
  and two of its consequences (see Corollaries \ref{cor:quasitorrelcoh} and \ref{cor:milntubelog}), 
  expressed in the language of quasi-toroidal subboundaries.
  As stated in~\autoref{rem:whyRelCoh}, these results are essential to confirm that 
  one obtains canonical representatives of Milnor fibrations over a compact 
  two-dimensional disk $\D$ centered at the origin of $\CC$, from 
  quasi-toroidalizations of smoothings of the input splice type surface singularities. 
  \medskip

       Using Siebenmann's topological local triviality theorem from~\cite[Corollary 6.14]{S 72},   
       Nakayama and Ogus proved  the following log version of 
       Ehresmann's theorem (see~\cite[Theorems 3.5 and 5.1]{NO 10}). 
       We will not give precise definitions of several terms involved in the statement 
       (\emph{relative coherence, separated, exact and relatively log smooth 
       morphisms, points where a morphism is vertical}), 
       since our interest in this result lies in one of its consequences, 
       namely,~\autoref{cor:milntubelog}  discussed below.
    
    \begin{theorem}   \label{thm:logehresm}
         Let $f\colon V \to W$ be a morphism of log complex analytic spaces, where $W$ is fine and 
         $V$ is relatively coherent. Assume that $f$ is proper, separated, exact 
         and relatively log smooth. Then,
         its rounding $f_{\log} \colon V_{\log} \to W_{\log}$ is a locally trivial 
         fibration whose fibers 
         are oriented topological manifolds with boundary. The union of the boundaries of the fibers  
         consists of those points of $V_{\log}$ sent by the rounding map 
         $\tau_V\colon V_{\log}\to \underline{V}$ to points of $\underline{V}$ where $f$ is not vertical.
    \end{theorem}

     \begin{remark}\label{rem:logehresm} 
         \autoref{thm:logehresm} generalizes earlier work of Kawamata
       concerning the structure of real oriented  blowups of 
      proper surjective toroidal and equidimensional morphisms of quasi-smooth toroidal varieties 
      (see~\cite[Theorem 2.4]{K 02}).
             Kawamata's definition of a real oriented blowup is a generalization of A'Campo's notion 
       with the same name (see~\autoref{rem:evolrealblowup}). But while  A'Campo's 
       original construction for normal 
       crossings divisors in smooth complex varieties  uses line bundles, 
       Kawamata's approach is to glue local models for quasi-smooth toroidal varieties 
       built from  \emph{simplicial} affine toric varieties, thus avoiding the use of log structures altogether.
    \end{remark}
      
 The next corollary to~\autoref{thm:logehresm} can be proved by
       translating the notions of \emph{relative coherence}, 
       {\em separatedness, exactness}, \emph{log smoothness}, 
       \emph{relative log smoothness} and \emph{verticality} 
       into the toroidal language when the target is the standard log disk 
       $(\D, \cO_{\D}^{\star}(- \{ 0 \}))$, and by using~\autoref{prop:toroidcharrelcoh}:

       \begin{corollary}  \label{cor:quasitorrelcoh} 
             Let $ \tilde{f} \colon V \to \D$ be a proper complex analytic morphism from 
             a complex analytic variety $V$ to an open disk $\D$ of $\CC$ centered at the origin.
             Let $\cD_V$ be a reduced divisor on $V$ such that 
             the complement $V \setminus \cD_V $ is smooth and with $  \tilde{f}^{-1}(0)  \subseteq  \cD_V $. 
         Choose the following  log enhancement  of $\phi$, 
         in the sense of~\autoref{def:indlogmorph}: 
                    \[   \tilde{f}^{\dagger} \colon (V, \cO^{\star}_V(- \cD_V)) \to (\D, \cO_{\D}^{\star}(- \{ 0 \})).   \]
         Assume that there exists a reduced divisor $\partial V$ of $V$ with the property that 
            $(V, \partial V)$ is toroidal, $\cD_V$ is a quasi-toroidal subboundary of $(V,\partial V)$ 
            in the sense of~\autoref{def:face-subdiv} 
                      and the morphism $ \tilde{f} \colon (V, \partial V) \to (\D, 0)$ is toroidal.             
          Then, the morphism of topological spaces 
              \[  \tilde{f}^{\dagger}_{\log} \colon (V, \cO_V^{\star}(-   \cD_V))_{\log} 
                    \to (\D, \cO_{\D}^{\star}(- \{ 0 \}))_{\log}  \]
          obtained by taking the rounding of $ \tilde{f}^{\dagger}$,  
          is a locally trivial topological fibration whose fibers are manifolds 
          with boundary. The union of the boundaries of the fibers 
         consists of those points of $ (V, \cO_V^{\star}(-   \cD_V))_{\log}$ 
         sent by the rounding map $(V, \cO_V^{\star}(-   \cD_V))_{\log} \to V$   
         to points $x$ of $V$ for which the germ $(\cD_V)_x$  strictly contains
         the germ $( \tilde{f}^{-1}(0))_x$. 
       \end{corollary}

      \autoref{cor:quasitorrelcoh} can be used as a tool to study Milnor fibers of smoothings 
      of isolated complex singularities. More precisely, if  $f \colon (Y,0) \to (\CC,0)$ 
      is such a smoothing, we consider a quasi-toroidalization  $\pi \colon \tilde{Y} \to Y$ of it 
      in the sense of~\autoref{def:quasitorsmoothing}, and we aim to 
      apply~\autoref{cor:quasitorrelcoh} to the triple
      $V := \tilde{Y}$, $\tilde{f} := f \circ \pi$, and $\cD_V := \tilde{f}^{-1}(0)$.   
       In order to achieve \emph{properness} of $\tilde{f}$, we work with a 
       \textbf{Milnor tube representative}\index{Milnor!tube representative}  
       of $f$. Such a representative is obtained by first considering the part of a representative 
       of $(Y,0)$ contained in a Milnor ball, and then restricting this set further to the preimage 
       by $f$ of a sufficiently small Euclidean disk $\D$ centered at the origin of $\CC$.
       There is a slight difference between this setting and that of~\autoref{cor:quasitorrelcoh}, 
       as  $\tilde{Y}$ has a topological boundary. However, since $\tilde{f}$ 
       is locally trivial near that boundary, 
       it is straightforward to show that~\autoref{cor:quasitorrelcoh} 
       generalizes to this slightly broader context:

        \begin{corollary}  \label{cor:milntubelog} 
             Let $f \colon Y \to \D$ be a Milnor tube representative of a smoothing. 
             Let $\pi \colon \tilde{Y} \to Y$ be a quasi-toroidalization of it 
             and $ \tilde{f} := f \circ \pi$ 
             be the lift of $f$ to $\tilde{Y}$. Fix $\cD_{ \tilde{Y}} := \tilde{f}^{-1}(0)$ and consider 
             the following log enhancement  of $\tilde{f}$,  in the sense of~\autoref{def:indlogmorph}: 
                    \[  \tilde{f}^{\dagger} \colon (\tilde{Y}, \cO^{\star}_{ \tilde{Y}}(- \cD_{ \tilde{Y}})) 
                             \to (\D, \cO_{\D}^{\star}(- \{ 0 \})).   \]
         Assume that there exists a reduced divisor $\partial  \tilde{Y}$ of $ \tilde{Y}$ with the property that 
            $( \tilde{Y}, \partial  \tilde{Y})$ is toroidal, $\cD_{\tilde{Y}}$ is a quasi-toroidal subboundary of 
            $( \tilde{Y},\partial  \tilde{Y})$ 
            in the sense of~\autoref{def:face-subdiv} 
                      and the morphism $\tilde{f} \colon ( \tilde{Y}, \partial  \tilde{Y}) \to (\D, 0)$ is toroidal.             
          Then the morphism of topological spaces 
              \[ \tilde{f}^{\dagger}_{\log} \colon ( \tilde{Y}, \cO_{ \tilde{Y}}^{\star}(-   \cD_{ \tilde{Y}}))_{\log} 
                    \to (\D, \cO_{\D}^{\star}(-  \{ 0 \}))_{\log}  \]
          obtained by taking the rounding of $\tilde{f}^{\dagger}$  
          is a locally trivial topological fibration whose fibers are manifolds 
          with boundary homeomorphic to the Milnor fibers of the 
          smoothing $f$. Moreover, the restriction of this fibration to the boundary 
          circle $(0, \cO_{\D | 0}^{\star}(- \{ 0 \}))_{\log} $ of the cylinder 
          $(\D, \cO_{\D}^{\star}(- \{ 0 \}))_{\log}$  
          is isomorphic to the Milnor fibration of $f$ over the circle.
       \end{corollary}

In the context of the Milnor fiber conjecture, we apply~\autoref{cor:milntubelog} 
to  smoothings of three different singularities: the input splice type surface singularity, 
       and the $a$- and $b$-side 
       singularities, whose associated splice diagrams are obtained by cutting the starting 
       splice diagram at an internal point of the edge $[a,b]$ 
       (see Steps (\ref{item:startlogpart}) and (\ref{item:startlogparta}) of~\autoref{sec:stepsproof}). 
       In this case, $\partial  \tilde{Y}$ is the intersection of 
                 $\tilde{Y}$ with the toric boundary of its ambient toric variety, and similarly 
                 for the $a$ and $b$ sides.
       In addition, in order to get representatives of the \emph{cut Milnor fibers} 
       appearing in the definition of the four-dimensional splicing 
       operation (see \autoref{def:fourdimsplice}), we use an analog of~\autoref{cor:milntubelog} 
       in which $\cD_{\tilde{Y}}$  strictly contains $\tilde{f}^{-1}(0)$ 
       (see the last paragraph of \autoref{ssec:ourusequasitor}).
     
     \medskip

\section{Tropical ingredients}   \label{sec:loctropNND}

In this section we elaborate on  the tropical techniques used in our proof of the 
Milnor fiber conjecture, discussed already in \autoref{ssec:loctrop}. 
We explain the notions of \emph{(positive) local tropicalization} (see \autoref{def:posloctrop}), 
\emph{(standard) tropicalizing fan} (see \autoref{def:tropfans}) and 
\emph{Newton non-degeneracy} (see Definitions \ref{def:NNDgerms}, \ref{def:NNDcomplint}). 
In particular, we state a local analog  
of global theorems of Tevelev,  Luxton and Qu showing that the strict transform of a 
Newton non-degenerate germ by a tropicalizing fan is boundary transversal inside 
the ambient toric variety (see \autoref{thm:NNDtoroidal}). 
\medskip

Throughout this section, we view $\CC^n$ as an affine toric variety, whose toric boundary 
$\partial \CC^n$ consists of the union of all coordinate hyperplanes. We let 
$\boxed{\sigma} :=(\Rp)^n$ be the cone of non-negative weight vectors. A vector 
$\wu{}=(\wu{1},\ldots, \wu{n}) \in \sigma$ endows each variable $z_i$ of $\CC^n$ with weight $\wu{i}$. 
Any fan $\cF$ with support $|\cF|$ contained in $\sigma$ determines a birational toric morphism
 \[  \pi_{\cF} \colon \tv_{\cF} \to \CC^n.  \] 
This morphism is proper (and, therefore, a modification of $\CC^n$) if, and only if,  
$|\cF| = |\sigma|=(\Rp)^n$. 

Fix  a germ $(X,0) \hookrightarrow (\CC^n, 0)$  of an irreducible complex analytic space not contained in the toric boundary of $\CC^n$. Even if $\pi_{\cF}$ is not proper, its restriction 
\begin{equation}\label{eq:piOnX}
  \pi \colon \tilde{X} \to X
\end{equation}
to the strict transform $\tilde{X}$ of $X$ by $\pi_{\cF}$ may very well be. Properness is controlled by a cone (i.e., a set closed under scaling by $\Rp$) inside $\sigma$, called the 
\emph{local tropicalization}  $\Trop X$ 
of $(X,0) \hookrightarrow (\CC^n, 0)$. This is the content of the next proposition, which we view as a local version of~\cite[Proposition 6.4.7]{MS 15} inspired by Tevelev's work~\cite{T 07} 
(see \autoref{rem:schoncompar}). 
More details (including a proof) can be found in~\cite[Proposition 3.15]{CPS 21}:

\begin{proposition}   \label{prop:purecodimlift}
   With the previous notations and hypotheses, the following properties hold:
       \begin{enumerate}
           \item \label{inclusionloctrop} 
              The morphism $\pi$ from~\eqref{eq:piOnX} is proper 
               if, and only if,  the support $|\cF|$ contains the local tropicalization $\Trop X$. 
            \item Assume that $\pi$ is proper. Then, $|\cF| = \Trop X$ if, and only if,  $\widetilde{X}$ 
               intersects every orbit $S$ of the toric variety $\tv_{\cF}$ along a non-empty  pure-dimensional subvariety with 
               $\codim_{\widetilde{Y}}(\widetilde{Y}\cap S)= \codim_{\tv_{\cF}} (S)$. 
       \end{enumerate}
\end{proposition}

\autoref{prop:purecodimlift} (\ref{inclusionloctrop}) gives a complete characterization of 
the local tropicalization of an irreducible germ not contained in the toric boundary. 
This construction extends readily to any finite union of germs of this type 
by setting its local tropicalization to be the union 
of the local tropicalizations of its irreducible components. The formal definition of local tropicalizations, provided below, follows the construction of global tropicalization for subvarieties of tori from~\cite[Theorem 3.2.3]{MS 15} and it implies this additivity property:

 \begin{definition}   \label{def:posloctrop} 
     Let $(X,0) \hookrightarrow \CC^n$ be a  germ of a complex analytic space defined 
     by an ideal $I$ of the power series ring
     $\boxed{\cO} := \CC\{z_1, \dots, z_n\}$. If $\wu{}\in \sigma$, 
     the \textbf{$\wu{}$-initial ideal} of $I$ is the ideal  $\boxed{\initwf{\wu{}}{I}\cO}$ of $\cO$ generated 
      by the $\wu{}$-initial forms of all elements in $I$. 
      
     The \textbf{local tropicalization of $X$}\index{local tropicalization} is the set of all vectors 
     $\wu{}\in \sigma$ such that the $\wu{}$-initial ideal $\initwf{\wu{}}{I}\cO \subseteq \cO$ of $I$  
     is  monomial-free. We denote it by  $\boxed{\Trop X}$.
        In turn, the {\bf positive local tropicalization of  $X$}\index{local tropicalization!positive} 
        is  the intersection of the local tropicalization with the positive orthant $(\R_{>0})^n$. 
        We denote it by  $\boxed{\ptrop X}$.
\end{definition}

The two previous notions of local tropicalization depend on the embedding 
      $(X,0) \hookrightarrow (\CC^n,0)$. For simplicity, we do not include this 
      embedding in the notation of $\Trop X$ since it can be inferred from context.

\begin{remark}   \label{rem:initdefltrop}
     Local tropicalizations were introduced by the last two authors  in a slightly 
    different form~\cite{PPS 13}, i.e., for germs of analytic or formal spaces 
     contained or even mapped to germs of arbitrary affine toric varieties. 
     In that paper, the two versions $\Trop X$ and $\Trop X_{>0}$ of local tropicalization 
      contained also strata ``at infinity'', corresponding to the local tropicalizations 
      of the intersections of $(X,0)$ with various torus-orbit closures. 
\end{remark}

\autoref{def:posloctrop}, combined with the existence of standard bases for ideals of $\cO$, 
ensures that  $\Trop X$ is the support of a fan (see \cite[Theorem 11.9]{PPS 13}). 
However, $\Trop X$  has a priori no preferred fan structure. Any fixed structure on $\Trop X$ 
can be  further refined to satisfy desired properties (e.g.,~regularity). Of particular interest 
to us are fan structures for which the initial ideals $\initwf{\wu{}}{I}$ are constant along 
the relative interiors of all cones of $\Trop X$. More precisely (see 
\cite[Definition 3.13]{CPS 21}):

\begin{definition}  \label{def:tropfans}
    Let $(X,0) \hookrightarrow (\CC^n,0)$ be a  germ of a complex analytic space defined 
    by an ideal $I$ of 
     $\cO$. A \textbf{tropicalizing fan}\index{tropicalizing fan} for $(X,0)$ is a fan $\cF$ whose support is 
     the local tropicalization $\Trop X$. In turn, 
     a \textbf{standard tropicalizing fan}\index{tropicalizing fan!standard} for $(X,0)$ is a 
     tropicalizing fan such that $\initwf{\wu{}}{I}$ is constant when $\wu{}$ varies along the relative interior 
     of any cone of $\cF$.
\end{definition}

The adjective ``standard'' makes reference to ``standard bases'', which are used in 
\cite[Section 9]{PPS 13} to define local tropicalizations, in analogy with the use of Gr\"obner bases to study global tropicalizations of subvarieties of tori. Note that when $(X,0)$ is defined by 
polynomial equations, the Gr\"obner complex of $X$ determines a tropicalizing fan for $(X,0)$, as was shown by Aroca, G\'omez-Morales and Shabbir in~\cite{AGS 13}. Standard tropicalizing 
fans always exist 
in the holomorphic context, as we proved in \cite[Proposition 3.11]{CPS 21}.  

\begin{remark}
If $(X,0) \hookrightarrow (\CC^n,0)$ is a hypersurface singularity defined 
   by a series $f \in \cO$, then $\Trop X$ admits a coarsest fan structure which, in addition, is a standard tropicalizing fan for $(X,0)$. Indeed, we can describe $\Trop X$ as  the subfan of the Newton fan of $f$ consisting of all cones of dimension at most $n-1$. The duality between the Newton fan and the Newton polyhedron of $f$ (i.e., the convex hull of the union of 
   $\sigma$-translates of the space of exponents 
   of monomials in the support of $f$) confirms this fact.
\end{remark}

\begin{figure}[tb]
  \includegraphics[scale=0.8]{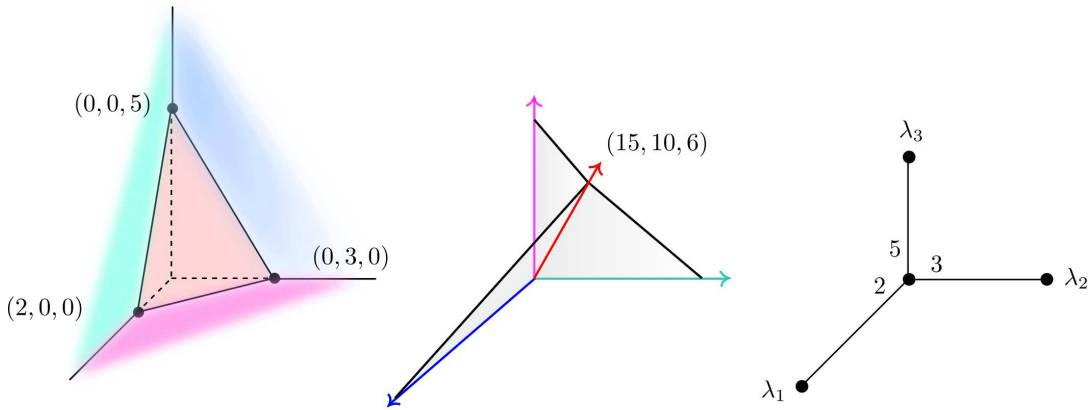}
    \caption{The Newton polyhedron, the local tropicalization and the splice diagram 
      of the $E_8$ surface singularity (see \autoref{ex:NpolE8}). 
      \label{fig:ExampleE8SingularityTropical}}
\end{figure}

   \begin{example}   \label{ex:NpolE8}
      Let  $(X,0) \hookrightarrow (\CC^3, 0)$ be  the $E_8$ surface singularity 
      from~\autoref{ex:quasihomE8}. As~\autoref{fig:ExampleE8SingularityTropical} shows, 
   its standard tropicalizing fan consists of the faces of the two-dimensional 
   cones spanned by  $\wu{} := (3 \cdot 5, 2 \cdot 5, 2 \cdot 3)$ 
   and each of the coordinate weight vectors. Note that $\wu{}$ is orthogonal to the unique compact 
   two-dimensional face of the Newton polyhedron of $f:=x^2 + y^3 + z^5$ because $f$ 
   is $\wu{}$-homogeneous. Note that a transversal 
   section of the standard tropicalization of $X$ is isomorphic to its associated splice diagram. As explained in~\autoref{rem:tropNND2D}, this property holds for any splice type singularity.
\end{example}

The positive local tropicalization of a germ  determines its local tropicalization, as the next statement confirms (see~\cite[Proposition 3.8]{CPS 21} for details).  

\begin{proposition}    \label{prop:compartrop} 
  Consider a germ $(X,0) \hookrightarrow (\CC^n, 0)$.
  The local tropicalization  $\Trop X$ is the topological 
  closure of the positive local tropicalization $\ptrop X$ 
  inside the cone $\sigma$.
\end{proposition}

This result was heavily used in~\cite{CPS 21} to compute local tropicalizations of splice type singularities. The same method  determines the local tropicalization of edge deformations of these germs, as we discuss in~\autoref{sec:edgedeform}.

\smallskip

The second main result in~\cite{CPS 21} confirms that splice type surface singularities are Newton non-degenerate. Such property characterizes the simplest germs from the toric perspective. More precisely: 

\begin{definition}  \label{def:NNDgerms}
       Let $(X,0) \hookrightarrow \CC^n$ be a reduced germ defined by an ideal $I$ of $\cO$. 
       We say that $X$ is \textbf{Newton non-degenerate}\index{Newton non-degeneracy} 
       if for any $\wu{} \in (\R_{>0})^n$, the $\wu{}$-initial ideal 
       $\initwf{\wu{}}{I}\subset \CC[z_1,\ldots, z_n]$ defines a smooth subscheme 
       of the algebraic torus $(\CC^*)^n$. 
\end{definition}

\begin{remark}\label{rem:AGMS_DefNND} An alternative (yet equivalent) formulation of Newton non-degeneracy was proposed by Aroca, G\'omez-Morales and Shabbir in~\cite[Definition 11.2]{AGS 13}. They require that for each  $\wu{}\in \sigma$, the extended initial ideal $\initwf{\wu{}}{I}\cO\subset \cO$ defines a smooth subscheme of the algebraic torus $(\CC^*)^n$. Therefore, any potential singularity of the germ defined by $I$ is contained in the toric boundary.
\end{remark}

For hypersurfaces,~\autoref{def:NNDgerms} coincides with the prototypical definition of Newton non-degeneracy given by Kouchnirenko \cite[Definition 1.19]{K 76}. It  was extended 
to complete intersections by Khovanskii in \cite[Section 2.4]{K 77} (see also Oka's 
book ~\cite[page 112]{O 97}).  Here, we use a slight variation of it, in agreement with the setting from \cite{CPS 21}: 

\begin{definition}  \label{def:NNDcomplint}
   Let $(X,0) \hookrightarrow \CC^n$ be a reduced germ defined by a regular sequence 
   $(f_1, \dots, f_k)$ of elements of $\cO$. We say that this sequence is a 
   \textbf{Newton non-degenerate complete intersection presentation}\index{Newton non-degeneracy} 
   of $(X,0)$ if 
   for any $\wu{} \in (\R_{>0})^n$, the sequence of $\wu{}$-initial forms 
 $(\initwf{\wu{}}{f_1}, \dots, \initwf{\wu{}}{f_k})$ defines either the empty set 
   or a smooth complete intersection of the algebraic torus $(\CC^*)^n$. 
\end{definition}

The difference with Khovanskii's definition lies in the requirement of regularity of the sequence 
$(f_1, \dots, f_k)$, i.e., it must define $(X,0)$ as a \emph{complete intersection} in the standard sense. 
Note that if $(f_1, \dots, f_k)$ is a Newton non-degenerate complete intersection presentation 
in the sense of~\autoref{def:NNDcomplint} and if each series $f_i$ is multiplied by a suitable monomial such that all those monomials contain a common variable, then the resulting sequence is no longer regular, but it  nevertheless defines a 
Newton non-degenerate complete intersection singularity in Khovanskii's sense.

\autoref{def:NNDcomplint} is more restrictive than~\autoref{def:NNDgerms}. More precisely:

\begin{proposition}  \label{prop:strongerNND}
    Let $(f_1, \dots, f_k)$ be a Newton non-degenerate complete intersection presentation of a germ 
    $(X,0) \hookrightarrow \CC^n$. Then, $(X,0)$ is Newton non-degenerate. 
\end{proposition}

\begin{proof}
        Let $\cF$ be a standard tropicalizing fan  of $X$ as in~\autoref{def:tropfans}. 
       The relative interior of each cone of $\cF$ contains at least one primitive 
       integral weight vector. This vector is unique if, and only if, the cone is a ray.  
       The constancy of initial ideals along  relative interiors of cones of $\cF$ 
       ensures that it is enough  to prove that for every primitive 
       integral vector $\wu{} \in (\R_{>0})^n$, 
       the subscheme $Z(\initwf{\wu{} }{I})$ of $(\CC^*)^n$ defined by 
       $\initwf{\wu{} }{I}$ is smooth. Here, $I$ denotes the ideal of $\cO$ generated 
       by  $(f_1, \dots, f_k)$. By hypothesis, $I$ defines the  germ $X$.  
       
         Let us fix a primitive integral vector $\wu{} \in (\R_{>0})^n$. 
        Consider the codimension one orbit $O_{\wu{} }$ inside the toric variety $\tv_{\R_{\geq 0} \wu{}}$.  
        There is a natural morphism of 
       algebraic tori $\varphi\colon (\CC^*)^n \to O_{\wu{} }$, corresponding to the quotient morphism of the 
       weight lattice $\Z^n$ of $(\CC^*)^n$ by the sublattice $\Z\, \wu{}$. 
       The scheme $Z(\initwf{\wu{} }{I})$ is the 
       preimage under $\varphi$ of the scheme-theoretic intersection $O_{\wu{} }  \cap  \tilde{X}$, 
       where $\tilde{X}$  is the strict transform 
       of $X$ by the toric birational morphism $\tv_{\R_{\geq 0} \wu{}} \to \CC^n$. 
       Thus, {\em it suffices to prove 
       that $O_{\wu{} }  \cap  \tilde{X}$  is smooth}.
       
        As $X$ is a complete 
       intersection germ, it is of pure dimension, say $d > 0$. Therefore, $\tilde{X}$ 
       is also of pure-dimension $d$, and so $O_{\wu{} }  \cap  \tilde{X}$ is pure  of dimension $d-1$. 
        Since $(f_1, \dots, f_k)$ is a Newton non-degenerate complete intersection presentation 
       of   $(X,0)$, we know that the strict transforms $\widetilde{Z(f_i)}$ 
       of the hypersurface germs $Z(f_i)$ 
       defined by the holomorphic germs $f_i$ intersect the orbit $O_{\wu{} }$ along hypersurfaces 
       which form a normal crossings divisor in a neighborhood of their intersection. Therefore 
       the scheme-theoretic intersection 
       $O_{\wu{} }  \cap   \bigcap_{i=1}^k\widetilde{Z(f_i)}$ is smooth of pure 
       dimension $d-1$.  Since     $O_{\wu{} }  \cap  \tilde{X}   \subseteq O_{\wu{} }  \cap   
           \bigcap_{i=1}^k\widetilde{Z(f_i)}$ is an inclusion 
       of schemes of pure dimension $d-1$, we deduce that $O_{\wu{} } \ \cap \ \tilde{X}$ 
       is a union of irreducible components of $O_{\wu{}} \bigcap_{i=1}^k\widetilde{Z(f_i)}$. 
       Thus, it is smooth.
       
      We claim that, furthermore,  the equality  
      $O_{\wu{} }  \cap  \tilde{X}   = O_{\wu{} }  \cap \bigcap_{i=1}^k\widetilde{Z(f_i)}$ holds. 
      To show the missing inclusion,  pick a point 
           $p \in O_{\wu{} }  \cap   \bigcap_{i=1}^k\widetilde{Z(f_i)}$ and let 
           $g_w \in \cO_p$ 
           be a defining function of $O_w$ in the local ring $\cO_p$  at $p$ of the complex 
           analytic variety $\tv_{\R_{\geq 0} \wu{}}$. 
           For every $i\in \{1,\ldots, k\}$, we pick  a defining function  
           $\tilde{f_i}  \in \cO_p$ 
           of the strict transform $\widetilde{Z(f_i)}$. Such functions 
           exist because $\tv_{\R_{\geq 0} \wu{}}$ is smooth. 
           As $O_{\wu{} }  \cap   \bigcap_{i=1}^k\widetilde{Z(f_i)}$ is pure of codimension $k$ 
           in $O_{\wu{} }$, we see that $(g_w, \tilde{f_1}, \dots, \tilde{f_k})$ is a regular sequence 
           in the local ring $\cO_p$. Therefore, the sequence $(\tilde{f_1}, \dots, \tilde{f_k}, g_w)$ 
           is also regular. Thus, $p$ lies in the closure  of the intersection 
           $(\CC^*)^n  \cap  \bigcap_{i=1}^k\widetilde{Z(f_i)}$. The latter equals
           $(\CC^*)^n  \cap   \tilde{X}$ by the complete intersection hypothesis. Thus, 
           $p \in O_{\wu{} }  \cap  \tilde{X}$, as desired. 
\end{proof}

The next result confirms the close interplay between Newton non-degeneracy, tropicalizing fans and toroidal varieties. It reinforces Teissier's suggestions from~\cite[Section 5]{T 04} 
{\em to take the $\partial$-transversality of a strict transform of a subgerm of $(\CC^n,0)$ by a toric modification of the ambient space as a general definition of Newton non-degeneracy 
for arbitrary germs in $(\CC^n,0)$}. More precisely:

\begin{theorem}  \label{thm:NNDtoroidal}
      Let $(X,0) \hookrightarrow \CC^n$ be a Newton non-degenerate germ 
      in the sense of \autoref{def:NNDgerms}, and let $\cF$ be 
      a standard tropicalizing fan for  it. Denote by $\tilde{X}$ the strict transform of $X$ by the 
      toric birational morphism $\pi_{\cF}\colon \tv_{\cF}\to \CC^n$. Then, $\tilde{X}$ 
      is $\partial$-transversal 
      in the toroidal variety $(\tv_{\cF}, \partial \tv_{\cF})$ in the sense of~\autoref{def:bdrytransv}.       
      Furthermore, the pair $(\tilde{X}, \tilde{X} \cap \partial \tv_{\cF})$ 
      and the morphism 
      $(\tilde{X}, \tilde{X} \cap \partial \tv_{\cF}) \to (\tv_{\cF}, \partial \tv_{\cF})$ are toroidal.
\end{theorem}

We view~\autoref{thm:NNDtoroidal} as a local version of Luxton and Qu's 
result~\cite[Theorem 1.5]{LQ 11} regarding \emph{sch\"on subvarieties} 
of algebraic tori $(\CC^*)^n$. Such subvarieties satisfy the conditions from~\autoref{def:NNDgerms}, 
but in the global setting (see~\autoref{rem:schoncompar}).

When $(X,0)$ is a  Newton non-degenerate hypersurface germ and $\cF$ is a  
regular fan (that is, its associated toric variety is smooth),~\autoref{thm:NNDtoroidal}  
is a consequence of Varchenko's 
results from~\cite[Section 10]{V 76} (see also Merle's work~\cite{M 80}). In turn, if $(X,0)$ 
is a  Newton non-degenerate complete intersection and  $\cF$ is a regular fan, the statement 
follows from~\cite[Theorem III.3.4]{O 97}. The last claim in the statement is obtained by combining 
the $\partial$-transversality of $\tilde{X}$ with~\autoref{prop:restoroid}.

\section{Edge deformations of splice type systems}   \label{sec:edgedeform}

In this section, we define the special smoothings of splice type singularities which 
we use to prove the Milnor fiber Conjecture (see \autoref{def:deformation}). 
They depend on the choice of an internal edge  of the associated splice diagram $\Gamma$  
and on a triple 
of positive integers. The smoothings are constructed by adding scalar multiples of 
suitable powers of a new deformation variable $z_0$ to each series $\fgvi{v}{i}$ from \eqref{eq:surfaceSeries} defining the associated splice type system $\sG{\Gamma}$.
\medskip

Throughout, we let $\Gamma$ be a splice diagram satisfying the determinant and  
semigroup conditions (see Definitions \ref{def:edgedet} and~\ref{def:semgpcond}). 
We fix two adjacent nodes $a,b$ of $\Gamma$.
  As illustrated in~\autoref{fig:splittingByEdge}, 
  we let $\roottree$ be any point in the interior of $[a,b]$ and 
  $\boxed{\wtG{\Gamma}}$  
  be the rooted tree obtained by subdividing $[a,b]$ along $\roottree$, and fixing its root at $\roottree$.  In order for $\wtG{\Gamma}$ to be a splice diagram 
  (in a slightly more general sense than we allowed in previous sections since the vertex $r$ has valency two), we must endow it with weights around $r$. This is done via the following lemma, whose proof is a direct consequence of the positivity of the determinant  of the edge $[a,b]$:

\begin{lemma}\label{lm:refineG} 
        There exist positive coprime integers  
        $k_a$ and $k_b$ satisfying the inequalities:
  \begin{equation}\label{eq:refinedG}
      \frac{\du{a}}{(\du{a,b})^2} < \frac{k_a}{k_b} <     \frac{(\du{b,a})^2}{\du{b}} .
  \end{equation}
    In particular, the decorated diagram $\wtG{\Gamma}$ seen on the right of~\autoref{fig:enrichedDiagram}, which is obtained from  $\Gamma$  
  by subdividing $[a,b]$ using $\roottree$ and setting $\boxed{\du{\roottree,a}} := k_a$, 
  $\boxed{\du{\roottree, b}}:=k_b$, satisfies the edge determinant condition. 
\end{lemma}

\begin{figure}[tb]
  \includegraphics[scale=0.45]{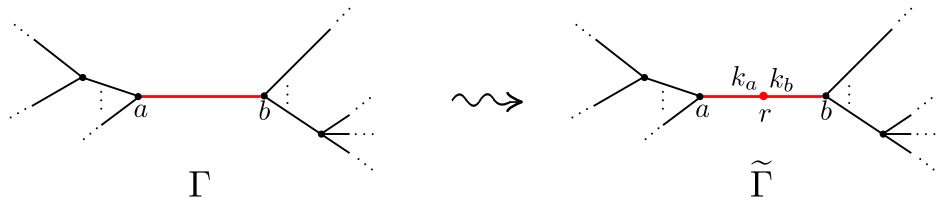}
    \caption{From left to right: a splice diagram $\Gamma$ and a subdivision of it induced by a point $\roottree$ in the relative interior of an internal edge $[a,b]$ (in red) producing a new splice diagram $\wtG{\Gamma}$ after decorating the edges around $\roottree$ with appropriate integers $k_a$ and $k_b$ (see \autoref{lm:refineG}).\label{fig:enrichedDiagram}}
\end{figure}

The weights on $\wtG{\Gamma}$ yield  a well-defined notion of 
  linking number $\wtuv{u}{v}$ of any two vertices $u,v$ of $\wtG{\Gamma}$.
  In turn, we use this to write  a weight vector for each node of $\wtG{\Gamma}$, including the root $\roottree$, by analogy with the construction of weight vectors for the nodes of 
  $\Gamma$ (see ~\eqref{eq:wu}).  Since $\wtG{\Gamma}$ has the same leaves as $\Gamma$, 
that is, $\leavesT{\wtG{\Gamma}} = \leavesT{\Gamma}$,    we  
view the lattices $\Nw{\Gamma}$ and $\Mw{\Gamma}$ from~\autoref{ssec:spltypesing} also as the weight lattice and lattice of 
exponent vectors of $\wtG{\Gamma}$. 
In particular, we set
      \begin{equation}\label{eq:weightsExtGamma}
           \boxed{\wu{\roottree}}:= \sum_{\lambda\in \leavesT{\Gamma}} 
               \wtuv{\roottree}{\lambda}     \wu{\lambda} \in N(\partial \Gamma).
       \end{equation}

      As was mentioned above, edge deformations of splice type systems depend on a  triple of positive integers. Here is the precise definition:

\begin{definition}\label{def:tripleAdapted} 
    A triple $\boxed{(k_a,k_b,D)}$ of positive integers is \textbf{adapted to the edge $[a,b]$} 
    of $\Gamma$ if $k_a$, $k_b$ satisfy the inequalities of~\autoref{lm:refineG} and 
      $D$ is divisible by all  decorations $\du{u,\roottree}$ 
     of  $\wtG{\Gamma}$, when $u$ varies among the nodes of $\Gamma$. 
     An \textbf{enrichment of $\Gamma$ relative to the edge $[a,b]$} is a choice 
     of a triple $(k_a,k_b,D)$ adapted to $[a,b]$, or equivalently, the datum of
     the splice diagram $\wtG{\Gamma}$ together with the integer $D$.   
\end{definition}

Such triples $(k_a,k_b,D)$ always exist,  by~\autoref{lm:refineG}.   
In order to build a  deformation of the  system $\sG{\Gamma}$ for a fixed triple, we introduce a new variable $z_0$ (the \textbf{deformation parameter}) and define two extended lattices
\begin{equation}\label{eq:extendedLattices}
        \boxed{\Nwbar{\partial \Gamma}} := \Z\langle \wu{0}\rangle \oplus \Nw{\Gamma}\simeq \Z^{n+1}  
              \qquad \text{ and } \qquad   
           \boxed{\Mwbar{\partial \Gamma}} := 
                \Z \langle \wu{0}^{\vee}\rangle \oplus \Mw{\Gamma}\simeq \Z^{n+1},
\end{equation}
where $\boxed{\wu{0}}$ and $\boxed{\wu{0}^{\vee}}$  denote the basis vectors 
corresponding to $z_0$ in $\Nwbar{\partial \Gamma}$ and $\Mwbar{\partial \Gamma}$, respectively. 
Analogously, for each $\lambda$ in $\leavesT{\Gamma}$, we let $\boxed{\wubar{\lambda}}$ 
be the  image in $\Nwbar{\partial \Gamma}$ of the basis vector $\wu{\lambda}$ from $\Nw{\Gamma}$. 
Similar notation applies to each vector $\boxed{\wubar{\lambda}^{\vee}}$ from $\Mwbar{\partial  \Gamma}$. 
We let $\boxed{\Nwbar{\partial  \Gamma}_{\R}} := \Nwbar{\partial  \Gamma} \otimes_{\Z} \R$ 
and $\boxed{\Mwbar{\partial  \Gamma}_{\R}} := \Mwbar{\partial  \Gamma} \otimes_{\Z} \R$ 
be the $\R$-vector spaces associated to the lattices in~\eqref{eq:extendedLattices}.

The  triple $(k_a,k_b,D)$ adapted to $[a,b]$  allows us to build new weight 
vectors in $\Nwbar{\partial \Gamma}$, i.e.,
\begin{equation}\label{eq:extWeights}
      \boxed{\wubar{u}}:= \wu{0} + \frac{D\,\wtuv{\roottree}{u}}{\du{u}} \wu{u} \in \Nwbar{\partial \Gamma} 
 \qquad \text{ for each node } u \text{ of } \wtG{\Gamma}.
\end{equation}
In particular: 
   \[ \wubar{\roottree} = \wu{0} + D\wu{\roottree}. \] 
      Notice that 
$\wubar{u}\in \Nwbar{\partial  \Gamma}$ since the divisibility constraint $\du{u,\roottree} | D$ imposed 
 on $D$ implies that $D\,\wtuv{\roottree}{u}/\du{u}\in \Z$.  The relevance of the weight vectors $\wubar{u}$ is explained in~\autoref{rem:extwhomog} below. 

\medskip       

Edge deformations of splice type systems adapted to an internal edge are constructed by analogy with \autoref{def:splicesystem}, as we now explain. 

\begin{definition}\label{def:deformation}\index{edge!deformation}
Let $\sG{\Gamma} = (\fgvi{v}{i}(\zu))_{v,i}$ be a splice type system.
  Fix an internal edge $e=[a,b]$ of $\Gamma$  and a triple $(k_a,k_b,D)$ adapted to it. We view $D$ as a supplementary decoration of the splice diagram $\wtG{\Gamma}$. 
\begin{itemize}
    \item An \textbf{edge-deformation} $\boxed{\dG{\wtG{\Gamma}}}$ of $\sG{\Gamma}$ 
        associated to the previous data 
        is a finite family of formal power series of the form:
         \begin{equation}\label{eq:3fold}
              \boxed{\fgvibar{v}{i}(z_0,\zu)}:= \fgvi{v}{i}(\zu) - \boxed{\cvi{v}{i}} \, z_0^{D\,\wtuv{\roottree}{v}}
              \quad \text{for all }   i\in \{1,\dots, \valv{v}-2\} \, \text{ and each node } v \text{ of } \Gamma,
         \end{equation}
         where $\cvi{v}{i}\in \CC^*$ and $\fgvi{v}{i}$ are as in~\eqref{eq:surfaceSeries}.  
     \item An \textbf{edge-deformation} of the splice type singularity  defined 
          by the system  $\sG{\Gamma}$ and associated to the previous data 
          is the subgerm at the origin of the affine space 
         $\CC^{n+1}$, which is defined by an edge deformation 
         system $\dG{\wtG{\Gamma}}$. 
         The deformation parameter is the new variable $z_0$. 
\end{itemize}
\end{definition}

\begin{remark}  \label{rem:extwhomog}
       By analogy with~\autoref{rm:HomogConditions} (\ref{item:whomog}), 
       we can show that our choice of exponents $D\,\wtuv{\roottree}{v}$ guarantees that  the polynomials 
     $\fvi{v}{i}(\zu) - \cvi{v}{i} \, z_0^{D\,\wtuv{\roottree}{v}}$ are $\wubar{u}$-homogeneous, 
     where  $\wubar{u}$ is the weight vector from~\eqref{eq:extWeights}. 
\end{remark}

\begin{remark}
        As we shall see in Step (\ref{item:toricmorpha}) of 
       \autoref{sec:stepsproof}, our proof of~\autoref{conj:MFC} requires an extension      
       of~\autoref{def:deformation} to the case where the edge $[a,b]$ 
       is not internal, but connects a node to a leaf. We do not discuss this generalization here, 
       to simplify the exposition.
\end{remark}

\begin{example}\label{ex:2NodesNW3fold} 
      We let $[a,b]$ be the unique internal edge of the splice diagram $\Gamma$ 
      from~\autoref{fig:numExampleSimpleEndCurve}, where $a=u$ and $b=v$. 
        We have multiple choices for the pairs $(k_a,k_b)$ satisfying $6/49 < k_a/k_b<11/70$. 
        For an illustration, we pick $(k_a,k_b) = (1,7)$. In particular,  
        $\wtuv{\roottree}{a} = 42 $ and $\wtuv{\roottree}{b}= 70$. 
        Thus, $\wu{\roottree} = (21,14, 10, 14, 35) \in \Z^5$. 
        Moreover, $\wu{a} = (147, 98, 60, 84, 210)$ and $\wu{b}=(210, 140, 110, 154, 385)$.
      
        The integer $D$ must be divisible by both $49$ and $11$, so we take $D=539$.
        A possible edge-deformation $\dG{\wtG{\Gamma}}$ of a strict splice-type 
        system  $\sG{\Gamma}$ 
        satisfying the Hamm determinant condition of \autoref{def:splicesystem} is  
 \begin{equation}\label{eq:numExample3fold} 
         \begin{cases} 
                \fvibar{a}{1}:= \;\;\;\, z_1^{2}\;\; \;-\;\;\;2\; z_2^3\; +\;\;\; z_4\,z_5\;\;\;  
                        + \;\;\; z_0^{22638} , \\
               \fvibar{b}{1}:=\;\;\;\; z_1z_2^4 +  z_3^7 + \;\; z_4^5\, - \;2155\; z_5^2    + z_0^{37730} , \\
               \fvibar{b}{2}:=  33\, z_1z_2^4 +  z_3^7 + 2\,z_4^5 - \; 2123 \, z_5^2  - z_0^{37730} .
        \end{cases}
 \end{equation}
 In particular, the three relevant extended weight vectors are 
     $\wubar{\roottree}  = (1, 79233, 52822, 32340, 45276, 113190)$, 
     $\wubar{a} = (1, 11319, 7546, 4620, 6468, 16170)$ and 
     $\wubar{b} = (1, 10290, 6860, 5390, 7546, 18865)$. 
\end{example}

In order for the germ defined by $\dG{\wtG{\Gamma}}$ to have a prescribed local tropicalization, 
we must impose further genericity constraints on the coefficients $\cvi{v}{i}$. 
To this end, given any $\wu{}\in (\R_{>0})^{n+1}\subset \Nwbar{\partial  \Gamma}_{\R}$ 
we consider the map
  \begin{equation}\label{eq:Fw}
      \boxed{F_{\wu{}}}\colon \CC^n\to \CC^{n-2} \qquad F_{\wu{}}(\zu) = (\initwf{\wu{}}{\fgvi{v}{i}}(\zu))_{v,i}\end{equation}
  whose entries are determined by the set of initial forms of all equations $\fgvi{v}{i}$ 
  defining the system $\sG{\Gamma}$. When restricted to codimension two coordinate subspaces of $\CC^n$, 
  the map $F_{\wu{}}$ satisfies the following key property:
    
  \begin{proposition}\label{pr:InitialRestr} For each $\wu{}\in \Trop X$ and any pair of distinct leaves $\lambda, \mu$ of $\Gamma$, the restriction map $\boxed{F_{\wu{},\lambda, \mu}}\colon \CC^{n-2} \to \CC^{n-2}$  to the coordinate subspace $Z(z_{\lambda},z_{\mu})$ of $\CC^n$ is generically finite, hence dominant.
  \end{proposition}

 This result allows us to specify explicit genericity conditions on the coefficients $\cvi{v}{i}$ from \eqref{eq:3fold} that are suitable for proving~\autoref{conj:MFC}. 
Under such genericity conditions, we can verify that the vanishing sets of both $\sG{\Gamma}$ and the edge-deformation $\dG{\wtG{\Gamma}}$  have similar behavior. More precisely,

\begin{theorem}\label{thm:NoPointsInTheBoundary}
    Assume that $(\cvi{v}{i})_{v,i}$ are generic and let  $\cY(\wtG{\Gamma})$ be the vanishing 
     set of the edge-deformation $\dG{\wtG{\Gamma}}$ in $\CC^{n+1}$. Then, 
    \begin{enumerate}
      \item the germ $(\cY(\wtG{\Gamma}),0)$  is a three-dimensional reduced and irreducible 
           isolated complete intersection singularity   
           not contained in the toric boundary of $\CC^{n+1}$;
      \item \label{NND3D} the series defining the edge-deformation $\dG{\wtG{\Gamma}}$ 
           determine a Newton non-degenerate complete intersection presentation of its vanishing 
            set $\cY(\wtG{\Gamma})$;
      \item the local tropicalization $\Trop \cY(\wtG{\Gamma})\subset (\Rp)^{n+1}$  
             is independent of $\dG{\wtG{\Gamma}}$ 
           and its coarsest fan structure is a standard tropicalizing fan for $\cY(\wtG{\Gamma})$.
    \end{enumerate}
 \end{theorem}

\begin{remark} \label{rem:tropNND3D} 
    The description of the top-dimensional cones of the standard tropicalization fan of 
    $\cY(\wtG{\Gamma})$ mentioned above is a bit more cumbersome than 
    for the splice type system $\sG{\Gamma}$ discussed in~\autoref{rem:tropNND2D}.   
    The explicit construction of this fan is used in Step (\ref{item:decompdivf}) 
    of~\autoref{sec:stepsproof} as well as in the proof 
    of~\autoref{thm:NoPointsInTheBoundary}~(\ref{NND3D}) under explicit genericity conditions. 
    The rays of $\Trop \cY(\wtG{\Gamma})$ are easy to list: they are generated by the weight 
    vectors $\wubar{u}$ indexed by all vertices  $u$ of the enriched splice diagram 
    $\wtG{\Gamma}$ plus one more ray corresponding to the deformation variable. 
    The fan is non-simplicial and its unique non-simplicial top-dimensional cone is spanned by 
    $\wu{0}, \wu{a}, \wu{b}$ and $\wu{\roottree}$. The presence of this last cone reveals the 
    product structure of the central component of the Milnor fiber of the germ defined 
    by $\sG{\Gamma}$ (see Step (\ref{item:prodstruct})).
\end{remark}

\section{Proof outline of the Milnor fiber conjecture}  \label{sec:stepsproof}

In this section, we outline our proof of Neumann and Wahl's {\em Milnor 
fiber conjecture} (see \autoref{conj:MFC}) through a sequence of 28 steps.  
Each step has a title, describing it briefly. The main statements proved at each step are written with boldface characters. The first four steps set up the deformations and smoothings of various splice type systems. The tropical techniques are used in Steps~(\ref{item:tropfanX}) through~(\ref{item:ttoriclift}), whereas logarithmic geometry features from  Step~(\ref{item:startlogpart}) onwards. This decomposition into steps is much more detailed 
than the decomposition into stages explained in \autoref{sec:introd}. The correspondence between 
them is as follows: Stage \ref{stage i} corresponds to Steps (\ref{item:enrich}) and (\ref{item:deform}); 
Stage \ref{stage ii} to Steps (\ref{item:toricmorpha}) and (\ref{item:deformsysta}); 
Stage \ref{stage iii} to Steps (\ref{item:tropfanX}), (\ref{item:tropfanY}) and (\ref{item:tropfanYa}); 
Stage \ref{stage iv} to Steps (\ref{item:torbirY}) and (\ref{item:torbirYa});
Stage \ref{stage v} to Steps (\ref{item:startlogpart}), (\ref{item:roundlogenhancmt}), 
      (\ref{item:removecollar}) and (\ref{item:startlogparta});
Stage \ref{stage vi} to Steps (\ref{item:decompdivf}), (\ref{item:subfanemb}), (\ref{item:ttoriclift}), 
  (\ref{item:isomlogf}), (\ref{item:isomtorlogf}), (\ref{item:combprevsteps}), (\ref{item:isomlogfa}), 
  (\ref{item:isomtorlogfa}), (\ref{item:combprevstepsa}), (\ref{item:commuttrianga}), 
  (\ref{item:isomtoriclog}) and  (\ref{item:cutfibinside});
Stage \ref{stage vii} to the remaining Steps (\ref{item:cutmiln}), (\ref{item:prodstruct}) and 
(\ref{item:finstep}). 
\medskip

We start from a splice diagram $\Gamma$ (see~\autoref{def:splicediag}) with $n$ leaves 
and at least two nodes, 
which satisfies the edge determinant condition of~\autoref{def:edgedet} 
and the semigroup condition of~\autoref{def:semgpcond}. 
We let $\boxed{(X,0)} \hookrightarrow (\CC^n, 0)$ be a splice type singularity defined 
 by a splice type system $\boxed{\cS(\Gamma)}$ as in~\autoref{def:splicesystem}. 
Fixing an internal edge $[a,b]$ of $\Gamma$ determines a partition of $\cS(\Gamma)$ 
into two systems: an \emph{$a$-side system} $\boxed{\cS_a(\Gamma)}$, combining the series 
associated to all the nodes seen from $b$ in the direction of $a$, and a 
\emph{$b$-side system} $\boxed{\cS_b(\Gamma)}$ involving the series associated 
to all the nodes of $\Gamma$ seen from $a$ in the direction of $b$.

   \medskip

    \begin{enumerate}
          \item \label{item:enrich}
              \underline{\em We enrich the splice diagram $\Gamma$.} 
              
                 We subdivide the splice diagram $\Gamma$ 
                 using an interior point $\boxed{\roottree}$ of the edge $[a,b]$ and we let
                  $\boxed{\wtG{\Gamma}}$ be the resulting tree, rooted at the vertex $\roottree$. 
                 We choose a triple $\boxed{(k_a, k_b, D)}$ adapted to $[a,b]$ in the sense of~\autoref{def:tripleAdapted} and we view $\wtG{\Gamma}$ as a splice diagram (with weights $\du{\roottree,a} = k_a$ and $\du{\roottree,b}=k_b$) enriched by $D$.

                \medskip    
         \item   \label{item:deform}
                \underline{\em We perform an edge deformation of the starting splice type system.}
               
               We consider an edge deformation  $\boxed{\dG{\wtG{\Gamma}}}$ 
               of the system $\cS(\Gamma)$ 
               in the sense of~\autoref{def:deformation} with deformation parameter 
               $\boxed{z_0}$. We assume that the coefficients $(\cvi{v}{i})_{v,i}\in (\CC^*)^{n-2}$ 
               satisfy the genericity constraints mentioned in~\autoref{sec:edgedeform}. We write
               $\dG{\wtG{\Gamma}}$ as the disjoint union of a \emph{deformed $a$-side system} 
               $\boxed{\dG{\wtG{\Gamma}}_a}$ and 
                a \emph{deformed $b$-side system} $\boxed{\dG{\wtG{\Gamma}}_b}$.

                We let $\boxed{(Y, 0)} \hookrightarrow \CC^{n+1}$ be the singularity defined by 
               $\dG{\wtG{\Gamma}}$ and denote by $\boxed{f}\colon Y \to \CC$ 
               the restriction of the linear form 
               $z_0\colon  \CC^{n+1} \to \CC$ 
               to $Y$. We prove that \textbf{$f$ is a smoothing of the splice type singularity $(X,0)$} 
               and  incorporate it into the following commutative diagram:               
                \begin{equation}  \label{eq:commutf}
                       \xymatrix{
                        X \:  \ar@{^{(}->}[r]
                        \ar[d]   &   Y \:  
                       \ar@{^{(}->}[rr]
                        \ar[d]_{f} & 
                        & \CC^{n+1}  \ar[d]^{z_0}  \\
                               0 \:  \ar@{^{(}->}[r] &   \CC \:  \ar@{^{(}->}[rr]  &  & \CC. }   
                  \end{equation}

              \medskip 
        \item   \label{item:toricmorpha}
              \underline{\em We define the notion of an $a$-side morphism associated to the given edge deformation.}
        
             Let $u$ be a node of the rooted tree $\Gamma_a$, seen as 
            a subtree of $\wtG{\Gamma}$ in~\autoref{fig:splittingByEdge}. 
            We prove that \textbf{the $\wu{u}$-initial forms
            of the series of the system $\dG{\wtG{\Gamma}}_b$ 
            are independent of the choice of $u$}. We let 
            $\boxed{\initwf{a}{\dG{\wtG{\Gamma}}_b  }}$ be
            the system determined by the vanishing of these $(n_b-1)$ initial forms.

            We prove that \textbf{the system $\initwf{a}{\dG{\wtG{\Gamma}}_b  }$ defines a 
            \emph{torus-translated toric subvariety}   of $\CC^{n+1}$ of dimension $n_a +1$. 
            Furthermore, this subvariety admits a normalization morphism}
                 \begin{equation} \label{eq:normaliza} 
                      \boxed{\varphi_a}  \colon \CC^{n_a +1} \to \CC^{n+1}, 
                  \end{equation}
            where  $\varphi_a$ is a monomial map (i.e., a torus-translated toric morphism). 
            Moreover, we have   
                  \begin{equation} \label{eq:pullbacka} 
                        \varphi_a^* z_0 = x_0,
                  \end{equation} 
            where $\boxed{x_0}$ is one of the variables of $\CC^{n_a +1}$.  
            We call $\varphi_a$ the  \emph{$a$-side morphism}.

            \medskip    
        \item    \label{item:deformsysta}
               \underline{\em We define an $a$-side deformation by a coordinate change of $\CC^{n_a+1}$ 
                 using the $a$-side morphism $\varphi_a$} \\
                 \underline{\em from~\eqref{eq:normaliza}.}
        
              We define a system $\boxed{\cD(\Gamma_a)}$ by pulling back the 
               system $\dG{\wtG{\Gamma}}_a$ via the $a$-side morphism 
               $\varphi_a$.  We let $\boxed{(Y_a, 0)} \hookrightarrow \CC^{n_a+1}$  be
               the singularity defined by the system $\cD(\Gamma_a)$.
                              Analogously, we let $\boxed{\cS(\Gamma_a)}$ be the pullback of the system 
               $\cS(\Gamma)$ via $\varphi_a$. By construction, $\cS(\Gamma_a)$  
               does not involve the variable $x_0$.  We identify  the 
               coordinate hyperplane $Z(x_0)$ of $\CC^{n_a + 1}$ with 
               $\CC^{n_a}$, and denote by $\boxed{(X_a,0)} \hookrightarrow (\CC^{n_a},0)$ the singularity    
               defined by the system $\cS(\Gamma_a)$.

               We show that  \textbf{$\cS(\Gamma_a)$ is a splice type system 
               with splice diagram $\Gamma_a$  
               and that the system $\cD(\Gamma_a)$ is an edge deformation 
               of $\cS(\Gamma_a)$ associated to the edge $[a, r_a]$, with deformation variable $x_0$ 
               (see~\autoref{fig:splittingByEdge}).} Notice that this last point requires us to extend 
               our definition of edge deformations to non-internal edges of splice diagrams.
                             As a consequence, \textbf{the restriction   of the linear map 
                             $x_0 \colon \CC^{n_a+1} \to \CC$ to $(Y_a,0)$ is a smoothing of   $(X_a,0)$.} 
                             We denote it by $f_a  \colon (Y_a,0) \to \CC$.
                              The above data fit into the following  commutative 
                             diagram analogous to~\eqref{eq:commutf}:
                               \begin{equation}  \label{eq:commutfa}
                   \xymatrix{
                        X_a \:  \ar@{^{(}->}[r]
                        \ar[d]   &   Y_a \:  
                       \ar@{^{(}->}[rr]
                        \ar[d]_{f_a} & 
                        & \CC^{n_a+1}  \ar[d]^{x_0}  \\
                               0 \:  \ar@{^{(}->}[r] &   \CC \:  \ar@{^{(}->}[rr]  &  & \CC. }     
                  \end{equation}

               As expected, we perform analogous constructions on the $b$-side, 
               denoting by $\boxed{y_0}$ the 
               corresponding deformation variable and by $\boxed{\varphi_b}$ 
               the $b$-side morphism. Since the construction is symmetric in $a$ and $b$, 
               we restrict our exposition to matters concerning only the $a$-side.

                 \medskip                    
        \item   \label{item:tropfanX}
              \underline{\em We build a standard tropicalizing fan for
                     $(X,0)$ and prove that $(X,0)$ is Newton non-degenerate.}
             
            \textbf{We construct a standard tropicalizing fan $\boxed{\cF_X}$ 
              for the embedding $(X, 0) \hookrightarrow \CC^{n}$  in the sense of~\autoref{def:tropfans} 
            and use it to prove  that the system $\cS(\Gamma)$ is a 
            Newton non-degenerate complete intersection presentation of $(X,0)$, 
            in the sense of~\autoref{def:NNDcomplint}.} Complete proofs 
            for these assertions can be found in \cite{CPS 21}.              
                 
                 We prove that {\bf the fan $\cF_X$ is a cone over a suitable embedding of the splice diagram 
                 $\Gamma$ in the standard simplex $\Delta_{n-1}\subset \R^n$ 
                 (see \cite[Theorem 1.2]{CPS 21})}.  
                 Its rays are in bijection with the vertices of $\Gamma$ and its 
                 two-dimensional cones are spanned by pairs of rays 
                 corresponding to adjacent vertices of $\Gamma$.   
                 Thus, the splice diagram appears as a transversal section of the local tropicalization of 
                 $(X, 0) \hookrightarrow \CC^{n}$. {\bf This gives the first tropical interpretation of 
                 splice diagrams, in the case when both the determinant and the semigroup 
                 conditions are satisfied.}

           \medskip                    
        \item   \label{item:tropfanY}
             \underline{\em  We build a standard  tropicalizing fan for $(Y,0)$
             and prove that this germ is Newton  non-degenerate.}
             
            Using the results of Step (\ref{item:tropfanX}),  
            \textbf{we describe a standard tropicalizing fan $\boxed{\cF}$ 
              for the embedding $(Y, 0) \hookrightarrow \CC^{n+1}$ 
              and prove that  the system $\dG{\wtG{\Gamma}}$ is a 
            Newton non-degenerate complete intersection presentation of $(Y,0)$.} 
            The genericity conditions on $\dG{\wtG{\Gamma}}$ are essential 
            to determine $\cF$, as discussed in~\autoref{thm:NoPointsInTheBoundary}. 
            A partial description of $\cF$ is given in~\autoref{rem:tropNND3D}. 
            Our proof uses results and techniques from~\cite{CPS 21}.  
            In particular, we show that the $2$-dimensional fan $\cF_X$ 
            introduced at Step (\ref{item:tropfanX}) is the union of strata at infinity of $\cF$  
            corresponding to the vanishing of the deformation parameter $z_0$ 
            (see \autoref{rem:initdefltrop}).

            \medskip
        \item    \label{item:torbirY}
             \underline{\em We describe a quasi-toroidalization of the smoothing $f$ of $(X,0)$ from 
              Step (\ref{item:deform}).}

            We let $\boxed{\pi_{\cF}}  \colon \tv_{\cF} \to \CC^{n+1}$ be the toric birational morphism 
               defined by the fan $\cF$ of Step (\ref{item:tropfanY}) and we denote by 
               $\boxed{\pi} \colon \tilde{Y} \to Y$ the restriction of 
           $\pi_{\cF}$ to the strict transform $\tilde{Y}$ of $Y$ by $\pi_{\cF}$. 
               Since $\cF$ is a tropicalizing fan 
           for  $(Y, 0) \hookrightarrow \CC^{n+1}$,~\autoref{prop:purecodimlift} 
           ensures that $\phi$ is a modification, unlike the case of the non-proper toric birational 
           map $\pi_{\cF}$.
          These data fit  into the commutative diagram           
             \[   \xymatrix{
                  \pi^{-1} (0)    =:    \boxed{\partial_0 \tilde{Y}} \:   \ar@{^{(}->}[rr]
                        \ar@<4ex>@{^{(}->}[d] & 
                        &\:\: \boxed{\partial_0 \tv_{\cF}}:= \pi_{\cF}^{-1} (0)  \ar@<-4ex>@{^{(}->}[d]   \\
                        \quad\quad\quad \quad
                        \tilde{Y} \: 
                       \ar@{^{(}->}[rr]
                       \ar@<4ex>[d]_{\pi} & 
                        & \:\tv_{\cF}                          \quad\quad\quad \:
                          \ar@<-4ex>[d]^{\pi_{\cF}}    \\
                        \quad\quad\quad \quad
                       Y  \:  \ar@{^{(}->}[rr]  &  & \: \CC^{n+1}.                         \quad\quad
             }     \]

           Using~\autoref{thm:NNDtoroidal}, we prove that \textbf{$\pi$ is a quasi-toroidalization of $f$} 
           (see~\autoref{def:quasitorsmoothing}). This statement follows 
           from the  fact that the deformed system $\dG{\wtG{\Gamma}}$ 
           is a Newton non-degenerate complete 
           intersection presentation of $(Y, 0)$, as discussed in  Step (\ref{item:tropfanY}).

            \medskip
         \item   \label{item:decompdivf}  
              \underline{\em We prove that the dual complex of the exceptional divisor of $\pi$ 
             is a subtree of $\wtG{\Gamma}$.} 
          
             The structure of the fan $\cF$ introduced in Step \eqref{item:tropfanY} 
             allows to prove that \textbf{the dual complex of the (compact) exceptional  
             divisor $\partial_0 \tilde{Y}$ of $\pi\colon \tilde{Y} \to Y$ is canonically isomorphic 
             to the unique connected subtree of 
             $\wtG{\Gamma}$  with vertex set equal to the set of nodes of $\tilde{\Gamma}$.} 
             This induces a decomposition of  $\partial_0 \tilde{Y}$ 
             as a sum  of three  reduced divisors, namely: 
                 \begin{equation} \label{eq:decompdiv}
                      \partial_0 \tilde{Y}  = 
                        \boxed{\partial_a \tilde{Y}}  + \boxed{\partial_{\roottree} \tilde{Y}}  + 
                              \boxed{\partial_b \tilde{Y}}.  
                  \end{equation}
              Here, $\partial_a \tilde{Y}$ is the sum of irreducible components of $\partial_0 \tilde{Y}$ 
              corresponding to the nodes of $\Gamma_a$ and similarly for $b$. In turn, 
              $\partial_\roottree \tilde{Y}$ is an irreducible variety corresponding 
              to the root $\roottree$ of $\wtG{\Gamma}$.

               \medskip
           \item  \label{item:tropfanYa}
                \underline{\em We perform the $a$-side analog of Step~(\ref{item:tropfanY}).} 
               
              \textbf{We determine a standard tropicalizing fan $\cF_a$ 
            for $(Y_a, 0) \hookrightarrow \CC^{n_a+1}$ and we use it to prove 
             that the system $\cD(\Gamma_a)$ introduced 
            in Step~(\ref{item:deformsysta}) is a Newton non-degenerate 
            complete intersection presentation of $(Y_a, 0)$}.  
              
              The rays of the fan $\cF_a$ correspond bijectively to the vertices of the rooted tree 
              $\Gamma_a$, excepted for a single ray, which is the 
                 coordinate ray associated to the deformation variable $x_0$. 
                 The cones spanned by pairs of rays 
                 corresponding to adjacent vertices  of $\Gamma_a$ belong to the set   
                 of two-dimensional cones of $\cF_a$. There are extra two-dimensional 
                 cones of $\cF_a$ not included in this list. The three-dimensional cones 
                 are spanned by some triples of rays of $\cF_a$.

             \medskip
           \item    \label{item:subfanemb}
               \underline{\em We compare objects associated to $X$ and their counterparts 
              on the $a$-side singularity $X_a$.} 
              
               Consider the torus-translated toric morphism 
                $ \varphi_a  \colon \CC^{n_a + 1} \to \CC^{n+1} $ 
                 introduced in Step (\ref{item:toricmorpha}).   
                     We prove that \textbf{the associated linear map 
                     $\phi_a  \colon N(\Gamma_a)_{\R} \to N(\Gamma)_{\R}$  between weight spaces 
                     \emph{almost embeds} the fan $\cF_a$ inside the fan $\cF$}. 
                     More precisely, it is an embedding when restricted to the subfan of 
                     $\cF_a$ spanned by the rays
                     associated to  any vertex of $\Gamma_a$ other than  the  root $\roottree_a$. 
                     Furthermore,  the ray $\boxed{l_{\roottree}}$ of $\cF$ 
                     associated to the root $\roottree$ of $\wtG{\Gamma}$ lies in the relative 
                     interior of the image under $\phi_a$ of the 
                     two-dimensional cone of $\cF_a$ spanned by the rays corresponding 
                     to the vertices $a$ and $\roottree_a$ of $\Gamma_a$.

                     We write  $\boxed{l_{a, \roottree}}:= 
                     \phi_a^{-1}(\{l_{\roottree}\}) \subseteq N(\Gamma_a)_{\R}$ 
                      and let $\boxed{\cF_{a, \roottree}}$ be the fan obtained by performing 
                     the stellar subdivision of $\cF_a$ along $l_{a, \roottree}$. 
                                                             Since it refines the standard tropicalizing fan $\cF_a$ for 
                     $(Y_a, 0) \hookrightarrow \CC^{n_a+1}$ from Step~(\ref{item:tropfanYa}), 
                     $\cF_{a, \roottree}$ is also a standard tropicalizing fan for $Y_a$.

        \medskip
           \item  \label{item:torbirYa} 
            
              \underline{\em We perform the $a$-side analog of Step~(\ref{item:torbirY}).}
             
            Let $\boxed{\pi_{a, \roottree}}\colon \tilde{Y}_a \to Y_a$ be the restriction of the  
            toric birational morphism 
            $\pi_{\cF_{a, \roottree}} \colon \tv_{\cF_{a, \roottree}} \to \CC^{n_a+1}$ 
           to the strict transform $\boxed{\tilde{Y}_a}$ of $Y_a$ by $\pi_{\cF_{a, \roottree}}$.             
           As $\cF_{a, \roottree}$ is a tropicalizing fan 
           for  $(Y_a, 0) \hookrightarrow \CC^{n+1}$ by Step (\ref{item:subfanemb}), the morphism 
           $\pi_{a, \roottree}$ is a modification. 
           This determines the commutative diagram
         \[   \xymatrix{
                 \pi_{a, \roottree}^{-1} (0) =:     \boxed{\partial_0 \tilde{Y}_a} \:   \ar@{^{(}->}[rr]
                        \ar@<4ex>@{^{(}->}[d] & 
                        &\:\: \boxed{\partial_0 \tv_{\cF_{a, \roottree}}}:=  \pi_{\cF_{a, \roottree}}^{-1} (0)     
                         \ar@<-4ex>@{^{(}->}[d]   \\
                        \quad\quad\quad \quad
                        \tilde{Y_a} \: 
                       \ar@{^{(}->}[rr]
                       \ar@<4ex>[d]_{\pi_{a, \roottree}} & 
                        & \:\tv_{\cF_{a, \roottree}}                          \quad\quad\quad \: 
                         \ar@<-4ex>[d]^{\pi_{\cF_{a, \roottree}}}    \\
                        \quad\quad\quad \quad
                       Y_a  \:  \ar@{^{(}->}[rr]  &  & \:\: \CC^{n_a+1}.                         \quad\quad
             }     \]
                            Since  the deformed system $\cD(\Gamma_a)$  
           is a \emph{Newton non-degenerate complete 
             intersection presentation} of $(Y_a, 0)$ by 
             Step~(\ref{item:tropfanYa}),~\autoref{thm:NNDtoroidal} confirms  
             that \textbf{$\pi_{a, \roottree}$  is a quasi-toroidalization of $f_a$}.

           \medskip
           \item    \label{item:ttoriclift}   
               \underline{\em We continue comparing objects associated to $X$ 
              with their counterparts on the $a$-side singularity $X_a$.}                            

                     We let  $\boxed{\cF_{a, \roottree}^-}$ be the subfan of $\cF_{a, \roottree}$ 
                     consisting of all  cones not containing the ray associated to the root 
                     $\roottree_a$ of $\Gamma_a$. By construction, the linear map $\phi_a$ from  
                     Step~(\ref{item:subfanemb})
                     embeds $\cF_{a, \roottree}^-$ into $\cF$.
                                          Therefore, \textbf{we can lift $ \varphi_a$ to a torus-translated morphism 
                     $\boxed{\Phi_a} \colon \tv_{\cF_{a, \roottree}^-} \to \tv_{\cF}$  
                     fitting into  the following commutative diagram
                     \[ \xymatrix{
                               \tv_{\cF_{a, \roottree}^-}
                                  \ar[rr]^{\Phi_a}
                                  \ar[d]_{\pi_{\cF_{a, \roottree}^-}} & 
                                                & \tv_{\cF}  \ar[d]^{\pi_{\cF}}  \\
                                  \CC^{n_a + 1} \ar[rr]_{\varphi_a}  &  & \CC^{n+1}, }  \]     
                     with the additional property that the morphism $\Phi_a$ is a toric embedding, 
                     that is, a toric morphism which is an embedding of algebraic varieties}.

            \medskip
        \item   \label{item:startlogpart}
              \underline{\em We build the log special fiber of a log enhancement of the lifting
              $\tilde{f}=f\circ \pi$, where $\pi$ is the quasi-}\\\underline{\em toroidalization of $f$
              from Step~(\ref{item:torbirY}).}

            Following the discussion preceding~\autoref{cor:milntubelog}, 
            we choose a  Milnor tube representative
           $f \colon Y \to \D$ of the smoothing $f$ and we consider its lift 
           $\boxed{\tilde{f}} := f \circ \pi  \colon \tilde{Y} \to \D$ to the modified space $\tilde{Y}$ 
            introduced in Step~(\ref{item:torbirY}). Recall that $\pi$ is a quasi-toroidalization of $f$.            
           We consider the \emph{log enhancement} 
           of $f$ relative to the divisors $Z(\tilde{f})$ and $\{0\}$  in the sense 
           of~\autoref{def:indlogmorph}, i.e., 
           \begin{equation*}   \label{eq:logenhancef}
                 \boxed{\tilde{f}^{\dagger}} \colon \tilde{Y}^{\dagger} \to \D^{\dagger},
           \end{equation*} 
         where   $\boxed{\tilde{Y}^{\dagger}} := (Y , \cO^{\star}_{\tilde{Y}}(- Z(\tilde{f})))$ 
         and $\boxed{\D^{\dagger}} := (\D, \cO^{\star}_{\D}(- \{ 0 \}))$ are log complex spaces. 
          At the level of sheaves 
           of monoids, $\tilde{f}^{\dagger}$ is simply the pullback of functions by $\tilde{f}$. 
        
              Consider now the \emph{log special fiber}   of the morphism $\tilde{f}^{\dagger}$
           \begin{equation}  \label{eq:logspecfibf} 
                 \boxed{(\tilde{f}_0)^{\dagger}}  \colon \boxed{Z(\tilde{f})^{\dagger}} 
                     \to \boxed{0^{\dagger}},
           \end{equation}
         obtained by restricting the log structures 
          of the source and target spaces to the special fiber of $\tilde{f}$ and to $\{0 \} \hookrightarrow \D$, 
          respectively (see~\autoref{rem:restrictions}). 
          The construction yields a commutative diagram  in the log category:
                     \begin{equation}  \label{eq:commdiaglog}
                           \xymatrix{
                               Z(\tilde{f})^{\dagger}
                                  \ar[rr]
                                  \ar[d]_{(\tilde{f}_0)^{\dagger}} & 
                                                & \tilde{Y}^{\dagger}  \ar[d]^{\tilde{f}^{\dagger}}  \\
                                  0^{\dagger} \ar[rr]  &  & \D^{\dagger}. }  
                       \end{equation} 
          Note that both horizontal arrows 
          are \emph{strict}, in the sense of~\autoref{def:strict}.

         \medskip
        \item   \label{item:roundlogenhancmt}
            \underline{\em We show that the rounding of the log enhancement of $\tilde{f}$
               is a representative of the Milnor  fibration} \\   \underline{\em of $f$.}

          Consider the \emph{rounding} of the diagram~\eqref{eq:commdiaglog} 
          in the sense of~\autoref{def:rounding}: 
              \begin{equation}   \label{eq:Brlogspec}
                     \xymatrix{
                               Z(\tilde{f})^{\dagger}_{\log}
                                  \ar[rr]
                                  \ar[d]_{(\tilde{f}_0)^{\dagger}_{\log}} & 
                                                & \tilde{Y}^{\dagger}_{\log}  \ar[d]^{\tilde{f}^{\dagger}_{\log}}  \\
                                  0^{\dagger}_{\log} \ar[rr]  &  & \D^{\dagger}_{\log}. }  
               \end{equation}
            Note that $0^{\dagger}_{\log}$ is a circle, identified canonically with  
            complex numbers of modulus one
            through the use of polar coordinates. As both horizontal arrows in \eqref{eq:commdiaglog} 
            are strict,~\autoref{prop:strictparallel} implies that \textbf{the diagram \eqref{eq:Brlogspec} 
              is cartesian in the topological category.}
            
            Using~\autoref{cor:quasitorrelcoh}, which is a direct consequence of 
            Nakayama and Ogus' local triviality 
           theorem (see~\autoref{thm:logehresm}), \textbf{we conclude that  
           the left vertical arrow of \eqref{eq:Brlogspec} 
            is a representative of the circular Milnor fibration of $f$.}

         \medskip  
        \item   \label{item:removecollar}
            \underline{\em We build a new representative of the Milnor fibration of $f$ by removing a collar
                    neighborhood  of the}\\ \underline{\em boundary
               of the total space of the fibration of Step (\ref{item:roundlogenhancmt}).}

           We consider  the rounding map of the complex  log space $Z(\tilde{f})^{\dagger}$: 
               \begin{equation*} \label{eq:BrY}
                       \tau_{Z(\tilde{f})^{\dagger}} \colon Z(\tilde{f})^{\dagger}_{\log} \to Z(\tilde{f}).
                 \end{equation*}
          \noindent
           \textbf{We build a new representative of the circular Milnor fibration of $f$ 
           using the  restriction 
                \begin{equation} \label{eq:reprcMFf}  
          (\tilde{f_0})^{\dagger}_{\log}|_{\tau_{Z(\tilde{f})^{\dagger}}^{-1}(\partial_0 \tilde{Y})}
         \colon   \tau_{Z(\tilde{f})^{\dagger}}^{-1}(\partial_0 \tilde{Y}) \to 0^{\dagger}_{\log} 
                 \end{equation}
            of the leftmost vertical arrow from~\eqref{eq:Brlogspec}    
            to the preimage $\tau_{Z(\tilde{f})^{\dagger}}^{-1}(\partial_0 \tilde{Y})$ 
            of the exceptional divisor $\partial_0 \tilde{Y} = \pi^{-1}(0)$ 
            of $\pi$ under the rounding map  $\tau_{Z(\tilde{f})^{\dagger}}$.} Note that 
            the divisor
            $\partial_0 \tilde{Y}$ already featured in Step~(\ref{item:torbirY}) 
            and that {\bf the complement of $\tau_{Z(\tilde{f})^{\dagger}}^{-1}(\partial_0 \tilde{Y})$ in 
            $Z(\tilde{f})^{\dagger}_{\log}$ is a collar neighborhood of the topological 
            boundary of $Z(\tilde{f})^{\dagger}_{\log}$}. The latter fact is crucial to prove the claim.

             \medskip  
        \item   \label{item:isomlogf}
             \underline{\em We prove that rounding yields a canonical decomposition of the source 
                of the Milnor fibration of $f$.}
        
            The decomposition \eqref{eq:decompdiv} induces the following decomposition of the 
            source space of the representative \eqref{eq:reprcMFf} of the circular Milnor fibration of $f$:
                \begin{equation*} \label{eq:decompBrYv0}
                      \tau_{Z(\tilde{f})^{\dagger}}^{-1}(\partial_0 \tilde{Y})  = 
                        \tau_{Z(\tilde{f})^{\dagger}}^{-1}(\partial_a \tilde{Y})   \: \: 
                           \cup \: \:  \tau_{Z(\tilde{f})^{\dagger}}^{-1}(\partial_{\roottree} \tilde{Y})   \: \: 
                           \cup  \: \:  \tau_{Z(\tilde{f})^{\dagger}}^{-1}(\partial_b \tilde{Y}) .  
                  \end{equation*}
                  We prove that \textbf{when restricted to the three parts of this decomposition, the 
                  rounding morphism $(\tilde{f}_0)^{\dagger}_{\log}$ is isomorphic 
                  to the roundings of the log morphisms 
                      \begin{equation} \label{eq:decompBrY}  
                                   (\partial_a \tilde{Y}, \cO^{\star}_{\tilde{Y}| \partial_a \tilde{Y}}( - Z(\tilde{f})) ) \to 
                                             0^{\dagger}    , \quad 
                                 (\partial_{\roottree} \tilde{Y}, \cO^{\star}_{\tilde{Y}| 
                                        \partial_{\roottree} \tilde{Y}}( - Z(\tilde{f})) ) 
                                 \to    0^{\dagger}    \; \text{ and } \;
                                     (\partial_b \tilde{Y}, \cO^{\star}_{\tilde{Y}| \partial_b \tilde{Y}}( - Z(\tilde{f})) ) \to 
                                             0^{\dagger}  
                      \end{equation}
                  obtained by \emph{restricting} the log special fiber \eqref{eq:logspecfibf} 
                  to the subdivisors $\partial_a \tilde{Y}, \partial_{\roottree} \tilde{Y}$ and 
                  $\partial_b \tilde{Y}$ of $Z(\tilde{f})$, respectively.}

             \medskip             
          \item \label{item:isomtorlogf}
             \underline{\em We prove that the three  log morphisms from~\eqref{eq:decompBrY} 
                 can be obtained by restrictions from the ambient} \\  \underline{\em  toric varieties.}

             Using the functoriality of restriction of log structures, 
             we prove that \textbf{the log morphisms of \eqref{eq:decompBrY} are isomorphic 
             to the log morphisms
                  \begin{equation} \label{eq:decompBrYtoric}
                      (\partial_a \tilde{Y}, \cO^{\star}_{\tv_{\cF}| \partial_a \tilde{Y}}( - Z(\tilde{z_0})) ) \to 
                                             0^{\dagger},\;\;
                      (\partial_{\roottree}\tilde{Y}, \cO^{\star}_{\tv_{\cF}| 
                                    \partial_{\roottree} \tilde{Y}}( - Z(\tilde{z_0})) )\to 
                                             0^{\dagger}     \text{ and } 
                       (\partial_b \tilde{Y}, \cO^{\star}_{\tv_{\cF}| \partial_b \tilde{Y}}( - Z(\tilde{z_0})) ) \to 
                                             0^{\dagger}  
                      \end{equation}
                   obtained by restricting 
                     the logarithmic enhancement   $\tilde{z_0}^{\dagger}\colon    \tv_{\cF}^{\dagger}                               
                                   \to \CC^{\dagger}$                                   
                      of the linear map $\tilde{z_0}\colon \tv_{\cF}  \to \CC$ 
                     to the subdivisors $\partial_a \tilde{Y}, \partial_{\roottree} \tilde{Y}$ 
                     and $\partial_b \tilde{Y}$ of $Z(\tilde{f})$, respectively}.
                  Here, $\boxed{\tv_{\cF}^{\dagger}}$ denotes the divisorial 
                  complex log space 
                  $ (\tv_{\cF}, \cO_{\tv_{\cF}}^{\star} ( - Z(\tilde{z_0})) )$.

            \medskip
           \item  \label{item:combprevsteps}
               \underline{\em We describe the $a$-side parts of the Milnor fibers of $f$.}
       
                  Using the results of Steps (\ref{item:startlogpart}) through (\ref{item:isomtorlogf}), 
                   we deduce that \textbf{the fibers of the rounding
                      \begin{equation*} \label{eq:bettiapart}
                         (\partial_a \tilde{Y}, \cO^{\star}_{\tv_{\cF}| \partial_a \tilde{Y}}( - Z(\tilde{z_0})) )_{\log} \to 
                                             0^{\dagger}_{\log}                    
                      \end{equation*} 
                   of the first arrow of~\eqref{eq:decompBrYtoric} are homeomorphic to the 
                   parts
                   of the fibers of the representative 
                   $ (\tilde{f}_0)^{\dagger}_{\log}  \colon Z(\tilde{f})^{\dagger}_{\log} \to 0^{\dagger}_{\log} $ 
                   of the circular Milnor fibration of $f$  contained inside 
                   $\tau_{Z(\tilde{f})^{\dagger}}^{-1}(\partial_a \tilde{Y})$.}

               \medskip
           \item  \label{item:startlogparta}
               \underline{\em We perform an $a$-side analog of Steps (\ref{item:startlogpart}), (\ref{item:roundlogenhancmt}) and~(\ref{item:removecollar}).}

               We choose a Milnor tube representative 
           $f_a \colon Y_a \to \D$ of the smoothing $f_a$ of $(X_a, 0)$ appearing in diagram 
           (\ref{eq:commutfa}) and we consider its lift 
           $\boxed{\tilde{f}_a} := f_a \circ \pi_{a, \roottree} \colon \tilde{Y}_a \to \D$ 
           to the modified space 
           $\tilde{Y}_a$  introduced in Step (\ref{item:torbirYa}).  
           Recall that $\pi_{a, \roottree}$ is a quasi-toroidalization of $f_a$.  

           Consider the log enhancement of $f_a$ relative to the divisors $Z(\tilde{f}_a)$ and $0$:          
           \begin{equation*}   \label{eq:logenhancefa}
                 \boxed{\tilde{f}^{\dagger}_a}  \colon \boxed{\tilde{Y}^{\dagger}_a} \to \D^{\dagger}.
           \end{equation*}
            
            In turn, we build the log special fiber of the morphism $\tilde{f}^{\dagger}_a$ and its rounding, i.e.
             \begin{equation}  \label{eq:Brlogspeca} 
                       (\tilde{f}_{a,0})^{\dagger}  \colon Z(\tilde{f}_a)^{\dagger} \to 0^{\dagger}  \qquad \text{ and }   
                       \qquad (\tilde{f}_{a,0})^{\dagger}_{\log}  \colon Z(\tilde{f}_a)^{\dagger}_{\log} 
                       \to 0^{\dagger}_{\log}.
             \end{equation}
          Let   $\tau_{Z(\tilde{f}_a)^{\dagger}} \colon Z(\tilde{f}_a)^{\dagger}_{\log} \to Z(\tilde{f}_a)$ 
                   be the rounding map of the complex  log space $Z(\tilde{f}_a)^{\dagger}$. 
                   \textbf{We show that the restriction
           of the rounding morphism $(\tilde{f}_{a,0})^{\dagger}_{\log}$ from~\eqref{eq:Brlogspeca}  
        to the preimage $\tau_{Z(\tilde{f}_a)^{\dagger}}^{-1}(\partial_0 \tilde{Y}_a)$             
            of the exceptional divisor $\boxed{\partial_0 \tilde{Y}_a} := \pi_{a, \roottree}^{-1}(0)$ 
            of $\pi_{a, \roottree}$ under the rounding map  $\tau_{Z(\tilde{f}_a)^{\dagger}}$  
            gives a representative of the circular   Milnor fibration of $f_a$.
                   }

                \medskip
           \item \label{item:isomlogfa}
           
                \underline{\em We perform an $a$-side analog of Step (\ref{item:isomlogf}).}     
              
              We prove that \textbf{when restricted to 
                   $\tau_{Z(\tilde{f}_a)^{\dagger}}^{-1}( \partial_0 \tilde{Y}_a )$, the 
                  rounding morphism $(\tilde{f}_{a,0})^{\dagger}_{\log}$ from \eqref{eq:Brlogspeca}  
                  is isomorphic 
                  to the rounding of the log morphism 
                      \begin{equation*} \label{eq:decompBrYa}
                                   (\partial_0 \tilde{Y}_a, 
                                    \cO^{\star}_{\tilde{Y}_a| \partial_0 \tilde{Y}_a}( - Z(\tilde{f}_a )) ) \to 
                                             0^{\dagger}   
                      \end{equation*}
                 obtained by restricting the log special fiber map 
                 $(\tilde{f}_{a,0})^{\dagger}$ from~\eqref{eq:Brlogspeca}
                  to the subdivisor $\partial_0 \tilde{Y}_a$ of $Z(\tilde{f}_a)$.} 

              \medskip
             \item \label{item:isomtorlogfa}
                  \underline{\em We perform an $a$-side analog of Step (\ref{item:isomtorlogf}).} 
                 
                 Recall that the variable $x_0$ of $\CC^{n_a + 1}$, 
                 introduced in Step (\ref{item:toricmorpha}), denotes 
                 the deformation variable of the $a$-side system $\cD(\Gamma_a)$ 
                 of Step (\ref{item:deformsysta}). Consider the tropicalizing fan  
                 $\cF_{a, \roottree}$ for $Y_a$
                  introduced in Step (\ref{item:subfanemb}).                  
                 Let $ \boxed{\tv_{\cF_{a, \roottree}}^{\dagger}} $ denote the divisorial 
                  complex log space 
                  $ (\tv_{\cF_{a, \roottree}}, \cO_{\tv_{\cF_{a, \roottree}}}^{\star} 
                         ( - Z(\tilde{x}_0 )) )$ 
                  and let
                    \begin{equation*}    \label{eq:logspectoricA-side}
                            \boxed{ \tilde{x_0}^{\dagger}}  \colon   \tv_{\cF_{a, \roottree}}^{\dagger}                               
                                   \to \D^{\dagger} 
                     \end{equation*} 
                   be the logarithmic enhancement of $\tilde{x_0}$ relative to the divisors  
                   $Z(\tilde{x}_0)$ and $0$. 
                 We prove that \textbf{the log morphism  \eqref{eq:decompBrYa} is isomorphic 
                  to the log morphism}  
                  \begin{equation} \label{eq:decompBrYtorica}
                          \tilde{x_0}^{\dagger}|_{\partial_0 \tilde{Y}_a} \colon      
                          (\partial_0 \tilde{Y}_a, \cO^{\star}_{\tv_{\cF_{a, \roottree}}| 
                         \partial_0 \tilde{Y}_a}( - Z(\tilde{x}_0 )) ) 
                            \to  0^{\dagger}  
                  \end{equation}
                   obtained by restricting   $\tilde{x_0}^{\dagger}$ 
                    to the subdivisor $\partial_0 \tilde{Y}_a$ of $Z(\tilde{f}_a)$.

               \medskip
           \item  \label{item:combprevstepsa}
           
               \underline{\em We perform an $a$-side analog of Step (\ref{item:combprevsteps}).} 
              
           Using the results of Steps (\ref{item:startlogparta}), (\ref{item:isomlogfa}) and (\ref{item:isomtorlogfa}), 
            we deduce that \textbf{the fibers of the rounding
                 \begin{equation*} \label{eq:bettiaparta}
                             (\partial_0 \tilde{Y}_a, \cO^{\star}_{\tv_{\cF_{a, \roottree}}| 
                         \partial_0 \tilde{Y}_a}( - Z(\tilde{x}_0 )) )_{\log} 
                            \to  0^{\dagger}_{\log}          
                \end{equation*}
                   of the log morphism~\eqref{eq:decompBrYtorica} are homeomorphic to 
                    the fibers of the representative 
                      \[ (\tilde{f}_{a,0})^{\dagger}_{\log} \colon   
                           \tau_{Z(\tilde{f}_a)^{\dagger}}^{-1}(\partial_0 \tilde{Y}_a) \to 0^{\dagger}_{\log}  \]
                    of the circular Milnor fibration of the smoothing $f_a$ of $X_a$.}

                     \medskip
               \item  \label{item:commuttrianga}   
                     \underline{\em We compare the objects associated
                      to the starting singularity $X$  and to the $a$-side singularity $X_a$,  as} \\
                      \underline{\em  a sequel to Steps   (\ref{item:subfanemb}) and (\ref{item:ttoriclift}), 
                      by  constructing a natural  map from an $a$-side  log morphism to a} \\
                       \underline{\em  log morphism associated with the initial smoothing.}

                   Recall the fan $\cF_{a, \roottree}^-$ and the torus-translated toric morphism 
                   $\Phi_a\colon \tv_{\cF_{a, \roottree}^-} \to \tv_{\cF}$ 
                    introduced in Step (\ref{item:ttoriclift}). 
                 The relation \eqref{eq:pullbacka} ensures that the following  triangle of 
                   torus-translated toric morphisms commutes:
                         \begin{equation*} \label{eq:commutrianga} 
                                \xymatrix{
                                      \tv_{\cF_{a, \roottree}^-}
                                       \ar@{->}[rr]^{\Phi_a}
                                  \ar[dr]_{\tilde{x}_0 } & 
                                         &  \tv_{\cF}   \ar[dl]^{\tilde{z}_0}  \\
                            &    \CC & }  .
                        \end{equation*}  
                         In turn, \textbf{we obtain the following commutative triangle in the 
                        logarithmic category}
                           \begin{equation} \label{eq:commutriangloga} 
                                \xymatrix{
                                      (\tv_{\cF_{a, \roottree}^-}, \cO^{\star}_{\tv_{\cF_{a, \roottree}^-}} ( - Z(\tilde{x}_0)))
                                       \ar@{->}[rr]^{\Phi_a^{\dagger}}
                                        \ar[dr]_{\tilde{x}_0^{\dagger} } & 
                                       & ( \tv_{\cF}, \cO^{\star}_{\tv_{\cF}}( - Z(\tilde{z_0})) )   
                                            \ar[dl]^{\tilde{z}_0^{\dagger}}  \\
                                     &   ( \CC,  \cO^{\star}_{\CC}( - \{ 0 \}) )& } 
                             \end{equation} 
                        \textbf{in which the log enhancement 
                        $\boxed{\Phi_a^{\dagger}}$ of $\Phi_a$ associated to 
                        the divisors $Z(\tilde{x}_0 )$ and $Z(\tilde{z}_0)$  is strict in the sense 
                        of~\autoref{def:strict}.}

                     \medskip
              \item   \label{item:isomtoriclog}
                      \underline{\em We compare the objects associated to the starting
                     singularity $X$ and to the $a$-side  singularity  $X_a$, by} \\  
                     \underline{\em  establishing an isomorphism of log   morphisms.}

                      Denote by $\boxed{\partial_0^-\tilde{Y}_a}$ the subdivisor of $\partial_0 \tilde{Y}_a$ 
                      obtained by removing the irreducible toric divisor corresponding to the ray  
                      $l_{a, \roottree}$ from Step (\ref{item:subfanemb}). By construction, 
                      $\partial_0^-\tilde{Y}_a$ equals the sum of  all 
                      components  of $\partial_0\tilde{Y}_a$ which are contained in the open set 
                      $\tv_{\cF_{a, \roottree}^-}$ of $\tv_{\cF_{a, \roottree}}$.
                       We prove that \textbf{the embedding $\Phi_a$ identifies 
                        $\partial_0^- \tilde{Y}_a$ with $\partial_a \tilde{Y}$.}  

                       By restricting the commutative triangle 
                        \eqref{eq:commutriangloga} to those compact subspaces  of 
                        the source and target of the embedding $\Phi_a$, we get the following 
                        commutative triangle  in the logarithmic category:
                       \begin{equation} \label{eq:triangisom} 
                                \xymatrix{
                (\partial_0^- \tilde{Y}_a, \cO^{\star}_{\tv_{\cF_{a, \roottree}^-}| \partial_0^- \tilde{Y}_a}
                      ( - Z(\tilde{x}_0 )) ) 
                                       \ar@{->}[rr]
                                  \ar[dr] & 
              & (\partial_a \tilde{Y}, \cO^{\star}_{\tv_{\cF}| \partial_a \tilde{Y}}( - Z(\tilde{z_0})) )  \ar[dl]  \\
                            &    0^{\dagger}. & } 
                        \end{equation}   
             
                   As $\Phi_a^{\dagger}$ is strict by Step (\ref{item:commuttrianga}), 
                   the horizontal arrow from~\eqref{eq:triangisom} is an isomorphism. 
                     Therefore, \textbf{this diagram allows us to factor the first log morphism in     
                     \eqref{eq:decompBrYtoric} through the log morphism from
                   \eqref{eq:decompBrYtorica}.}

                     \medskip
              \item  \label{item:cutfibinside}
                  \underline{\em We continue comparing objects associated to $X$ and $X_a$, 
                      by looking  at the rounding of the previous} \\  
                      \underline{\em  commutative diagram of log morphisms.}

                   Consider the rounding   of the  diagram \eqref{eq:triangisom}:
                      \begin{equation} \label{eq:triangisomBr} 
                                \xymatrix{
           (\partial_0^- \tilde{Y}_a, \cO^{\star}_{\tv_{\cF_{a, \roottree}}| \partial_0^- \tilde{Y}_a}
                 ( - Z(\tilde{x}_0 )) )_{\log} 
                                       \ar@{->}[rr]
                                  \ar[dr] &   &
               (\partial_a \tilde{Y}, \cO^{\star}_{\tv_{\cF}| \partial_a \tilde{Y}}
                     ( - Z(\tilde{z_0})) )_{\log}  \ar[dl]  \\
                            &    0^{\dagger}_{\log}.  }  
                        \end{equation}       
                   By Step  (\ref{item:combprevsteps}), 
                   the fibers of the rightmost arrow of \eqref{eq:triangisomBr}   
                   are homeomorphic to the Milnor fibers of $X$. 
                  By  Step (\ref{item:combprevstepsa}), the fibers of the topological morphism 
                        \begin{equation}\label{eq:topMorph}    
                           (\partial_0 \tilde{Y}_a, \cO^{\star}_{\tv_{\cF_{a, \roottree}}| 
                         \partial_0 \tilde{Y}_a}( - Z(\tilde{x}_0 )) )_{\log} 
                          \to  0^{\dagger}_{\log}
                        \end{equation}
                        are homeomorphic to the Milnor fibers of $X_a$. 

                    Note that the previous topological morphism is not the leftmost arrow 
                    of~\eqref{eq:triangisomBr}. In fact, by its definition in Step (\ref{item:isomtoriclog}), 
                    $\partial_0^- \tilde{Y}_a$ is a subdivisor of $\partial_0 \tilde{Y}_a$.  Thus, the 
                    map~\eqref{eq:topMorph} is the composition of  the leftmost map on~\eqref{eq:triangisomBr} 
                    with the canonical strict log morphism
                        \[   (\partial_0^- \tilde{Y}_a, \cO^{\star}_{\tv_{\cF_{a, \roottree}}| \partial_0^- \tilde{Y}_a}
                 ( - Z(\tilde{x}_0 )) )
                              \to
                        (\partial_0 \tilde{Y}_a, \cO^{\star}_{\tv_{\cF_{a, \roottree}}| 
                         \partial_0 \tilde{Y}_a}( - Z(\tilde{x}_0 )) )    
                                  \] 
                       obtained by restriction.

                   \medskip
            \item     \label{item:cutmiln}
                \underline{\em We give a rounding presentation of the $a$-side cut Milnor fibers.}
            
                   We prove that \textbf{the fibers of the leftmost arrow from~\eqref{eq:triangisomBr} 
                   are homeomorphic to the Milnor fibers of $X_a$ cut by the variable corresponding 
                   to the root of $\Gamma_a$.} We have an analogous fact concerning the $b$-side.

                 \medskip
             \item      \label{item:prodstruct}
                 \underline{\em We prove the product structure of the central pieces.}
             
                We prove that \textbf{the fibers of 
                 $ (\tilde{f}_0)^{\dagger}_{\log}  \colon Z(\tilde{f})^{\dagger}_{\log} \to 0^{\dagger}_{\log} $ 
                 contained in $\tau_{Z(\tilde{f})^{\dagger}}^{-1}(\partial_{\roottree} \tilde{Y})$ have the desired 
                 product structure}. This comes from the fact that the divisor 
                 $\partial_{\roottree} \tilde{Y}$ is the 
                 cartesian product of two smooth projective curves which are one-point compactifications 
                 of affine curves diffeomorphic to the Milnor fibers of $g_a\colon Y_a \to \CC$ and 
                 $g_b \colon Y_b \to \CC$. 
                 
               \medskip
            \item    \label{item:finstep}
                 \underline{\em We prove that the gluing agrees with the prediction 
                of~\autoref{conj:MFC} done by Neumann and Wahl}.  
             
                Combining the results of the last two steps (\ref{item:cutmiln}) and 
                  (\ref{item:prodstruct}), we get a decomposition of the Milnor fibers of $f$ into 
                  three pieces which have the expected structure described  in the Milnor fiber conjecture.   
                  Moreover, we prove that  they are glued together as predicted by Neumann and Wahl. 
                  This establishes the conjecture.
    \end{enumerate}

    \section*{Acknowledgments}
Maria Angelica Cueto was supported by an NSF
postdoctoral fellowship DMS-1103857 and NSF Standard Grants DMS-1700194 and DMS-1954163 (USA). 
Patrick Popescu-Pampu was supported by French grants ANR-12-JS01-0002-01 SUSI, 
ANR-17-CE40-0023-02 LISA and Labex CEMPI (ANR-11-LABX-0007-01). 
Labex CEMPI also financed a one month research stay of the first author in Lille 
in 2018. Part of this project was carried out during two Research in triples programs, 
one at the Centre International de Rencontres Math\'ematiques 
(CIRM, Marseille, France, 2014, Award number: 1173/2014)
 and one at the Center International Bernoulli (EPFL, Lausanne, Switzerland, 2017).  
 We would like to thank both institutes for their hospitality, 
 and for providing excellent working conditions. We are also grateful 
 to Norbert A'Campo, Luc Illusie, Eric Katz, Johannes Nicaise, 
 Arthur Ogus, Bernard Teissier and Jonathan Wahl for answering our questions. 
 We are particularly grateful 
 to Dan Abramovich for directing us to Nakayama and Ogus' paper \cite{NO 10} 
 as the best reference in the current literature for understanding how log structures 
 allow to cut a complex variety along a divisor, a fact the second author 
 had learnt about a few years before, in a talk by Kawamata. The second author 
 is also grateful to the organizers 
 of the \emph{Workshop on Singularities: topology, valuations and semigroups} (held in 2019 at the Complutense University of Madrid, Spain), of the summer school \emph{Milnor fibrations, 
 degenerations and deformations from modern perspective} (held in 2021 at CIRM, France) 
 and of the \emph{Annual meeting of the GDR singularity} (held in 2022 at Aussois, France)  
 for giving him the opportunity to present parts of this work during research mini-courses. 
 Finally, we warmly thank Anne Pichon and the referee for their pertinent remarks 
 on a previous version of this paper.


\medskip
\vspace{3ex}

 
\noindent
\textbf{\small{Authors' addresses:}}
\smallskip
\

\noindent
\small{M.A.\ Cueto,  Mathematics Department, The Ohio State University, 231 W 18th Ave, Columbus, OH 43210, USA.
\\
\noindent \emph{Email address:} \url{cueto.5@osu.edu}}
\vspace{2ex}

\noindent
\small{P.\ Popescu-Pampu,
  Univ.~Lille, CNRS, UMR 8524 - Laboratoire Paul Painlev{\'e}, F-59000 Lille, France.
  \\
\noindent \emph{Email address:} \url{patrick.popescu-pampu@univ-lille.fr}}
\vspace{2ex}

\noindent
\small{D.\ Stepanov, Laboratory of Algebraic Geometry and Homological Algebra, Department of Higher Mathematics, Moscow Institute of Physics and Technology,
9 Institutskiy per., Dolgoprudny, Moscow, 141701, Russia.
  \\
\noindent \emph{Email address:} \url{stepanov.da@phystech.edu}}


\begin{thebibliography}{8}


   \bibitem{A 75}  N. A'Campo, \href{https://eudml.org/doc/139622}{\emph{La fonction z\^eta d'une monodromie}}.  Comment. Math. Helv. \textbf{50} (1975), 233--248. 
           

    \bibitem{ACGHOSS 13}  D. Abramovich, Q. Chen,  D. Gillam,  Y. Huang, M. Olsson, 
      M. Satriano, S.  Sun, \href{http://www.math.boun.edu.tr/instructors/wdgillam/abramovichetmany.pdf}{\emph{Logarithmic geometry and moduli}}.
      In \emph{Handbook of moduli.} Vol. I, 
      1--61, Adv. Lect. Math. (ALM), \textbf{24}, Int. Press, Somerville, MA, 2013. 
      
  \bibitem{AO 20} P. Achinger, A. Ogus, \href{http://front.math.ucdavis.edu/1802.02234}{\emph{Monodromy and log geometry}}. Tunis. J. Math.  \textbf{2} (2020), no. 3, 455--534.
  
  
  \bibitem{A 21} H. Arg\"uz, \href{https://arxiv.org/abs/1610.07195}{\emph{Real loci in (log) Calabi-Yau manifolds via Kato-Nakayama spaces of toric degenerations}}. European Journal of Mathematics \textbf{7}, 869--930 (2021).

    
  \bibitem{AGS 13}   F. Aroca,  M. G\'omez-Morales, K. Shabbir, 
      \href{https://arxiv.org/abs/1209.5104}{\emph{Torical modification of Newton non-degenerate ideals}}.   
       Rev. R. Acad. Cienc. Exactas F\'{\i}s. Nat. Ser. A Math. RACSAM \textbf{107} (2013), 
       no. 1, 221--239. 
       
 \bibitem{A 10} S. Awodey, \emph{Category theory.} Second Edition, Oxford Univ. Press, 2010.
 
 \bibitem{B 66} E. Brieskorn,  \href{https://www.maths.ed.ac.uk/~v1ranick/papers/brieskorn.pdf}{\em Beispiele zur Differentialtopologie von Singularit\"aten.} 
     Inventiones Math. {\bf 2} (1966), 1--14. 
                    
  \bibitem{B 00} E. Brieskorn, \textit{Singularities in the work of Friedrich Hirzebruch.} 
      Surveys in Differential Geometry  \textbf{VII} (2000), 17--60. 
             
   \bibitem{BN 20} E. Bultot, J. Nicaise, \href{https://link.springer.com/article/10.1007/s00209-019-02342-5}{\emph{Computing motivic zeta functions on log smooth models}}. 
   Math. Zeitschrift \textbf{295} (2020), 427--462. 
   
   \bibitem{CFP 21} J.-B. Campesato, G. Fichou, A. Parusi\'nski, \href{https://arxiv.org/abs/2111.14881v1}{\emph{Motivic, logarithmic, and topological Milnor fibrations}} (2021), \texttt{arXiv:2111.14881v1}. 
             
   \bibitem{C 16} T. Cauwbergs, \href{https://www.sciencedirect.com/science/article/pii/S1631073X16300498}{\emph{Logarithmic geometry and the Milnor fibration}}. 
       C. R. Acad. Sci. Paris, Ser. I \textbf{354} (2016), 701--706. 
       
         
  \bibitem{CPS 21} M. A. Cueto, P. Popescu-Pampu, D. Stepanov, 
 \href{https://arxiv.org/abs/2108.05912}{\emph{Local tropicalizations of splice type surface singularities}} (2021), \texttt{arXiv:2108.05912}.
        
   \bibitem{D 07} M. Dehn, \emph{Berichtigender Zusatz zu III AB 3 Analysis Situs}. 
      Jahresber. Deutsch. Math.-Verein. \textbf{16} (1907), 573.
      
   \bibitem{D 92} A. Dimca, \emph{Singularities and topology of hypersurfaces.}  Springer, 1992. 
      

   \bibitem{EN 85} D. Eisenbud, W. Neumann, \emph{Three-dimensional link theory and 
      invariants of plane curve singularities.} Princeton Univ. Press, 1985. 
      
      
  \bibitem{FP 22} J. Fern\'andez de Bobadilla, T. Pelka, \href{https://arxiv.org/abs/2204.07007}{\em Symplectic monodromy at radius zero and equimultiplicity of $\mu$-constant families.} (2022),  \texttt{arXiv:2204.07007.} 
     
    
    \bibitem{Gi 11} W. Gillam, \href{http://www.math.boun.edu.tr/instructors/wdgillam/orb.pdf}{\emph{Oriented real blowup}}. Manuscript of 2011 available on its author's webpage. 
    
   \bibitem{GM 15} W. Gillam, S. Molcho, \href{https://arxiv.org/abs/1507.06752}{\emph{Log differentiable spaces and manifolds with corners.}} (2015), \texttt{arXiv:1507.06752.}
    
   
        
    \bibitem{G 99} C. McA. Gordon,  \href{http://homepages.warwick.ac.uk/~masgar/Teach/2021_3MFDS/References/1999three_dimensional_topology_up_to_1960.pdf}{\emph{3-dimensional topology up to 1960}}. In   \emph{History of topology}, 449--489, North-Holland, Amsterdam, 1999.
    
  \bibitem{G 62} H. Grauert,  \href{https://eudml.org/doc/160940}{\em {\"U}ber Modifikationen und
    exzeptionnelle analytische Mengen.} Math. Ann. \textbf{146}
    (1962), 331-368.
    
  \bibitem{GLS 07} G.-M. Greuel, C. Lossen, E. Shustin, 
     \emph{Introduction to singularities and deformations.}  Springer, Berlin, 2007. 
        
   \bibitem{G 11} M. Gross, \emph{Tropical geometry and mirror symmetry.}   
        Amer. Math. Soc., 2011. 
        
        \bibitem{H 69} H. Hamm, \emph{Die Topologie isolierter Singularit{\"a}ten von vollst{\"a}ndigen Durchschnitten komplexer Hyperfl{\"a}chen}. Dissertation Bonn 1969.

        
   \bibitem{H 71} H. Hamm, \href{https://eudml.org/doc/162134}{\emph{Lokale Topologische Eigenschaften Komplexer R\"aume}}. 
      Math. Ann. \textbf{191} (1971), 235--252.
      
    \bibitem{H 72} H. Hamm \href{https://eudml.org/doc/162267}{\emph{Exotische Sph\"aren als Umgebungsr\"ander in speziellen 
        komplexen R\"aumen}}. Math. Ann. \textbf{197} (1972), 44--56. 
            
  \bibitem{HPV 00} J. Hubbard, P. Papadopol, V. Veselov, \href{https://projecteuclid.org/euclid.acta/1485891316}{\emph{A compactification of H\'enon mappings in $\CC^2$ as dynamical systems}}.  
  Acta Mathematica \textbf{184}, No. 2 (2000), 203--270.
    
 \bibitem{I 94} L.  Illusie, \emph{Logarithmic spaces (according to K. Kato)}. In 
      \emph{Barsotti Symposium in Algebraic Geometry} (Ed. V. Cristante, W. Messing), 
       Perspectives in Math. \textbf{15}, Academic Press, 1994, pp. 183--203.
       
 \bibitem{IKN 05} L. Illusie, K. Kato, C. Nakayama, \href{http://www.ms.u-tokyo.ac.jp/journal/abstract/jms120101.html}{\emph{Quasi-unipotent Logarithmic 
     Riemann-Hilbert Correspondences}}. J. Math. Sci. Univ. Tokyo \textbf{12} (2005), 1--66.
     
  \bibitem{JN 83} M. Jankins, W. Neumann, \href{https://www.ams.org/open-math-notes/omn-view-listing?listingId=110712}{\emph{Lectures on Seifert manifolds.}} Brandeis 
      Lecture Notes \textbf{2}, 1983.  
     
  \bibitem{J 16} D. Joyce, \href{http://front.math.ucdavis.edu/1501.00401}{\emph{A generalization of manifolds with corners}}. 
      Advances in Mathematics \textbf{299} (2016) 760--862. 
  
  \bibitem{KN 08} T. Kajiwara, C. Nakayama, \href{https://www.ms.u-tokyo.ac.jp/journal/pdf/jms150205.pdf}{\emph{Higher direct images of local systems in log Betti cohomology}}. J. Math. Sci. Univ. Tokyo 15 (2008) 291--323.
 
  
  \bibitem{K 88} K. Kato, \href{http://www.math.brown.edu/dabramov/LOGGEOM/Kato-log.pdf}{\emph{Logarithmic structures of Fontaine-Illusie}}. In   
     \emph{Algebraic Analysis, Geometry and Number Theory}, the Johns Hopkins Univ. Press, 
     Baltimore, 1988, 191--224.
  
  \bibitem{KN 99} K. Kato, C. Nakayama, \href{https://projecteuclid.org/euclid.kmj/1138044041}{\emph{Log Betti cohomology, log \'etale cohomology, 
    and log de Rham cohomology of log schemes over $\CC$}}. Kodai Math. J. \textbf{22} (1999), 161--186. 
    
       
\bibitem{K 02} Y. Kawamata, \href{https://arxiv.org/abs/math/0107160}{\emph{On algebraic fiber spaces}}.  
         In \emph{Contemporary trends in 
     algebraic geometry and algebraic topology} (Tianjin, 2000), vol. 5 of Nankai Tracts Math., 
     World Sci. Publ., River Edge, NJ, 2002,  135--154.
  
       
  \bibitem{KN 94} Y. Kawamata, Y. Namikawa, \href{https://eudml.org/doc/144241}{\emph{Logarithmic deformation of normal crossing varieties and smoothing of degenerate Calabi-Yau varieties}}.  
      Inventiones Math. \textbf{118} (1994), 395--409. 
    
   \bibitem{KKMS 73} G. Kempf, F. F. Knudsen, D. Mumford, B. Saint-Donat, 
      \emph{Toroidal embeddings. I.} Lecture Notes in Maths. \textbf{339}. 
      Springer-Verlag, Berlin-New York, 1973.
      
   \bibitem{K 77} A. G. Khovanskii, \href{https://www.math.toronto.edu/askold/1977-Faa-4-english.pdf}{\emph{Newton polyhedra and toroidal varieties}}.  
     Funkt. An. i Prilozhen. \textbf{11} (1977), 56--64; English translation, Funct.  
     An. Appl. \textbf{11}:4 (1977), 289--296.
    
    \bibitem{KS 79} R. C. Kirby, M. G. Scharlemann, \href{http://www.maths.ed.ac.uk/~aar/papers/kirbysch.pdf}{\emph{Eight faces of the Poincar'e homology 3-sphere}}. In {\em Geometric topology} 
     (Proc. Georgia Topology Conf., Athens, Ga., 1977), 113--146, Academic Press, 
     New York-London, 1979.
     
     \bibitem{KM 15} 
        C. Kottke, R. B. Melrose, \href{https://arxiv.org/abs/1107.3320}{\emph{Generalized blow-up of corners and fibre products}}. Trans. Amer. Math. Soc. \textbf{367} (2015), 651--705.
     
     
    \bibitem{K 76} A. G. Kouchnirenko, \href{https://gdz.sub.uni-goettingen.de/id/PPN356556735_0032}{\emph{Poly\`edres de Newton et nombres de Milnor}}. 
         Inventiones Math. \textbf{32} (1976), 1--31. 
     
   \bibitem{L 09} P. J. Lamberson, \href{https://www.ams.org/journals/tran/2009-361-09/S0002-9947-09-04647-9/S0002-9947-09-04647-9.pdf}{\emph{The Milnor fiber conjecture and iterated branched cyclic covers.}} Trans. Amer. Math. Soc. \textbf{361} (2009), 4653--4681. 
   
   \bibitem{L 71}  H. Laufer, \emph{Normal two-dimensional singularities.}  Princeton Univ. Press, 1971. 
     
          
    \bibitem{LNS 20} D. T. L\^e, J. J. Nu\~{n}o-Ballesteros, J. Seade,     \emph{The topology of the 
      Milnor fibration.}   Chapter 6 of {\em Handbook of geometry and topology of singularities I}. 
      Springer, 2020, 321--388. 
      
  \bibitem{L 84} E. N. Looijenga,  \emph{Isolated singular points on
    complete intersections.} London Math. Soc. Lecture Notes Series
    \textbf{77}, Cambridge Univ. Press, Cambridge, 1984.
          
   \bibitem{LQ 11} M. Luxton,  Z. Qu, \href{https://www.ams.org/journals/tran/2011-363-09/S0002-9947-2011-05254-2/}{\emph{Some results on tropical compactifications}}. Trans. Amer. Math. Soc. \textbf{363} (2011), 4853--4876.
      
   \bibitem{MS 15} D. Maclagan,  B. Sturmfels,  \emph{Introduction to tropical geometry.}  
Graduate Studies in Mathematics, \textbf{161}. Amer. Math. Soc., Providence, RI, 2015. 

  \bibitem{M 98} S. Mac Lane, \emph{Categories for the working mathematician.} Second edition. 
     Springer-Verlag, New York, 1998.

     \bibitem{M 84} H. Majima, \emph{Asymptotic analysis for integrable connections with 
         irregular singular points.}  Lecture Notes in Maths. \textbf{1075}, Springer, 1984. 
         
     \bibitem{M 80} M. Merle, \href{http://www.numdam.org/item/SSS_1976-1977____A18_0/}{\emph{Poly\`edre de Newton, \'eventails et d\'esingularisation, d'apr\`es A. N. Varchenko}}. In 
     \emph{S\'eminaire sur les singularit\'es des surfaces (1976-1977)}, Expos\'e no. 16, 
     Springer Lect. Notes in Maths. \textbf{777} (1980) 289--294. 
         
     \bibitem{M 68}  J. Milnor, \emph{Singular points of complex hypersurfaces}. 
     Annals of Mathematics Studies, No. \textbf{61}. Princeton Univ. Press, Princeton, N.J.; 
      University of Tokyo Press, Tokyo 1968.
      
      \bibitem{M 00} G. M\"uller, \emph{Resolution of weighted homogeneous surface singularities.} 
         Progress in Math. {\bf 181}, Birkh\"auser, 2000, 507--517. 

  \bibitem{M 61} D. Mumford, \href{http://www.numdam.org/item/PMIHES_1961__9__5_0/}{\emph{The topology of normal singularities of an algebraic surface 
       and a criterion for simplicity}}. Inst. Hautes \'Etudes Sci. Publ. Math. no. \textbf{9} (1961), 5--22.
  
   \bibitem{NO 10} C. Nakayama, A. Ogus, \href{https://msp.org/gt/2010/14-4/p09.xhtml}{\emph{Relative rounding in toric and logarithmic geometry}}. 
       Geometry \& Topology \textbf{14} (2010), 2189--2241. 
       
    \bibitem{NO 09} A. N\'emethi, T. Okuma, \href{https://arxiv.org/abs/math/0610465}{\emph{On the Casson invariant conjecture of  Neumann-Wahl}}. J. Algebraic Geom. \textbf{18} (2009), no. 1, 135--149.
         
         
    \bibitem{N 77} W. Neumann, \href{https://www.researchgate.net/profile/Walter_Neumann/publication/226996861_Brieskorn_complete_intersections_and_automorphic_forms/links/56a63d6e08ae2c689d39e1e8.pdf}{\emph{Brieskorn complete intersections and automorphic 
       forms}}. Inventiones Math. \textbf{42} (1977), 285--293.      
         
   \bibitem{N 81}  W. Neumann, \href{https://www.ams.org/journals/tran/1981-268-02/S0002-9947-1981-0632532-8/}{\emph{A calculus for plumbing applied to the topology of complex surface singularities and degenerating complex curves}}. Trans. Amer. Math. Soc. \textbf{268}, 2 
      (1981), 299--344.
       
    \bibitem{N 83} W. Neumann, \emph{Abelian covers of quasihomogeneous surface singularities.} 
        Singularities, Part 2 (Arcata, Calif., 1981), 233--243, Proc. Sympos. Pure Math. \textbf{40}, 
        Amer. Math. Soc., Providence, RI, 1983.
        
    \bibitem{NR 78} W. Neumann, F. Raymond, \href{https://www.researchgate.net/publication/226994499_Seifert_manifolds_plumbing_m-invariant_and_orientation_reversing_maps}{\emph{Seifert manifolds, plumbing, 
         $\mu$-invariant and orientation reversing maps.}} In \emph{Algebraic and Geometric Topology. 
         Santa Barbara 1977.} K. E. Millett ed. Lect. Notes in Maths. \textbf{664}, 1978, 162--196. 
         
       
    \bibitem{NW 90} W. Neumann, J. Wahl, \href{https://www.maths.ed.ac.uk/~v1ranick/papers/neumann010.pdf}{\emph{Casson invariants of links of singularities.}} 
          Comment. Math. Helv. \textbf{65} (1990), 58--78.
          
    \bibitem{NW 02} W. Neumann, J. Wahl, \href{https://arxiv.org/abs/math/0110167}{\emph{Universal abelian covers of surface singularities.}}
        In \emph{Trends in Singularities}, edited by A. Libgober and M. Tib\u{a}r, Birkh\"auser, 
        Basel (2002), 181--190.
        
       
   \bibitem{NW 05} W. Neumann, J. Wahl, \href{http://emis.ams.org/journals/GT/ftp/main/2005/2005-18.pdf}{\emph{Complex surface singularities with integral 
      homology sphere links.}} Geometry \& Topology \textbf{9} (2005), 757--811. 
      
    \bibitem{NW 05bis} W. Neumann, J. Wahl,  \href{https://arxiv.org/abs/math/0407287}{\emph{Complete intersection singularities of splice type as universal abelian covers}}.  
        Geometry \& Topology \textbf{9} (2005), 699--755. 
        
       
    \bibitem{O 18} A. Ogus, \emph{Lectures on logarithmic algebraic geometry}. 
        Cambridge Univ. Press, 2018. 
        
     \bibitem{O 97} M.  Oka,  \emph{Non-degenerate complete intersection singularity.}  
       Actualit\'es Math\'ematiques. [Current Mathematical Topics] Hermann, Paris, 1997.
       
          
     \bibitem{O 72} P. Orlik,  \emph{Seifert manifolds.} Lect. Notes in Maths. \textbf{291}, Springer, 1972. 
     
     \bibitem{OW 71} P. Orlik, Ph. Wagreich, \emph{Isolated singularities of algebraic surfaces 
         with $\CC^*$-action.}  Annals of Maths. {\bf 93} No. 2 (1971), 205--228. 
         
   \bibitem{P 98} A. Parusi\'nski, \href{https://arxiv.org/abs/math/9701221}{\emph{Blow-Analytic Retraction onto the Central Fibre}}. 
      In \emph{Real analytic and algebraic singularities}, Proceedings of the Kuo Symposium, eds.     
       Fukuda et al. Pitman Research Notes in Mathematics \textbf{381}, Longman 1998, 43--61.
       
       \bibitem{P 77} U. Persson, \emph{On degenerations of algebraic surfaces.}    Memoirs 
          of the Amer. Math. Soc. \textbf{189}, 1977.
          
      \bibitem{P 04} H. Poincar\'e, \href{http://analysis-situs.math.cnrs.fr/-Cinquieme-complement-.html}{\emph{Cinqui\`eme compl\'ement \`a l'Analysis Situs}}.  
           Rendiconti Circolo Mat.  Palermo \textbf{18} (1904), 45--110.
      
       
    \bibitem{PPS 13} P. Popescu-Pampu, D. Stepanov, \href{http://front.math.ucdavis.edu/1204.6154}{\emph{Local tropicalization}}.  
           Algebraic and combinatorial aspects of tropical geometry, 253--316, 
           Contemp. Math. \textbf{589}, Amer. Math. Soc., Providence, RI, 2013. 
           
         
    \bibitem{SG 17}  H. P. de Saint Gervais, \href{http://analysis-situs.math.cnrs.fr/-Variete-dodecaedrique-de-Poincare-.html}{\emph{La vari\'et\'e dod\'eca\'edrique de Poincar\'e}}.  Section of the 
   website \href{http://analysis-situs.math.cnrs.fr}{\em Analysis Situs}, January 2017. 
   
     
   
  \bibitem{S 33} H. Seifert,  \href{https://www.maths.ed.ac.uk/~v1ranick/papers/seifert3.pdf}{\emph{Topologie dreidimensionaler gefaserter R\"aume}}. Acta Math. \textbf{60} (1933), 147--238. 
  An English translation by W. Heil appeared under the title 
  \emph{Topology of $3$-dimensional fibered spaces} in 
  \href{https://www.maths.ed.ac.uk/~v1ranick/papers/seifthreng.pdf}{\emph{Seifert and Threlfall: a textbook of topology}}. Academic Press, 1980. 
       
   \bibitem{S 72} L. C. Siebenmann, \href{http://www.math.psu.edu/petrunin/papers/akp-papers/SIEBENMANN.pdf}{\emph{Deformation of homeomorphisms on stratified sets. I, II}}. 
   Comment. Math. Helv. \textbf{47} (1972),  123--136; ibid. 137--163.     
   
   \bibitem{S 80} L. C. Siebenmann, \emph{On vanishing of the Rohlin invariant and nonfinitely 
      amphicheiral homology 3-spheres.} Topology Symposium, Siegen 1979 (Proc. Sympos., 
      Univ. Siegen, Siegen, 1979), pp. 172--222, Lecture Notes in Math. \textbf{788}, Springer, Berlin, 1980.
      
   \bibitem{S 61} S. Smale, \href{http://www.math.uchicago.edu/~shmuel/tom-readings/Smale, PC.pdf}{\emph{Generalized Poincar\'e's conjecture in dimensions greater 
       than four}}. Annals of Maths. Second Series, Vol. 74, No. 2 (1961), 391--406. 
     
   
   \bibitem{TV 18} M. Talpo, A. Vistoli, \href{https://arxiv.org/abs/1611.04041}{\emph{The Kato-Nakayama space as a transcendental root stack}}. 
        Int. Math. Res. Not. IMRN \textbf{2018}, no. 19, 6145--6176.
   
   \bibitem{T 04} B. Teissier, \href{https://webusers.imj-prg.fr/~bernard.teissier/documents/MSRI_09_02.pdf}{\emph{Monomial ideals, binomial ideals, polynomial ideals}}. In \emph{Trends in commutative algebra}, 211--246, Math. Sci. Res. Inst. Publ., \textbf{51},  Cambridge Univ. Press, Cambridge, 2004.
     
    
   \bibitem{T 07} J. Tevelev, \href{http://people.math.umass.edu/~tevelev/trop80.pdf}{\emph{Compactifications of subvarieties of tori}}. Amer. J. of Maths. \textbf{129} (2007), 1087--1104.
   
   \bibitem{T 70} G. N.  Tyurina,  \emph{Locally semiuniversal flat deformations of isolated 
        singularities of complex spaces}. Math. USSR-Izv., \textbf{3} (1970), no. 5, 967-999.
  
  \bibitem{V 44} P. Du Val, \textit{On absolute and non-absolute
    singularities of algebraic surfaces.} Revue de la Facult{\'e} des
    Sciences de l'Univ. d'Istanbul (A) \textbf{91} (1944), 159-215.
  
  \bibitem{V 76}   A. N. Varchenko, \href{https://eudml.org/doc/142438}{\emph{Zeta-function of monodromy and Newton's diagram}}. 
      Inventiones Math. \textbf{37} (1976), 253--262.
     
   \bibitem{W 06} J. Wahl, \href{https://arxiv.org/abs/math/0509085}{\emph{Topology, geometry, and equations of normal surface singularities}}.  
       In \emph{Singularities and computer algebra}. London Math. Soc. Lect. Note Series \textbf{324}. 
        Cambridge Univ. Press, 2006, 351--371. 
  
  \bibitem{W 22} J. Wahl, \href{https://arxiv.org/abs/2202.00587}{\emph{Splice diagrams and splice-quotient surface singularities.}} (2022), \texttt{arXiv:2202.00587}. 
  
  \bibitem{W 73} H. E. Winkelnkemper, \href{https://www.ams.org/journals/bull/1973-79-01/S0002-9904-1973-13085-X/S0002-9904-1973-13085-X.pdf}{\em Manifolds as open books.} 
         Bull. Amer. Math. Soc. {\bf 79} No. 1 (1973), 45--51. 
           
       
\end{thebibliography}
\end{document}